\documentclass[12pt]{article}
\RequirePackage[OT1]{fontenc}
\RequirePackage{amsthm,amsmath}
\usepackage{mathrsfs}
\usepackage{cite}
\RequirePackage[numbers]{natbib}

\usepackage{amsfonts,amssymb,amsbsy,latexsym,graphicx}
\usepackage{psfrag,epsfig,epsf} 
\usepackage{subfigure}
\usepackage{color}
\usepackage{verbatim}
\usepackage{comment}
\usepackage[margin=2cm]{geometry}
\usepackage{cleveref}

\newcommand{\Expect}[1]{\mathbb{E} \left[{#1}\right]}

\newcommand{\Var}[1]{\mbox{Var} \left[{#1}\right]}

\newcommand{\Corr}[1]{\mbox{Corr} \left[{#1}\right]}
\newcommand{\Cov}[1]{\mbox{Cov} \left[{#1}\right]}
\newcommand{\Prob}[1]{\mathbb{P} \left({#1}\right)}

\newcommand{\Abs}[1]{\left\vert{#1}\right\vert}

\newcommand{\md}{\mbox{d}}

\newcommand{\pih}{\widehat{\pi}}

\newcommand{\bma}{\mathbf{a}}

\newcommand{\bmc}{\mathbf{c}}
\newcommand{\bmC}{\mathbf{C}}

\newcommand{\bmF}{\mathbf{F}}

\newcommand{\bmh}{\mathbf{h}}

\newcommand{\bmI}{\mathbf{I}}

\newcommand{\bmm}{\mathbf{m}}

\newcommand{\bmS}{\mathbf{S}}

\newcommand{\bmu}{\mathbf{u}}
\newcommand{\bmU}{\mathbf{U}}

\newcommand{\bmV}{\mathbf{V}}

\newcommand{\bmx}{\mathbf{x}}
\newcommand{\bmX}{\mathbf{X}}
\newcommand{\bmy}{\mathbf{y}}
\newcommand{\bmY}{\mathbf{Y}}
\newcommand{\bmz}{\mathbf{z}}
\newcommand{\bmZ}{\mathbf{Z}}

\newcommand{\bmzero}{\mathbf{0}}

\newcommand{\bmepsilon}{\mbox{\boldmath$\epsilon$}}

\newcommand{\bmSigma}{\mbox{\boldmath$\Sigma$}}

\newcommand{\PP}{{\sf P\hspace*{-0.9ex} \rule{0.15ex}{1.5ex}\hspace*{0.9ex}}}

\newcommand{\X}{\mathcal{X}}

\newcommand{\U}{\mathcal{U}}

\newcommand{\loc}{\mathrm{loc}}
\newcommand{\scale}{\mathrm{scale}}

\newcommand{\aux}{\gamma}
\newcommand{\Aux}{\Gamma}

\newcommand{\floor}[1]{\lfloor {#1} \rfloor}
\newcommand{\FF}{\mathcal F}
\newcommand{\EE}{\mathbb{E}}
\renewcommand{\PP}{\mathbb{P}}
\newcommand{\OO}{\mathcal{O}}
\newcommand{\Normal}[1]{\textbf{N}\left( #1 \right)}
\newcommand{\RR}{\mathbb{R}}
\newcommand{\LL}{\mathscr{L}}

\renewcommand{\SS}{\mathcal{S}}

\newcommand{\alphahat}{\widehat{\alpha}}
\newcommand{\muhat}{\widehat{\mu}}
\newcommand{\pihat}{\widehat{\pi}}
\newcommand{\lambdahat}{\widehat{\lambda}}
\newcommand{\sigmahat}{\widehat{\sigma}}
\newcommand{\gammahat}{\widehat{\gamma}}

\newcommand{\wtilde}{\widetilde}
\newcommand{\dist}{\overset{\mathcal{D}}{\sim}}

\newcommand{\A}{\mathcal{A}}

\renewcommand{\epsilon}{\varepsilon}
\renewcommand\phi{\varphi}

\newcommand{\eff}{\mathrm{Eff}}

\newcommand{\esjd}{\textrm{ESJD}}

\newcommand{\acc}{\alpha}

\newcommand{\loglik}{\ell}

\newcommand{\accrwm}[2]{\alpha\left({#1};{#2}\right)}
\newcommand{\acca}[2]{\alpha_1\left({#1};{#2}\right)}

\newcommand{\accb}[2]{\alpha_2\left({#1};{#2}\right)}

\newcommand{\accado}[1]{\alpha_1^{(d)}\left({#1}\right)}

\newcommand{\bra}[1]{\langle #1 \rangle}
\newcommand{\BK}[1]{ {\left( #1 \right)} }
\newcommand{\sqBK}[1]{ {\left[ #1 \right]} }
\newcommand{\curBK}[1]{ {\left\{ #1 \right\}} }

\newcommand{\damh}{delayed-acceptance Metropolis-Hastings }
\newcommand{\PsMMH}{pseudo-marginal Metropolis-Hastings }

\numberwithin{equation}{section}
\theoremstyle{plain}
\newtheorem{thm}{Theorem}[section]
\newtheorem{prop}{Proposition}[section]
\newtheorem{lem}{Lemma}[section]

\theoremstyle{remark}

\newtheorem{example}{Example}[section]
\theoremstyle{definition}
\newtheorem{assumptions}{Assumptions}

\title{Efficiency of Delayed-Acceptance \\Random Walk Metropolis Algorithms}
\author{Chris Sherlock$^1$, Alexandre H. Thiery$^2$ and Andrew Golightly$^3$}
\date{$^1$c.sherlock@lancaster.ac.uk, $^2$a.h.thiery@gmail.com and $^3$andrew.golightly@ncl.ac.uk}

\begin{document}
\maketitle

\begin{abstract}
Delayed-acceptance Metropolis-Hastings and 
delayed-acceptance pseudo-marginal Metropolis-Hastings 
algorithms can be applied
when it is computationally expensive to calculate the true posterior
or an unbiased stochastic approximation thereof, but
a computationally cheap deterministic approximation is available. 
An initial accept-reject stage uses the cheap approximation for computing 
the Metropolis-Hastings ratio; proposals which are accepted at 
this stage are subjected to a further accept-reject step which corrects for the error
in the approximation. Since the expensive posterior, or the approximation
thereof, is only evaluated
for proposals which are accepted at the first stage, the cost of the
algorithm is reduced and larger scalings may be used.  

We focus on the random walk
Metropolis (RWM) and consider the delayed-acceptance RWM and the delayed-acceptance pseudo-marginal RWM. 
We provide a framework for incorporating relatively general 
deterministic approximations into
the theoretical analysis of high-dimensional targets. Justified by 
diffusion-approximation arguments, we derive expressions for the 
limiting efficiency and acceptance rates in high dimensional settings. 
Finally, these theoretical insights are leveraged to formulate practical guidelines 
for the efficient tuning of the algorithms. 
The robustness of these guidelines and predicted properties are verified 
against simulation studies, all of which are strictly outside of the 
domain of validity of our limit results. 
\end{abstract}

\emph{Keywords}:
Markov Chain Monte Carlo, Delayed-Acceptance, Pseudo-Marginal MCMC, Diffusion limit.

\section{Introduction}
\label{sec.introduction}

The Metropolis-Hastings algorithm is widely used
to approximately compute expectations with respect to
complicated high-dimensional posterior distributions
\cite{Gilks/Richardson/Spiegelhalter:1996, MCMCHandbook}. The algorithm requires
that it be possible to evaluate point-wise the posterior density $\pi$ up to a fixed but arbitrary constant of
proportionality. In many cases each such evaluation can be computationally expensive, prompting the use of a surrogate model to accelerate the computations.

The delayed-acceptance Metropolis-Hastings algorithm
\cite{ChristenFox:2005,Moulton/etal:2008,MCMCHandbook16,cui2011bayesian,Banterle2015,SGH2017,SherlockLee2017}, also called the modified Metropolis algorithm \cite{AuBeck2001,CatanachBeck2018}, preconditioned MCMC \cite{EffendievHouLuo2006} and two-stage MCMC \cite{efendiev2005efficient}, and a special case of the surrogate transition method \cite{liu2001monte},  assumes that the exact
posterior $\pi$ is available up to a constant of
integration, but is computationally expensive to evaluate. This framework is particularly relevant to the Bayesian approach to inverse problems \cite{kaipio2006statistical,stuart2010inverse} where point estimations of the posterior density typically involve numerically solving sets of partial differential equations. 
A fast approximation is therefore employed as a
first ``screening'' stage, with proposals which are rejected at the
screening stage simply discarded. The correct posterior,
$\pi$, is only  evaluated for proposals which pass the screening stage. A second  accept-reject step, which corrects for the error in the fast approximation, is then applied so that the desired true posterior is obtained as the limiting distribution of the Markov chain. The \damh  algorithm thus provides a principled method to leverage deterministic approximations to the posterior distribution in inverse problem modeling. In the sequel, we give several examples where a tenfold gain in efficiency is easily obtained by a well-tuned delayed-acceptance strategy.
 
The pseudo-marginal Metropolis-Hastings algorithm \cite{beaumont03,AndrieuRoberts:2009} allows
Bayesian inference when only an unbiased stochastic estimate of the target density, possibly up to an unknown normalisation constant, is available. The particle marginal Metropolis-Hastings algorithm \cite{AndrieuDoucetHolenstein:2010}, a special instance of the \PsMMH  algorithm when the unbiased estimates are obtained by using a particle filter, is a popular method for estimating parameters in hidden Markov models \cite{GolightlyWilkinson:2011,Knape/deValpine:2012}. The existing literature on tuning the \PsMMH is reviewed in Section \ref{sect.lit.review}.

The computational expense involved in creating each unbiased
stochastic estimate suggests that an initial accept-reject stage using a computationally cheap, deterministic, approximation to the posterior might be beneficial. This motivates the delayed-acceptance pseudo-marginal Metropolis-Hastings algorithm  \cite{Smith:2011,GolightlyHendersonSherlock:2013,SGH2017,Quiroz2018,EvRow2017, ViholaFranks2016}. Although the theoretical understanding of delayed-acceptance methods is still limited, several results are available. \cite{SherlockLee2017} compares the ergodicity properties of a delayed-acceptance algorithm with those of the parent MH algorithm, while \cite{FranksVihola2017} compares the asymptotic variance of the ergodic average from a delayed-acceptance algorithm with the variance of an importance-sampling estimator which takes as its proposal a sample from an MCMC algorithm targeting a surrogate.  Historically, insights into the performance and tuning of MCMC algorithms have been obtained by examining the limiting behaviour of a rescaled version of the Markov chain as the dimension of the statespace increases to infinity \cite{Roberts/Gelman/Gilks:1997,RobertsRosenthal:1998,Roberts/Rosenthal:2001, BedardA:2007,BedardRosenthal:2008,SherlockThieryRobertsRosenthal:2013, zanella2017dirichlet, yang2019optimal}.  In this article, we focus on random-walk proposals since this class of methods has the advantage of not requiring further information about the target, such as the local gradient or Hessian.
Thus, we concentrate on the delayed-acceptance random walk Metropolis (DARWM) and the delayed-acceptance pseudo-marginal random walk Metropolis (DAPsMRWM) algorithms: we obtain tuning and efficiency insights into these important algorithms through a 
diffusion-approximation.

%
%
\subsection{Contributions} 
\label{sec.contrib}
When an accurate approximate posterior distribution is available, the use of well-tuned DARWM and DAPsMRWM algorithms can lead to large computational savings. Unfortunately, the tuning of these methods is delicate: it involves choosing an appropriate scale for the random walk proposals and, for the DAPsMRWM, a computational budget allocated to the creation of unbiased estimates of the posterior distribution. Tuning these parameters by estimating the Effective Sample Size (ESS) is typically impractical since the ESS is notoriously difficult and computationally expensive to estimate. These tuning difficulties have hindered the adoption of these powerful methods.

We examine the efficiency of the DARWM and DAPsMRWM algorithms when exploring high-dimensional posterior distributions. We express the efficiency of the methods as a function of the scaling of the random walk proposals and, for the DAPsMRWM, of the quality and computational cost of the unbiased estimates 
of the posterior distribution. One of our main innovations is to circumvent the difficulty of characterising the infinite variety of problem-specific errors in the cheap approximations to the posterior distribution by assuming that the error is a realisation of a random function -- importantly, we empirically demonstrate that, in high-dimensional settings, this framework leads to robust conclusions that can be leveraged to develop efficient tuning guidelines. Under assumptions, we obtain MCMC diffusion limits through homogenization arguments. 
For the DAPsMRWM algorithm, we focus on a specific standard asymptotic regime which occurs 
for instance when the unbiased stochastic estimates are obtained
through a particle filter or when using a product of importance samplers for panel data.

We imagine that a practitioner has tuned a (pseudo-marginal) RWM algorithm,
found it too inefficient, and implemented a delayed-acceptance (pseudo-marginal) RWM
algorithm. Our analysis shows that the relative efficiency of the optimally tuned
delayed-acceptance algorithm when compared to the optimally tuned parent algorithm,
as well as the relative changes in the optimal random-walk scaling and computational budget allocated 
to the creation of unbiased estimates, can be characterised by two parameters: {\bf (1)} the relative computational cost of the cheap approximation compared to the cost of the posterior distribution 
{\bf (2)} a measure of the accuracy of the cheap approximation involving the acceptance rate for proposals that have passed the first, screening stage.
Crucially, these parameters can be estimated easily from a single additional short MCMC simulation. In practical terms, this means that once the parent algorithm (i.e. RWM or pseudo-marginal RWM) is approximately tuned, a single additional MCMC simulation is sufficient to tune the associated delayed acceptance algorithm.

Simulation studies verify different aspects of the theory, the theoretical predictions, a pivotal result (Lemma \ref{lem.pivotal}) on the relationship between changes in the posterior and changes in the deterministic approximation, and the tuning advice.

%
%
%

\section{Delayed-acceptance Random Walks}
\label{sect.intro.the.algs}

Consider a posterior distribution $\pi(\md \bmx)$ on a state-space $\X \subseteq \mathbb{R}^d$. We assume throughout this text that $\pi$ possesses a density $\pi(\bmx)$ with respect to the Lebesgue measure. The Random-Walk Metropolis (RWM) updating scheme provides a general class of algorithms for obtaining approximate samples from the distribution $\pi$ by constructing a Markov chain that is reversible with respect to $\pi$. Given the current value $\bmx \in \X$ of the Markov chain, a perturbation $\bmx^*$ distributed as 
\begin{align} \label{eq.rwm}
\bmx^* \; = \; \bmx + \lambda \, \xi
\end{align}
is generated, for a standard Gaussian random variable $\xi \sim \Normal{0, \mathbf{I}_d}$ and a scale parameter $\lambda > 0$. The proposal $\bmx^*$ is accepted with probability
$\accrwm{\bmx}{\bmx^*} \; = \; 1 \, \wedge \, [\pi(\bmx^*)/\pi(\bmx)]$. Upon acceptance, the proposal $\bmx^* \in \X$ becomes the next current value. Otherwise the current value is left unchanged. For a given scale parameter $\lambda > 0$, and $X\sim \pi$, the acceptance rate of the RWM algorithm is defined as
\begin{align}
\alpha_{\textrm{rwm}}(\lambda) = \EE[\alpha(X, X + \lambda \, \xi)].
\end{align}
This setting is more general than it  might appear since for any matrix $V=AA^\top$ with square $A$, exploration of the posterior of $X'$ using a proposal variance matrix of $\lambda^2V$ is equivalent to exploring $X=A^{-1} X'$ using \eqref{eq.rwm}.

\subsection{Delayed-Acceptance strategies}
\label{sec.algo.description}
As described in the introduction,
there are many situations where $\pi$ is computationally expensive to calculate while a computationally cheap approximation $\pi_a(\bmx)$ to the density $\pi(\bmx)$ is available and can be leveraged within MCMC schemes using the delayed-acceptance algorithm.
At the $k$-th iteration and given the current value $\bmx_k \in \X$ of the parameter, the DARWM first generates a proposal $\bmx^* \in \X$ distributed as \eqref{eq.rwm} and proceeds as follows.
%
%
\begin{enumerate}
\item {\bf Stage-One}: compute the approximation $\pi_a(\bmx^*)$ and the screening acceptance probability
  $\alpha_1(\bmx_k; \bmx^*)  =  1  \wedge \frac{\pi_a(\bmx^*)}{\pi_a(\bmx_k)}$.
%
%
With probability $\acca{\bmx_k}{\bmx^*}$ proceed to Stage-Two. Otherwise set $\bmx_{k+1}=\bmx_k$ and iterate.
\item{\bf Stage-Two}: compute the posterior distribution $\pi(\bmx^*)$ and the second stage probability
$\alpha_2(\bmx_k; \bmx^*)  =  
1  \wedge   \frac{\pi(\bmx^*) \, \pi_a(\bmx_k)}{\pi(\bmx_k) \, \pi_a(\bmx^*)}$.
%
%
With probability $\alpha_2(\bmx_k; \bmx^*)$, set $\bmx_{k+1} = \bmx^*$. Otherwise, set $\bmx_{k+1} = \bmx_k$.
\end{enumerate}
This defines a Markov chain that is reversible with respect to the posterior distribution $\pi$. The Stage-One acceptance rate is defined as
\begin{align}
\label{eqn.alphaonelambda}
  \alpha_1(\lambda) = \EE[\alpha_1(X, X^*)]
\end{align}
where $X \sim \pi$ and $X^* = X + \lambda \, \xi$.
Clearly, the more accurate the approximation $\pi_a$, the higher the Stage-Two acceptance probability. 
The overall acceptance probability is
$\alpha_{12}(\bmx_k;\bmx^*)  =  \alpha_{1}(\bmx_k; \bmx^*)  \times  \alpha_{2}(\bmx_k; \bmx^*)$,
%
whilst the overall acceptance rate is
$\alpha_{12}(\lambda)=\EE[\alpha_{12}(X,X^*)]$.
Our tuning guidelines are based of the conditional Stage-Two acceptance rate $\alpha_{2|1}(\lambda)$ defined as
\begin{align}
  \label{eqn.condtwodef}
\alpha_{2|1}(\lambda) = \frac{\alpha_{12}(\lambda)}{\alpha_1(\lambda)}.
\end{align}

Pseudo-marginal Metropolis-Hastings algorithms 
\cite{beaumont03,AndrieuRoberts:2009} presume 
that it is computationally infeasible to evaluate the posterior density $\pi(\bmx)$, even up to a multiplicative constant, but that it is possible to generate a positive and unbiased estimate of it: $\pih(\bmx;\bmu)$. The quantity $\bmu \in \U$ represents a sample from a source of randomness necessary to produce the stochastic estimate, and $\pih: \X \times \U \to [0,\infty)$ is a deterministic function that, given $\bmx \in \X$ and a random sample $\bmu \in \U$, produces the estimate $\pih(\bmx, \bmu)$. Without loss of generality, one can assume that the auxiliary variable $\bmu \in \U$ is sampled from a fixed and known distribution with density $\rho(\bmu)$. For any $\bmx \in \X$ we have that $\int \pih(\bmx, \bmu) \, \rho(\bmu) \, d \bmu = \pi(\bmx)$. The DAPsMRWM defines a Markov chain on the extended space $\X \times \U$ that can be described as follows. At the $k$-th iteration, given the current value $(\bmx_k, \bmu_k) \in \X \times \U$, the DAPsMRWM generates a Gaussian perturbation $\bmx^* \in \X$ distributed \eqref{eq.rwm}. The Stage-One screening procedure is identical to that of the DARWM. If this screening procedure is successful, a new proposal $\bmu^*$ is generated from $\rho(\bmu)$, leading to an estimate $\pih(\bmx^*; \bmu^*)$ to the  posterior distribution $\pi(\bmx^*)$. The modified Stage-Two acceptance probability reads
\begin{align*}
\alpha_2(\bmx_k, \bmu_k, \bmx^*, \bmu^*) = 1 \, \wedge \,  
\frac{\pih(\bmx^*, \bmu^*) \, \pi_a(\bmx_k)}{\pih(\bmx_k, \bmu_k) \, \pi_a(\bmx^*)}.
\end{align*}
With probability $\alpha_2(\bmx_k, \bmu_k, \bmx^*, \bmu^*)$ one sets $(\bmx_{k+1}, \bmu_{k+1}) = (\bmx^*, \bmu^*)$. Otherwise, one sets $(\bmx_{k+1}, \bmu_{k+1}) = (\bmx_k, \bmu_k)$. Standard arguments show that the DAPsMRWM is reversible with respect to the extended density $\pih(\bmx, \bmu) \, \rho(\bmu)$ on $\X \times \U$. Indeed, this extended density has $\pi(\bmx)$ as marginal density. 
Particle marginal MCMC \cite{AndrieuDoucetHolenstein:2010} is a special case of pseudo-marginal MCMC where the unbiased estimate of the posterior is obtained using a particle filter. It has become one of the key generic methodologies for Bayesian inference of hidden Markov 
models \cite{flury2011bayesian,GolightlyWilkinson:2011,DahlinSchon2019}. The conditional Stage-Two acceptance rate is again defined through \eqref{eqn.condtwodef} with $\alpha_1(\lambda)$ as in \eqref{eqn.alphaonelambda}, but 
$\alpha_{12}(\lambda) = \EE[\alpha_1(X, X^*) \, \alpha_{2}(X,U,X^*, U^*)]$ with
 $(X,U)$ distributed according to the stationary distribution of the DAPsMRWM Markov chain, $X^* = X + \lambda \, \xi$ and $U^* \sim \rho(du)$.

\subsection{Tuning of RWM algorithms}
\label{sect.lit.review}
The efficiency of a given RWM algorithm varies enormously with the scale $\lambda > 0$ of the Gaussian perturbations $\bmx^* = \bmx + \lambda \, \xi$. Small proposed jumps lead to high acceptance rates but little movement across the state-space, whereas large proposed jumps lead to low acceptance rates and again to inefficient exploration of the
state-space. Optimisation of the scale of the proposal has been tackled for various shapes of target
\cite{Roberts/Gelman/Gilks:1997,Roberts/Rosenthal:2001,BedardA:2007,BeskosRobertsStuart:2009,Sherlock/Roberts:2009,Sherlock/Fearnhead/Roberts:2010,Sherlock:2013} and has led to the following rule of thumb: choose the scale so that the acceptance rate is approximately $\alphahat_{\textrm{rwm}}\approx 23\%$. Although nearly all of the theoretical results are based upon limiting arguments in high dimension, the rule of thumb appears to be applicable even in relatively low dimensions \cite{Sherlock/Fearnhead/Roberts:2010}.

In discussing the literature on optimising pseudo-marginal algorithms
it is helpful to define  
a standard asymptotic regime
which
is made precise in Assumptions \ref{ass.noise.diff.indep}--\ref{ass.effort.vs.variance} in Section \ref{sec.stoch.approx}, where its justification and wide applicability is discussed further.

A relatively tractable lower bound on the efficiency of a pseudo-marginal Metropolis-Hastings algorithm is provided
in \cite{Doucetetal:2013}.
Under the standard asymptotic regime, it is shown that the inefficiency of the bounding chain, taking into account the computational cost,
is minimised when the variance of
the noise in the estimated log-posterior is
between $0.92^2$ and $1.68^2$.
In \cite{SherlockThieryRobertsRosenthal:2013}  
the pseudo-marginal random walk Metropolis algorithm is examine under various
regimes for the noise in the estimate of the posterior. Mixing efficiency is considered in
terms of both limiting expected squared jump distance and the speed of a
limiting diffusion, and an overall
efficiency (ESJD/time) is defined, which takes into
account the total computational time.  Under the standard asymptotic regime, joint optimisation of this
efficiency with respect to the variance of the noise in the log-target
and the RWM scale parameter is considered. It is shown that the optimal
scaling occurs when the acceptance rate is approximately $\alphahat_{\textrm{pm}}\approx 7.0\%$ and
the variance of the noise in the estimate of the log-posterior is
approximately $\sigmahat_{\textrm{pm}}^2 \approx 1.82^2$. It is also noted in \cite{SherlockThieryRobertsRosenthal:2013} that for the two different noise distributions considered in
the article, the optimal scaling appears to be insensitive to the
noise variance, and even to the distribution. This phenomenon is shown to hold across a large class of
noise distributions in \cite{Sherlock:2015}.

This article extends \cite{SherlockThieryRobertsRosenthal:2013} to
the corresponding delayed-acceptance algorithm, of which the DARWM is a special case. Results on limiting acceptance rates and mixing efficiency are proved, as is a diffusion
limit. For the DARWM, and for the DAPsMRWM under the standard asymptotic regime, efficiency is then considered in detail, leading to the robust, practical tuning advice that we describe and demonstrate next.

\subsection{Tuning the DARWM}
\label{sec.DARWMtune}
In the interest of brevity, we focus here on tuning guidelines for the scaling parameter $\lambda_{\mathrm{da}} > 0$ of the DARWM, leading to an estimate, $\lambdahat_{\mathrm{da}}$. The rationale for these is provided in Section \ref{sec.optimisation}. Analogous guidelines for tuning both the scaling and the number of particles in a  DAPsMRWM algorithm are presented in the \emph{Supplementary Material}.

Assume that a RWM algorithm targeting $\pi$ has been constructed, with associated approximately optimal scaling $\lambdahat_{\textrm{rwm}} > 0$ and corresponding empirical acceptance rate $\alphahat_{\textrm{rwm}}(\lambdahat_{\textrm{rwm}})$, perhaps found using the acceptance-rate heuristic described in Section \ref{sect.lit.review}, or perhaps by directly maximising an empirical measure of efficiency. In order to accelerate inference, the DARWM algorithm makes use of a computationally cheap approximation $\pi_a$. Our tuning guidelines for implementing the DARWM rely on two diagnostics: {\bf(1)} the empirical conditional Stage-Two acceptance rate $\alphahat_{2|1}(\lambdahat_{\textrm{rwm}})$, and {\bf (2)} the empirical relative computational cost $\eta > 0$ of evaluating $\pi_a(x)$ compared to $\pi(x)$. Standard timing functions give the latter, whilst the former is  the ratio of the number of proposals that were accepted at both stages to the number accepted at Stage-One. The quantity $\alphahat_{2|1}(\lambdahat_{\textrm{rwm}})$ is a measure of the accuracy of $\pi_a$: if the approximation were perfect, this would equal one.

\begin{figure}[h]
\begin{center}
\includegraphics[height=0.3\textheight, width=0.45\textwidth]{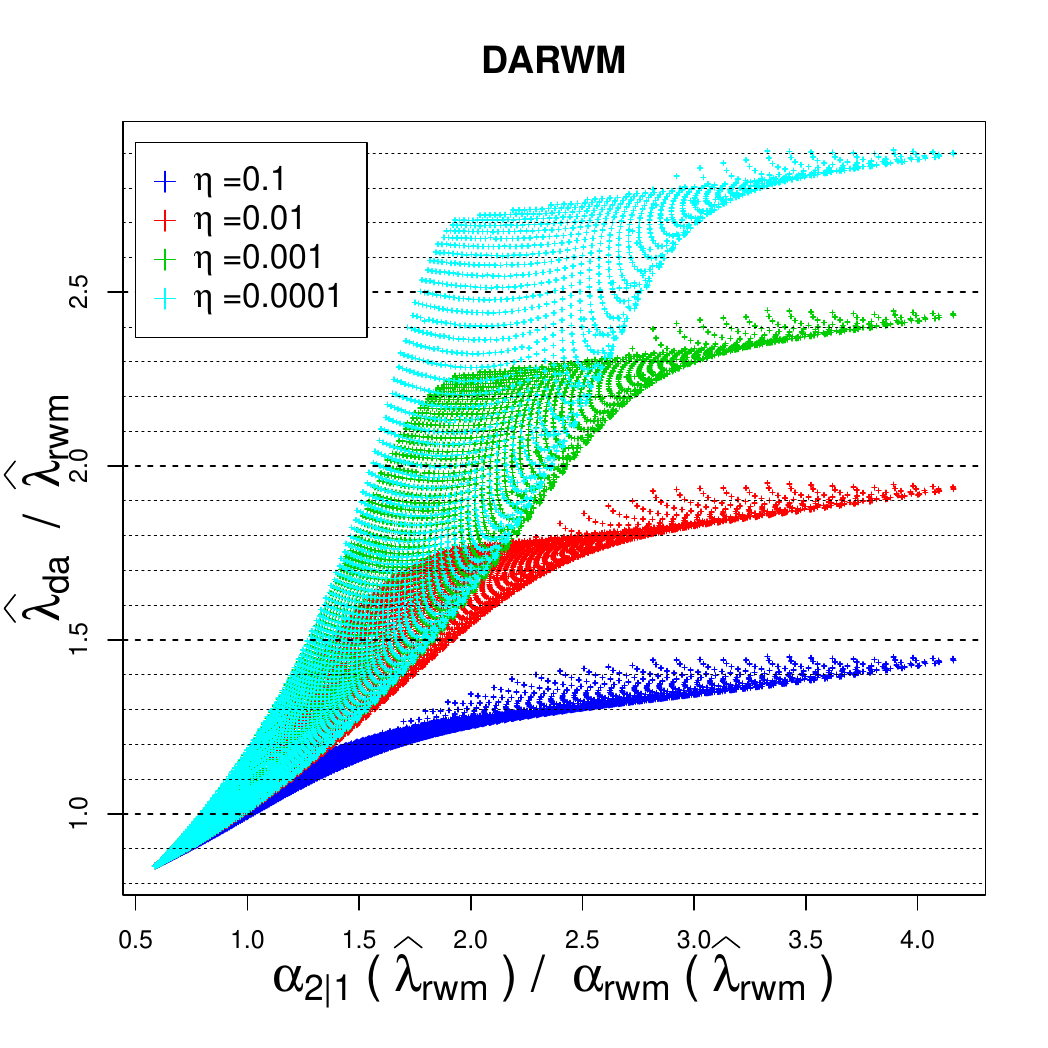}
\includegraphics[height=0.3\textheight, width=0.45\textwidth]{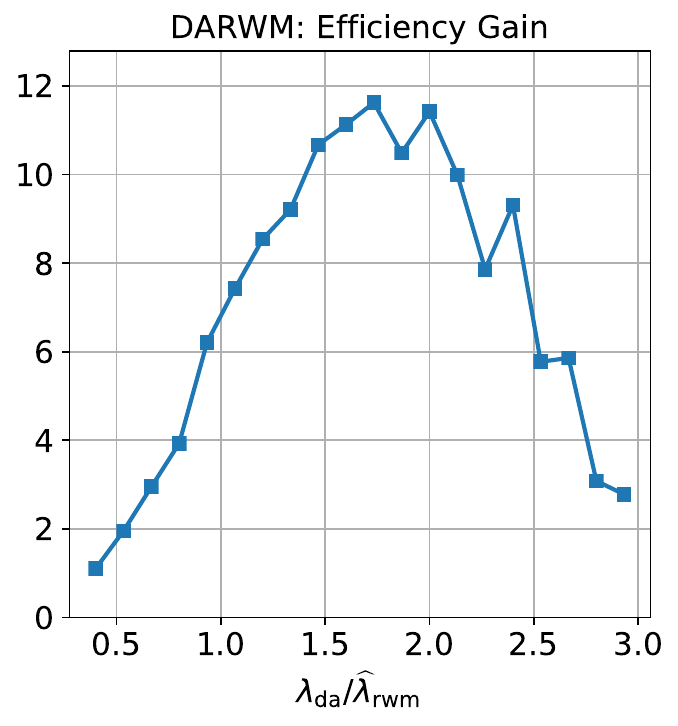}
\caption{
{\bf Left:} Range for $\lambdahat_{\textrm{da}}/\lambdahat_{\textrm{rwm}}$ as a function of $\alpha_{2|1}(\lambdahat_{\textrm{rwm}})/\alpha_{\textrm{rwm}}(\lambdahat_{\textrm{rwm}})$ for different values of $\eta$.
{\bf Right:} Efficiency gains for the ODE model \eqref{eq.ode.model}, defined as the ratio between the minESS/sec when running the DARWM algorithm with parameter $\lambda_{\textrm{da}}$, and the minESS/sec when running the optimally tuned RWM with scale parameter $\widehat{\lambda}_{\textrm{rwm}}$.
\label{fig.da.optimal.lambda}
}
\end{center}
\end{figure}

Figure \ref{fig.da.optimal.lambda}, the creation of which is explained in Section \ref{sec.DARWMEFF} and a larger version of which is given in the {\it Supplementary Material}, then provides a direct look up of the range for $\lambdahat_{\textrm{da}}/\lambdahat_{\textrm{rwm}}$ given $\eta$ and $\alphahat_{2|1}(\lambdahat_{\textrm{rwm}})/\alphahat_{\textrm{rwm}}(\lambdahat_{\textrm{rwm}})$; provided that the approximation is reasonably accurate, which is when delayed acceptance is most helpful, this range is narrow. Given a tuned RWM algorithm with optimal scaling $\widehat{\lambda}_{\mathrm{rwm}}$ and acceptance rate $\alphahat_{\textrm{rwm}}(\lambdahat_{\textrm{rwm}})$, the tuning of the DARWM proceeds as follows:
\begin{enumerate}
\item Run the DARWM with scaling $\widehat{\lambda}_{\mathrm{rwm}}$ and estimate
  $\alpha_{2|1}(\widehat{\lambda}_{\mathrm{rwm}})$.
\item Determine $\eta$ by timing evaluations of $\pi$ and $\pi_a$.
  \item Set $\widehat{\lambda}_{\mathrm{da}}$ according to Figure \ref{fig.da.optimal.lambda} (left).
\end{enumerate}

If the envelope of possible values on the y-axis is wide we suggest using the upper bound of the envelope (see the {\it Supplementary Material}). Moreover, in examples here and in the \emph{Supplementary Material} the predicted optimal scaling is slightly below the true optimum. The practitioner might, therefore, wish to consider a slight upwards shift of the predicted optimal scaling.

We now demonstrate the method and check its efficacy on two examples; further examples are given in the {\it {\it Supplementary Material}}. Throughout, we measure empirical efficiency in terms of the minimum (over all components) number of effective samples per second, minESS/sec. 
Details of parameter settings, priors \emph{etc} are provided in the {\it Supplementary Material}. 

\begin{example}
\label{example.darwmA}A four-state Markov modulated Poisson Process  \cite{FearnheadSherlock2006} is initialised in state $1$, and only transitions $1\rightarrow 2\rightarrow 3\rightarrow 4\rightarrow 1$ are allowed. The model has $8$ unknown parameters. Proposed jumps are $N(0,\lambda^2 I_8)$, and the approximation, $\pi_a$ is the product of the prior density and a Student-$t_5$ density centred at the maximum likelihood estimate (MLE) and with parameter $\Sigma^{-1}$ set to the negative Hessian at the MLE.   

The RWM is approximately optimised at $\lambdahat_{\textrm{rwm}}\approx 0.11$ which gives $\alphahat_{\textrm{rwm}}(\lambdahat_{\textrm{rwm}})\approx 0.19$ and minESS/sec of $1.28$. The single run of the DARWM gives $\alphahat_{2|1}(\lambdahat_{\textrm{rwm}})\approx 0.75$ and $1/\eta \approx 2\times 10^4$. Figure \ref{fig.da.optimal.lambda} with this $\eta$ and with $\alphahat_{2|1}(\lambdahat_{\textrm{rwm}})/\alphahat_{\textrm{rwm}}(\lambdahat_{\textrm{rwm}})\approx 3.93$ suggests $\lambdahat_{\textrm{da}}/\lambdahat_{\textrm{rwm}}\approx 2.9$; i.e. $\lambdahat_{\textrm{da}}\approx 0.32$. 

The minESS/sec at $\lambdahat_{\textrm{da}}$ was $\approx 17.4$,
a thirteenfold improvement.
A thorough grid search using long MCMC runs suggested the true optimal scaling was $\approx 0.4$, although the minESS/sec was only slightly improved, at $18.7$. 
\end{example}

\begin{example} \label{example.ode} 
Consider the ordinary differential equation (ODE) in $\RR^5$:
\begin{align}
\label{eq.ode.model}
\dot{x}_t 
= \phi\BK{x_t \, (1 - x_t) + \bra{x_t, A \, x_t}}
\end{align}
where the product $x_t \, (1 - x_t)$ is to be understood component-wise, $A \in \RR^{5,5}$ is a skew-symmetric matrix, $\bra{\cdot, \cdot}$ is the usual dot-product in $\RR^5$ and the real-valued function $\phi(u) \equiv 20 \cdot \arctan(u/20)$ is applied component-wise. We collect observations at discrete time $t_k = k \, \Delta t$ for $\Delta t \leq t_k \leq 4$ and $\Delta t = 0.2$ with additive Gaussian noise: $y_k = x(t_k) + \xi_k$ with $\xi_k \sim \Normal{0, \sigma^2_{\textrm{noise}}}$, with $\sigma_{\textrm{noise}}=0.03$ fixed and known. For a known initial position $x_0$ and from the set of noisy observations $\{y_k\}_{k=1}^{19}$, we would like to infer the $d=10$ unknown coefficients of the skew-symmetric matrix $A \in \RR^{5,5}$.  We approximate the solution of the ODE with a standard Euler discretization with step $\epsilon = 10^{-3}$. The posterior distribution $\pi$ is highly non-isotropic. Proposed jumps are
$N(0,\lambda^2\widehat{\Sigma})$, where $\widehat{\Sigma}$ is an approximation of the covariance matrix of $\pi$ estimated from a preliminary RWM run. Optimal efficiency is obtained with $\widehat{\lambda}_{\textrm{rwm}} \approx 0.75$, leading to an acceptance rate of $\alphahat_{\textrm{rwm}}(\lambdahat_{\textrm{rwm}})\approx 0.15$. As the computationally cheap approximate posterior distribution $\pi_a$, we simply use a coarser Euler discretization with step $\epsilon = 10^{-1}$, leading to a speed-up of $1/\eta \approx 10^2$. Running a DARWM at $\lambdahat_{\textrm{rwm}}$, one obtains that $\alpha_{2|1}(\lambdahat_{\textrm{rwm}}) \approx 0.55$. Figure \ref{fig.da.optimal.lambda} with this $\eta = 10^{-2}$ and with $\alphahat_{2|1}(\lambdahat_{\textrm{rwm}}) / \alphahat_{\textrm{rwm}}(\lambdahat_{\textrm{rwm}}) \approx 0.55 / 0.15 \approx 3.6$ suggests $\lambdahat_{\textrm{da}}/\lambdahat_{\textrm{rwm}} \approx 1.9$; i.e. $\lambdahat_{\textrm{da}}\approx 1.4$. This is in good agreement with the results presented in Figure \ref{fig.da.optimal.lambda} (right) and leads to an approximately elevenfold efficiency gain (measured in minESS/sec).
\end{example}

In both examples, choosing $\lambdahat_{\textrm{rwm}}$ so that $\alphahat_{\textrm{rwm}}(\lambdahat_{\textrm{rwm}})\approx 23\%$ led to a slightly lower $\lambdahat_{\mathrm{da}}$ and a minESS/sec approximately $80\%$ of the optimum.

%
%
\section{High dimensional regime}
\label{sec.high.dim}
In this section we introduce the high-dimensional asymptotic regime to
be analysed in Sections \ref{sec.asymp.analysis} and
\ref{sec.optimisation}. In Section
\ref{sec.product.form}, the target distributions are  described. In
Section \ref{sec.deter.approx} and \ref{sec.stoch.approx}
respectively, we introduce the deterministic and stochastic
approximation to the target distribution and the associated
notations.

\subsection{Product form target distributions}
\label{sec.product.form}
We consider target densities that have a simple product form. A research program along these lines was
initiated in the pair of papers \cite{Roberts/Gelman/Gilks:1997,RobertsRosenthal:1998}. Although only
simple exchangeable product form targets were considered, a range of subsequent
theoretical analyses confirmed that the results obtained in these articles also hold for more complex target distributions, such as products of one-dimensional distributions with different variances and elliptically
symmetric distributions \cite{Roberts/Rosenthal:2001,breyer2004optimal,Sherlock/Roberts:2009,BedardA:2007,Sherlock/Fearnhead/Roberts:2010}. 
We consider a target distribution $\pi^{(d)}(d\bmx)$ in $\RR^d$ with a density $\pi^{(d)}(\bmx)$ with respect to the Lebesgue measure that can expressed as
\begin{align} \label{eq.target.distribution}
\pi^{(d)}(\bmx) = \pi^{(d)}(x_1, \ldots, x_d) = \prod_{i=1}^d \pi(x_i)
\end{align}
for a one-dimensional density $\pi(x) \equiv \pi^{(1)}(x)$ on the real line. 
Furthermore, we only consider the random walk algorithm with Gaussian perturbations. For a current position $\bmx \in \RR^d$, the proposal $\bmX^*$ is distributed as
\begin{align}
\label{eq.proposal.high.dim}
\bmX^* = \bmx + \lambda^{(d)} \, \bmZ^{(d)}
\qquad \textrm{with} \qquad
\lambda^{(d)} = {\mu} \, d^{-1/2}\,/\,I 
\end{align}
and a standard centred Gaussian random variable $\bmZ^{(d)}$ and a tuning parameter $\mu>0$. The  target-dependent coefficient $I > 0$ is given by
\begin{align} \label{eq.I}
I^2
= \Expect{  \partial_x (\log \pi)(X)^2 }
= -\Expect{ \partial_{xx} (\log \pi)(X) }
\end{align}
for a scalar random variable $X \dist \pi$. The second equality in
Equation \eqref{eq.I} follows from an integration by parts that is
justified, for example, by the regularity Assumptions
\ref{ass.regularity} described in Section \ref{sec.asymp.analysis}. The constant $I>0$ is introduced to simplify the statements of the results to follow.
The scaling $d^{-1/2}$ ensures that, in the high-dimensional regime $d \to \infty$ the mean acceptance probability of a standard Random Walk Metropolis algorithm with proposals \eqref{eq.proposal.high.dim} and target distribution \eqref{eq.target.distribution} stays bounded away from zero and one. Under mild assumptions, this scaling is optimal \cite{Roberts/Gelman/Gilks:1997,BedardA:2007,BeskosRobertsStuart:2009,mattingly2012diffusion}.

%
%
\subsection{Deterministic approximation}
\label{sec.deter.approx}
To circumvent the difficulty of characterising the infinite variety of problem-specific errors in the cheap approximations $\pi_a$ to the posterior distribution $\pi$, we model the discrepancy $s(\bmx)=\log[ \pi_a(\bmx) / \pi(\bmx)]$ as the realisation of a random function. In our setting the target distribution is a $d$-dimensional product of one-dimensional distributions and we imagine that each of the terms in this product is approximated through an independent realisation of a random function.
Thus, the deterministic approximation
$\pi^{(d)}_a(\bmx) = \pi^{(d)}(\bmx) \times \exp\left( s^{(d)}(\bmx) \right)$ to
the posterior density $\pi^{(d)}(\bmx)$ has a deterministic
error, on a logarithmic scale:
\begin{align} \label{eq.deterministic.log.error}
s^{(d)}(\bmx) = \sum_{i=1}^d \SS(x_i, \aux_i),
\end{align}
where $\{\aux_i\}_{i \geq 1}$ is the realisation of an i.i.d sequence
of auxiliary random variables $\{\Aux_i\}_{i \geq 1}$. Without loss of
generality, we can assume that these auxiliary random variables are
uniformly distributed on the interval $[0,1]$. 
We assume that the deterministic function $\SS: \RR \times [0,1] \to \RR$ in Equation \eqref{eq.deterministic.log.error} satisfies the regularity Assumptions \ref{ass.regularity} stated below. The following two properties of the function $\SS$ directly influence the limiting efficiency of the delayed acceptance algorithm,
\begin{align}
\label{eq.betas}
\beta_1 \, = \, {  \Expect{ \, \partial_{xx} \SS(X, \Aux) \,} }/{I^2} 
\quad \textrm{and} \quad
\beta_2 \, = \, \curBK{ { \Expect{ \, \partial_x \SS(X,\Aux)^2 \, } }/{I^2} }^{1/2} ,
\end{align}
where expectation is taken over two independent random variables $\Aux \dist \textrm{Uniform}([0,1])$ and $X \dist \pi$. Equation \eqref{eq.betas} and an integration by parts give that
%
$I^2 \, |\beta_1| 
 = 
|\Expect{\partial_{xx} \SS(X, \Aux)}| 
 =  |\Expect{ \partial_x (\log \pi)(X) \, \partial_x \SS(X,\Aux) }|
$.
The Cauchy-Schwarz inequality and the definition \eqref{eq.I} of the quantity $I>0$ then give that
the coefficients $\beta_1, \beta_2$ satisfy the inequality
\begin{align} \label{ineq.betas}
-\beta_2 \leq \beta_1  \leq \beta_2.
\end{align}
%

%
%
\subsection{Stochastic approximation}
\label{sec.stoch.approx}
We now describe our modeling assumptions on the stochastic approximations to the target distributions $\pi^d$. For modeling purposes, it is more natural to express the stochastic approximation on a logarithmic scale and define the new quantity $w \in \RR$ as
\begin{align*}
\pih^{(d)}(\bmx, \bmu) = \pi^{(d)}(\bmx) \, e^w.
\end{align*}
In other words, $w = \log[\pih^{(d)}(\bmx, \bmu) / \pi^{(d)}(\bmx)]$. For a given value of $\bmx \in \RR^d$, the distribution of the quantity $\log[\pih^{(d)}(\bmx, \bmU^*) / \pi^{(d)}(\bmx)]$, where $\bmU^* \sim \rho(d \bmu)$, is denoted as $\pi_{W^*|X^*}(w^* | \bmx)$. The Markov-Chain $\{(\bmx_k, \bmu_k)\}_{k \geq 0}$ on $\RR^d \times \U$ can equivalently be described as a Markov Chain $\{(\bmx_k, w_k)\}_{k \geq 0}$ on $\RR^d \times \RR$. At the $k$-th iteration, the proposal $(\bmx_k, \bmu_k) \mapsto (\bmx^*, \bmu^*)$ is equivalently expressed as $(\bmx_k, w_k) \mapsto (\bmx^*, w^*)$ where, conditionally upon $\bmx^* \in \RR^d$, the proposal $w^* \in \RR$ is distributed as $\pi_{W^*|X^*}(w^* | \bmx^*)$. The property $\int \pih^{(d)}(\bmx, \bmu) \, \rho(\bmu) \, d \bmu = \pi(\bmx)$ means that for any $\bmx \in \RR^d$ we have $\int e^{w^*} \, \pi_{W^*|X^*}(w^*|\bmx) \, dw^* = 1$. Furthermore, since the Markov Chain $\{(\bmx_k, \bmu_k)\}_{k \geq 0}$ is reversible with respect to the density $\pih^{(d)}(\bmx, \bmu) \, \rho(\bmu)$, one can check that the Markov Chain $\{(\bmx_k, w_k)\}_{k \geq 0}$ is reversible with respect to the density
\begin{align}
\label{joint.posterior}
\pi^{(d)}(\bmx) \, e^w \, \pi_{W^*|X^*}(w | \bmx).
\end{align}
In the remainder of this article, we write $\accb{\bmx;w}{\bmx^*;w^*}$ for denoting the Stage-Two acceptance probability of the proposal $(\bmx, w) \to (\bmx^*, w^*)$.\\

\noindent
{\bf Standard Asymptotic Regime:} similarly to the articles \cite{Pittetal:2012,Doucetetal:2013,SherlockThieryRobertsRosenthal:2013}, we adopt the following three assumptions. These modeling assumptions constitute the standard asymptotic regime alluded to in Section \ref{sect.lit.review}. 
%
%
\begin{assumptions}
\label{ass.noise.diff.indep}
The distribution of the additive noise $W^*$ in the estimated log-target is independent of the proposal value itself. There exists a density $\pi_{W^*}(w^*)$ such that, for any $\bmx^* \in \RR^d$, $\pi_{W^*|X^*}(w^*|\bmx^*)=\pi_{W^*}(w^*)$.
\end{assumptions}
An asymptotic argument justifying this assumption for panel data, where the unbiased estimate is obtained from a product of importance-sampling estimates, and hidden-Markov models, where it is obtained from a particle filter, using the posterior concentration as the number of observations increases to infinity is given in \cite{SDDP2018}. 
It follows from Equation \eqref{joint.posterior} that the Markov Chain $\{(\bmx_k, w_k)\}_{k \geq 0}$ is reversible with respect to distribution $\pi^{(d)} \otimes \pi_W$ where the real valued distribution $\pi_W$ is given by the change of probability
\begin{align}
\label{eq.change.W.W}
\frac{d \pi_{W}}{d \pi_{W^*}}(w) = \exp(w).
\end{align}
This is Lemma $1$ of \cite{Pittetal:2012}.
In Section \ref{sec.optimisation}, we examine the behaviour of the algorithm under the following Gaussian assumption.
%
%
\begin{assumptions}
\label{ass.noise.gauss}
In addition to being independent of the proposal, $\bmX^*$, the additive noise in the estimated log-target at the proposal, $W^*$, is Gaussian: 
\begin{align}
\label{eqn.gauss.noise}
W^* \dist \Normal{ -\sigma^2 / 2, \sigma^2 }.
\end{align}
\end{assumptions}
In Equation \eqref{eqn.gauss.noise} the mean is determined by the variance so as to give an unbiased estimate of the posterior, $\Expect{\exp \left( W^* \right) }=1$. It follows from \eqref{eq.change.W.W} that at stationarity, under Assumptions \ref{ass.noise.gauss}, we have
\begin{align}
\label{eqn.gauss.noise.W}
W \dist \Normal{ \sigma^2 / 2,\sigma^2 }.
\end{align}
This article focuses on algorithms where the stochastic
approximation to the likelihood is computationally expensive. In most
scenarios of interest
\cite{GolightlyWilkinson:2011,Knape/deValpine:2012,GolightlyHendersonSherlock:2013,FilipponeGirolami:2014}
the stochastic approximation is obtained through Monte-Carlo methods
(e.g. importance sampling, particle filter) that converge at the standard $N^{-1/2}$ rate where $N$ designates the number of samples/particles used.
To take into account the computational costs necessary to produce a
stochastic estimate of the target-density, we thus assume the
following in the rest of this article.
%
%
\begin{assumptions}
\label{ass.effort.vs.variance}
The computational time required to obtain an estimate of the log-target density with variance $\sigma^2$ is: $\textrm{(compute-time)} \propto \sigma^{-2}$.
%
\end{assumptions}
The article \cite{berard2013lognormal} shows that for state-space models (and panel data) the unbiased estimate of the likelihood obtained from standard particle methods \cite{DelMoral:2004} (or a product of importance sampling estimators) satisfies a log-normal central limit theorem, as the number of observations and particles (or importance samples) goes to infinity, if this number is of the same order as the number of noisy observations.
This justifies the Gaussian approximation \eqref{eqn.gauss.noise} and shows that the log-error is asymptotically inversely proportional to the number of particles used, justifying Assumptions \ref{ass.effort.vs.variance}. The article \cite{sherlock2017pseudo} studies the tuning of pseudo-marginal MCMC methods when Assumptions \ref{ass.effort.vs.variance} is not appropriate.

%
%
\section{Asymptotic analysis}
\label{sec.asymp.analysis}
In this section we investigate the behaviour of the DAPsMRWM, and hence of the DARWM as a special case, in the high-dimensional regime described in Section \ref{sec.high.dim}. We make the following regularity assumptions.
%
%
\begin{assumptions}
\label{ass.regularity}
The density $\pi: \RR \to (0, \infty)$ and the function $\SS: \RR \times [0,1] \to \RR$ satisfy the following.
\begin{enumerate}
\item
The function $x \mapsto \log \pi$ is thrice differentiable, with second and third derivative bounded and $\Expect{  (\partial_x \log \pi)^2(X)}$ is finite, for
$X \dist \pi$. 
\item
The first three derivatives with respect to the first argument of the function $(x,\aux) \mapsto \SS(x, \gamma)$ exist and are bounded over $(x,\aux) \in \RR \times [0,1]$.
\end{enumerate}
\end{assumptions}
Assumptions \ref{ass.regularity} are used to control
the behaviour of second-order Taylor expansions; they could be relaxed
in several directions at the costs of increasing technicality in the
proofs. When the current position of the algorithm is $(\bmx^{(d)}, w^{(d)}) \in \RR^d \times \RR$, a proposal $(\bmx^{(d),*}, w^{(d),*})$ is generated, distributed as
\begin{align}
\label{eq.proposals.X.W}
\bmX^{(d),*}  =
\bmx^{(d)} +  {\mu } \, d^{-1/2} \, \bmZ^{(d)}\,/\,I
\qquad \textrm{and} \qquad
W^{(d),*} 
\dist \pi_{W^*},
\end{align}
where $\bmZ^{(d)}\sim\Normal{\bmzero,\bmI_d}$. Since the following quantities repeatedly appear in the analysis to follow, we set
\begin{align}
\label{eq.def.q.s}
\left\{
\begin{aligned}
q^{(d)}_{\Delta}(\bmx^{(d)}, \bmx^{(d),*}) &= \log \big[ \pi^{(d)}(\bmx^{(d),*}) / \pi^{(d)}(\bmx^{(d)}) \big],\\
s^{(d)}_{\Delta}(\bmx^{(d)}, \bmx^{(d),*}) &=  s^{(d)}(\bmx^{(d),*}) - s^{(d)}(\bmx^{(d)}).
\end{aligned}
\right.
\end{align}
The following lemma, proved in Section \ref{sec.proof.lemma.pivotal} of the \emph{Supplementary Material}, is pivotal to our analysis of the DAPsMRWM algorithm. It shows that the quantities defined in \eqref{eq.def.q.s} converge jointly to a Gaussian distribution whose parameters can be expressed in terms of the scaling $\mu > 0$ of the RWM perturbations, as well as the parameters $(\beta_1, \beta_2) \in \RR \times \RR^{+}$ that describe the properties of the deterministic approximation to the target distribution.
%
%
\begin{lem} \label{lem.pivotal}
Let the regularity Assumptions \ref{ass.regularity} hold.
Let $\{\aux_i\}_{i \geq 1}$ be a realisation of the sequence of auxiliary random variable used to described the deterministic approximation \eqref{eq.deterministic.log.error}.
Let $\{x_i\}_{i \geq 1}$ be the realisation of an i.i.d sequence marginally distributed as $\pi$. For $d \geq 1$, set $\bmx^{(d)} = (x_1, \ldots, x_d) \in \RR^d$ and let $\bmX^{(d),*}$ and $\bmZ^{(d)}$ be as defined in \eqref{eq.proposals.X.W}. 
For almost all realisations $\{x_i\}_{i \geq 1}$ and $\{\aux_i\}_{i \geq 1}$ and $w \in \RR$, the following limit 
\begin{align}
\label{eq.limiting.Q.S}
\lim_{d \to \infty} \;
  \begin{bmatrix}
    q^{(d)}_{\Delta}(\bmx^{(d)}, \bmX^{(d),*})\\
    s^{(d)}_{\Delta}(\bmx^{(d)}, \bmX^{(d),*})
  \end{bmatrix}
\;=\;
  \begin{bmatrix}
    Q^{\infty}_{\Delta}\\
    S^{\infty}_{\Delta}
  \end{bmatrix}
\dist \Normal{
-\frac{\mu^2}{2} \,
  \begin{bmatrix}
  1\\
  -\beta_1
  \end{bmatrix}
  ,
\mu^2 \,
  \begin{bmatrix}
  1&-\beta_1\\
-\beta_1&\beta_2^2
  \end{bmatrix}
}
\end{align}
holds in distribution with parameters $\beta_1$ and $\beta_2$ defined in \eqref{eq.betas}.
\end{lem}
That the correlation is $-\beta_1/\beta_2 \in [-1,1]$ is another manifestation of inequality
\eqref{ineq.betas}. 
In general, Lemma \ref{lem.pivotal} shows that
if the approximating density has an average excess of (negative) curvature (i.e. $\beta_1 < 0$), the limiting random variables $Q^{\infty}_{\Delta}$ and $S^{\infty}_{\Delta}$ are positively correlated.

The product form Assumptions \eqref{eq.target.distribution} and
\eqref{eq.deterministic.log.error} from which we derive the bivariate Gaussian distribution in Lemma
\ref{lem.pivotal} are chosen for convenience. We expect the same
conclusions to hold, at least approximately, in much broader settings. Detailed, empirical verification of 
Lemma \ref{lem.pivotal} for the delayed-acceptance ODE Example \ref{example.ode} described in Equation \eqref{eq.ode.model} and for a delayed-accept pseudo-marginal example is provided in the \emph{Supplementary Material}, demonstrating the robustness of the results proved in this article.
\subsection{Limiting acceptance probability}
\label{sec.limiting.acc.proba} Since the acceptance rates of the DAPsMRWM can be very simply expressed in terms of the quantities defined in Equation \eqref{eq.def.q.s}, Lemma \ref{lem.pivotal} leads to tractable expression for the acceptance rates as $d \to \infty$. For $\bmx^{(d)} \in \RR^d$, the Stage-One acceptance rate can be expressed as
\begin{align*}
\alpha_1^{(d)}(\bmx^{(d)})
=
\EE\sqBK{
F\BK{ q^{(d)}_{\Delta}(\bmx^{(d)}, \bmX^{(d),*})  + s^{(d)}_{\Delta}(\bmx^{(d)}, \bmX^{(d),*}) }}
\end{align*}
where $F(u) = 1 \wedge \exp(u)$ is the Metropolis-Hastings accept-reject function. Similarly, for $(\bmx^{(d)}, w^{(d)}) \in \RR^d \times \RR$, the overall acceptance rate $\acc_{12}^{(d)}(\bmx^{(d)}, w^{(d)})$ can be expressed as the expectation of the product
\begin{align*}
F\BK{ q^{(d)}_{\Delta}(\bmx^{(d)}, \bmX^{(d),*})  + s^{(d)}_{\Delta}(\bmx^{(d)}, \bmX^{(d),*})}
\times F\left( W^{(d),*} - w^{(d)}  - s^{(d)}_{\Delta}(\bmx^{(d)}, \bmX^{(d),*})\right)
\end{align*}
where $W^{(d),*} \sim \pi_{W*}$ is independent from all the other sources of randomness. The following proposition, whose proof directly follows from Lemma \ref{lem.pivotal} and the dominated convergence theorem, gives the limiting values of these acceptance rates. Convergence is to be understood in the $L^2$ sense: for a sequence of random variables $\{V_k\}_{k \geq 0}$ and a constant $V_\infty \in \RR$, the notation $\lim_{k \to \infty} \, V_k \stackrel{L^2}{=} V_\infty$ indicates that $\EE[(V_k - V_\infty)^2] \to 0$ as $k \to \infty$.
\begin{prop} \label{prop.limi.accept.proba}
Let Assumptions \ref{ass.noise.diff.indep} and \ref{ass.regularity} hold.
For almost every realisation $\{\aux_i\}_{i \geq 1}$ of the sequence of auxiliary random variables used to describe the deterministic approximation \eqref{eq.deterministic.log.error}, we have 
\begin{align*}
\lim_{d \to \infty} \; \accado{\bmX^{(d)}} \; \stackrel{L^2}{=} \; \acc_1
\quad \textrm{and} \quad
\lim_{d \to \infty}  \acc_{12}^{(d)}(\bmX^{(d)}, W^{(d)}) \; \stackrel{L^2}{=} \; \acc_{12}
\end{align*}
where the limiting acceptance rates are given by
\begin{align*}
\left\{
\begin{aligned}
\alpha_1 
&= \Expect{ F\left( Q^{\infty}_{\Delta}  + S^{\infty}_{\Delta} \right) } \\
\alpha_{12} 
&= \Expect{  F\left( Q^{\infty}_{\Delta}  + S^{\infty}_{\Delta}\right) \times F\left( W_{\Delta}  - S^{\infty}_{\Delta} \right) }
\end{aligned} \right.
\end{align*}
for $(Q^{\infty}_{\Delta} ,S^{\infty}_{\Delta} )$ as described in \eqref{eq.limiting.Q.S} and $W_{\Delta} = W^*-W$ for $(W^*,W) \dist \pi_{W^*} \otimes \pi_{W}$.
The dependence of $\alpha_1$ and $\acc_{12}$ upon $(\mu,\beta_1,\beta_2,\pi_{W})$ is implicit.
\end{prop}
For the remainder of our discussion of acceptance rates we suppose
that Assumptions \ref{ass.noise.gauss} holds: there is additive Gaussian noise in the
logarithm of the stochastic approximation. We also make the
dependence of the acceptance rate on the approximation parameters,
$\beta_1$ and $\beta_2$, explicit.
Standard computations (e.g.  Proposition $2.4$ of \cite{Roberts/Gelman/Gilks:1997}) yield that, for $\xi \sim \Normal{\mu, \sigma^2}$, we have that $\Expect{ F(\xi)} = \Phi(\mu
/ \sigma) + \exp \left(\mu + \sigma^2/2 \right) \, \Phi(-\sigma - \mu/\sigma)$, with $\Phi:\RR \to (0,1)$ the standard Gaussian cumulative
distribution function. This permits straightforward evaluation of the Stage-One acceptance rate $\alpha_1(\mu;\beta_1,\beta_2) = \Expect{F(Q^\infty_\Delta+S^\infty_\Delta)}$, 
the overall acceptance rate
\begin{align}
\label{eq.acc.gauss}
\alpha_{12}(\mu,\sigma^2;\beta_1,\beta_2)
= 
\Expect{F(Q^\infty_\Delta+S^\infty_\Delta) \times F(W_\Delta-S^\infty_\Delta)},
\end{align}
as well as the ratio 
$\alpha_{2|1}(\mu,\sigma^2;\beta_1,\beta_2)
=
\alpha_{12}(\mu,\sigma^2;\beta_1,\beta_2)/\alpha_1(\mu;\beta_1,\beta_2)$,
in terms of standard functions and, for $\alpha_{12}$, a one-dimensional numerical integral, as detailed in the  {\it Supplementary Material}.
The limit as $\beta_1\to 0$ and $\beta_2 \to 0$ corresponds to the case when
there is no deterministic error and leads to the usual
\cite{Roberts/Gelman/Gilks:1997,mattingly2012diffusion} limiting
acceptance rate of $2 \times \Phi(- \mu / 2)$. For computing the limiting
overall acceptance rate, note that under the Gaussian Assumption
\ref{ass.noise.gauss} we have $W_\Delta \dist \Normal{ -\sigma^2, 2\sigma^2}$. 

The following result, proved in the {\it Supplementary Material}, shows that it is possible to
characterise the (unknown) values of $\mu$ and $\sigma^2$ in terms of
the Stage-One and the conditional Stage-Two acceptance rates. 
\begin{prop}
\label{prop.acc.rates.dec}
Let Assumptions \ref{ass.noise.diff.indep}, \ref{ass.noise.gauss} and
\ref{ass.regularity} hold.
\begin{enumerate}
\item For any $\beta_2>0$ and $\beta_1<1$ the Stage-One acceptance rate $\alpha_1(\mu;\beta_1,\beta_2)$ is a continuous
decreasing bijection in $\mu$ from $[0,\infty)$ to $(0,1]$.
\item
For any fixed $\mu, \beta_2>0$ and $\beta_1$,
the conditional  Stage-Two acceptance rate
$\alpha_{2|1}(\mu,\sigma;\beta_1,\beta_2)$ is a decreasing bijection in $\sigma$ from $[0,\infty)$ to $(0,\alpha_{2|1}(\mu,0;\beta_1,\beta_2)]$. 
\end{enumerate}
\end{prop}
\noindent
For the DARWM algorithm, $W_\Delta\equiv0$. Therefore, Equation \eqref{eq.acc.gauss} yields that
$\alpha_{12}(\mu,0;\beta_1,\beta_2)= 
\Expect{F(Q^\infty_\Delta+S^\infty_\Delta) \times F(-S^\infty_\Delta)}$,
%
%
which can be evaluated via a one-dimensional numerical integral. 
When $\beta_1=\beta_2^2$, which necessitates $\beta_2\le 1$ by \eqref{ineq.betas}, 
$\Cov{Q^\infty_\Delta+S^\infty_\Delta,-S^{\infty}_\Delta}=0$ so that the random variables 
$Q^\infty_\Delta+S^\infty_\Delta$ and $W_\Delta-S^{\infty}_\Delta$ are independent. In that case, algebra gives that
\begin{align}
\label{eq.alpha2g1uncor}
\begin{aligned}
\alpha_{2|1}(\mu,\sigma^2;\beta_2^2,\beta_2)
&= 
2\Phi\left(-\frac{1}{2}\sqrt{\beta_2^2\mu^2+2\sigma^2}\right).
\end{aligned}
\end{align}
This is the limiting acceptance probability of a pseudo-marginal RWM algorithm with a scaling of $\beta_2 \, \mu$ and a noise variance of $\sigma^2$ \cite[see][]{SherlockThieryRobertsRosenthal:2013}. 
Substituting $\sigma^2=0$ into \eqref{eq.alpha2g1uncor}, we find that for the DARWM, 
$\alpha_{2|1}(\mu,0;\beta_2^2,\beta_2)=2\Phi(-\beta_2\mu/2)$, the limiting acceptance probability for a RWM algorithm with a scaling of $\beta_2 \, \mu$ \cite[see][]{Roberts/Gelman/Gilks:1997}. 
In Section \ref{sec.optimisation} the insights arising from this phenomenon 
help to motivate our approach to understanding the efficiency and tuning of the DARWM and DAPsMRWM algorithms.

%
%
\subsection{Limiting expected squared jumping distance}
A standard measure of efficiency \cite{Sherlock/Roberts:2009,BeskosRobertsStuart:2009,Sherlock:2013} for local algorithms is
the Euclidian Expected Squared Jumping Distance ($\esjd$); see \cite{roberts2014minimising,pasarica2010adaptively} for detailed discussions.
Theoretical motivations for our use of the $\esjd$ are given by the diffusion approximation described in Section \ref{sec.diffusion.approx}.
In our $d$-dimensional setting, it is defined as
\begin{align}
\esjd^{(d)} = \EE \Big[ \, \big\| \bmX^{(d)}_{k+1} - \bmX^{(d)}_{k} \big\|^2 \, \Big]
\end{align}
where the Markov chain $\left\{ (\bmX^{(d)}_k, W^{(d)}_k\right\}_{k \geq 0}$
is stationary and $\| \cdot \|$ is the standard Euclidian norm.
%
%
\begin{prop} \label{prop.limiting.esjd}
Let Assumptions \ref{ass.noise.diff.indep} and \ref{ass.regularity} hold and let $\alpha_{12}$ be the limit identified in Proposition \ref{prop.limi.accept.proba}. 
For almost every realisation $\{\aux_i\}_{i \geq 1}$,
\begin{align} \label{eq.limiting.esjd}
\lim_{d \to \infty} \esjd^{(d)} = \alpha_{12} \times \left( \frac{\mu}{I} \right)^2 \equiv J(\mu)
\end{align}
The dependence of the limiting expected squared jumping distance $J(\mu)$ upon $(\beta,\pi_{W})$ is implicit.
\end{prop}
\subsection{Diffusion limit} \label{sec.diffusion.approx}
We wish to prove that the DAPsMRWM algorithm in high dimensions can be
well-approximated by an appropriate diffusion limit as this provides
theoretical underpinning to our use of the $\esjd$ as measure of
efficiency \cite{bedard2012scaling,roberts2014minimising}. The
connection between $\esjd$ and diffusions arises because the
asymptotic jumping distance $\lim_{d \to \infty} \,
\esjd^{(d)}=J(\mu)$ is equal to the square of the limiting process's
diffusion coefficient and is proportional to the drift coefficient. By a simple time change argument, the asymptotic
variance of {\em any} Monte Carlo estimate of interest is inversely
proportional to $J(\mu)$. Consequently, $J(\mu)$ becomes, at least in the high-dimensional limit $d \to \infty$, unambiguously
the right quantity to optimise.

It is important to stress that the existence of the diffusion limit in this
argument cannot be circumvented. MCMC algorithms which have non-diffusion
limits can behave in very different ways and ESJD may not be a natural way
to compare algorithms.
The main result of this section is a diffusion 
limit for a rescaled version $V^{(d)}$ of the first coordinate process. For time $t \geq 0$ we define 
the piecewise constant continuous time process
\begin{align}
V^{(d)}(t) \equiv X^{(d)}_{\floor{d \times t},1} \ .
\end{align}
with the notation $\bmX^{(d)}_k = (X^{(d)}_{k,1}, \ldots, X^{(d)}_{k,d}) \in \RR^d$. In general, the process $V^{(d)}$ is not Markovian; the next theorem shows nevertheless that in the limit $d \to \infty$ the process $V^{(d)}$ 
can be approximated by a Langevin diffusion.
%
%
\begin{thm} \label{thm.diff.lim}
Let Assumptions \ref{ass.noise.diff.indep} and \ref{ass.regularity} hold.
Let $T>0$ be a finite time horizon and suppose that for all $d \geq 1$ the DAPsMRWM Markov chain  starts at stationarity, $(\bmX^{(d)}_0, W^{(d)}_0) \dist \pi^{(d)} \otimes \pi_{W}$.  Then, as $d \to \infty$, the sequence of processes $V^{(d)}$ converges weakly to $V$ in the Skorokhod topology on $D([0,T], \RR)$ where the diffusion process $V$ satisfies the Langevin stochastic differential equation
\begin{align}
\label{e.limiting.diffusion}
dV_t = \frac12 \, J(\mu) \, (\log \pi)'(V_t) \, dt + J^{1/2}(\mu) \, dB_t
\end{align}
with initial distribution $V_0 \dist \pi$. The process $B_t$ is a standard scalar Brownian motion.
\end{thm}
As with Propositions \ref{prop.limi.accept.proba} and
\ref{prop.limiting.esjd}, the Gaussian Assumption
\ref{ass.noise.gauss} is not necessary for the conclusion of Theorem
\ref{thm.diff.lim} to hold. The proof can be found in Section
\ref{sec.proof.diff.lim} of the {\it Supplementary Material}. It relies on an homogenization 
argument necessary to average-out the rapidly mixing $\bmu$-process.
Theorem \ref{thm.diff.lim} shows that the
rescaled first coordinate process converges to a Langevin diffusion
$V$ that is a time-change of the diffusion $d \overline{V}_t = 
\frac12 \, (\log \pi)'( \overline{V}_t) \, dt + dB_t$; indeed, $t \mapsto V_{t}$
has the same law as $t \mapsto \overline{V}_{J(\mu) \, t}$. This
reveals that when speed of mixing is measured in terms of the number
of iterations of the algorithm, the higher $J(\mu)$, the faster the mixing of the Markov
chain.  See \cite{roberts2014complexity} for a detailed
discussion and rigorous results. However any measure of overall
efficiency should
also take into account the computational time required for each
iteration of the algorithm, and this is the subject of the next section.

%
%
\section{Optimising the efficiency}
\label{sec.optimisation}
When examining the efficiency of a standard RWM the computational time is usually either
not taken into account or is implicitly supposed to be independent of the choice of tuning
parameter(s). In any delayed-acceptance scenario, the computational time depends on the number of acceptances at Stage-One; furthermore, in any pseudo-marginal setting the computational time also depends on the variance of the stochastic estimate of $\log \pi$. 
For this article, we measure the efficiency through a rescaled version of the expected squared jump distance,
\begin{align} \label{def.efficiency.rough}
\textrm{(Efficiency)} \equiv \frac{\textrm{(Expected Squared Jump Distance)}}{\textrm{(Averaged one-step computing time)}} .
\end{align}
The quantity 
$\mathscr{F}(\esjd) / \textrm{(Averaged one step computing time)}$, for any increasing function $\mathscr{F}$, is a valid measure of efficiency \cite[e.g.][]{Pittetal:2012,Doucetetal:2013,SherlockThieryRobertsRosenthal:2013,GolightlyHendersonSherlock:2013}; the discussion at the start of Section 
\ref{sec.diffusion.approx} reveals nonetheless, because of the diffusion approximation proved in Theorem \ref{thm.diff.lim}, that  \eqref{def.efficiency.rough} is the essentially unique measure of efficiency valid in the high-dimensional asymptotic regime considered in this article.
Proposition \ref{prop.limiting.esjd} shows that the limiting $\esjd$ equals $\alpha_{12} \times (\mu / I)^2$ where $I$, defined in Equation \eqref{eq.I}, is a constant irrelevant for the optimisation of the efficiency discussed in this section; the constant also appears in the same form in the limiting $\esjd$ for the equivalent non-delayed acceptance algorithm, and so it may also safely be ignored when calculating relative efficiencies. We examine the efficiency of the DARWM first, then move on to the DAPsMRWM. 

\subsection{Delayed-acceptance random walk Metropolis}
\label{sec.DARWMEFF}
For the DARWM we define an evaluation of $\pi$ as taking one unit of time and 
define $\eta>0$ to be the time for an evaluation of $\pi_a$: the one-step cost 
of a DARWM algorithm is $\eta+\alpha_1$. 
Following Equation \eqref{def.efficiency.rough} and eliminating unnecessary constants, the limiting efficiency of the DARWM can be quantified by
the following efficiency functional
\begin{align} \label{eq.efficiency.functional.noW}
\eff_{\textrm{da}}(\mu) = \frac{\mu^2 \, \alpha_{12}(\mu, 0)}{\eta \,  + \alpha_1(\mu)},
\end{align}
with the dependence upon $\beta_1$ and $\beta_2$ implicit. In the limit where $\pi_a\equiv \pi$ and $\eta=0$ the movement of the Markov chain becomes that of a RWM on $\pi$, but the efficiency reduces to $\mu^2$ rather than the $\mu^2\alpha_1$ of a RWM because, in this limit, only acceptances are associated with a computational cost.
Using the same timescale, the efficiency of the RWM is
$\eff_{\textrm{rwm}}(\mu)\equiv2\mu^2\Phi(-\mu/2)$ \cite{Roberts/Gelman/Gilks:1997}, which is
optimised at $\mu=\muhat_{\textrm{rwm}}\approx 2.38$. We therefore define
the relative efficiency of the DARWM algorithm to the optimal
efficiency of the RWM:
\begin{align}
\label{eqn.eff.rel.CF}
\eff^{\textrm{rel}}_{\textrm{da}}(\mu)\equiv
\frac{\eff_{\textrm{da}}(\mu)}{\eff_{\textrm{rwm}}(\muhat_{\textrm{rwm}})}.
\end{align}
In the special case of $\beta_1=\beta_2^2$, and as investigated in and around \eqref{eq.alpha2g1uncor}, 
\begin{align}
  \label{eqn.simple.special}
\eff^{\textrm{rel}}_{\textrm{da}}(\mu;\beta_2^2,\beta_2)
=
\frac{\mu^2}{\muhat^2_{\textrm{rwm}}}
\frac{\alpha_1(\mu;\beta_1,\beta_2) \, \Phi(-\mu\beta_2/2)}
{(\eta+\alpha_1(\mu;\beta_1,\beta_2)) \, \Phi(-\muhat_{\textrm{rwm}}/2)}.
\end{align}
In the limit as $\eta\downarrow 0$,
the efficiency is maximised 
at $\muhat_{\textrm{da}}=\muhat_{\textrm{rwm}}/\beta_2$, giving an overall relative efficiency of $\eff^{\textrm{rel}}_{\textrm{da}}(\muhat_{\textrm{da}})=1/\beta_2^2$. 
In reality, $\eta>0$, and if $\mu$ is large enough so that $\alpha_1(\mu) \lesssim \eta$ then
$\mu^2\alpha_{12}(\mu;\beta_1,\beta_2)$ decreases rapidly with $\mu$, as does the efficiency.
This suggests that $\alpha_{2|1}(\muhat_{\textrm{rwm}};\beta_1,\beta_2)$ might 
provide insight into the optimal scaling, $\muhat_{\textrm{da}}$, 
provided that $\eta>0$ is also taken into account. 
Figure \ref{fig.effRWM.betas} shows $\alpha_{2|1}(\muhat_{\textrm{rwm}};\beta_1,\beta_2)$ and 
$\muhat_{\textrm{da}}/\muhat_{\textrm{rwm}}$ 
as functions of $\beta_1$ and $\beta_2$ when $\eta=0.01$. 
The shapes of the contours are almost identical (as is the shape for $\eff^{\textrm{rel}}_{\textrm{da}}(\muhat_{\textrm{da}})$, though not shown), indicating 
that \emph{whatever the values of $\beta_1$ and $\beta_2$}, the quantity
$\alpha_{2|1}(\muhat_{\textrm{rwm}};\beta_1,\beta_2)$ provides information 
on the optimal increase in scaling, relative to the optimal scaling for the RWM, as well as 
the corresponding increase in efficiency.
Along the line where $\beta_1=\beta_2^2$, as predicted, at $\beta_2=1$, $\muhat_{\textrm{da}}\approx\muhat_{\textrm{rwm}}/\beta_2$, but, since $\eta>0$, as $\beta_2$ decreases the optimal scaling increase less quickly than \eqref{eqn.simple.special} suggests.
\begin{figure}
\begin{center}
\subfigure{
  \includegraphics[width = 0.4 \textwidth,angle=0]{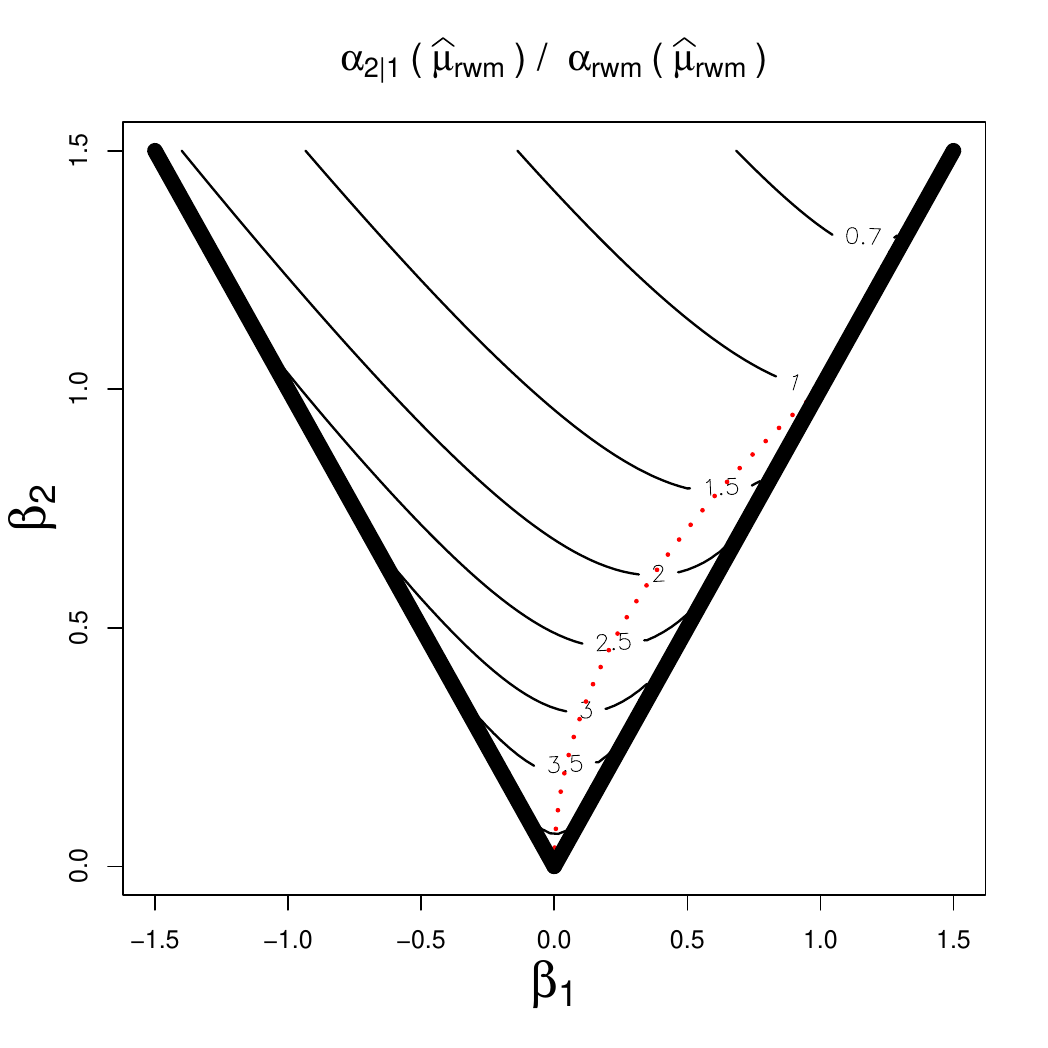}
}
\subfigure{
  \includegraphics[width = 0.4 \textwidth,angle=0]{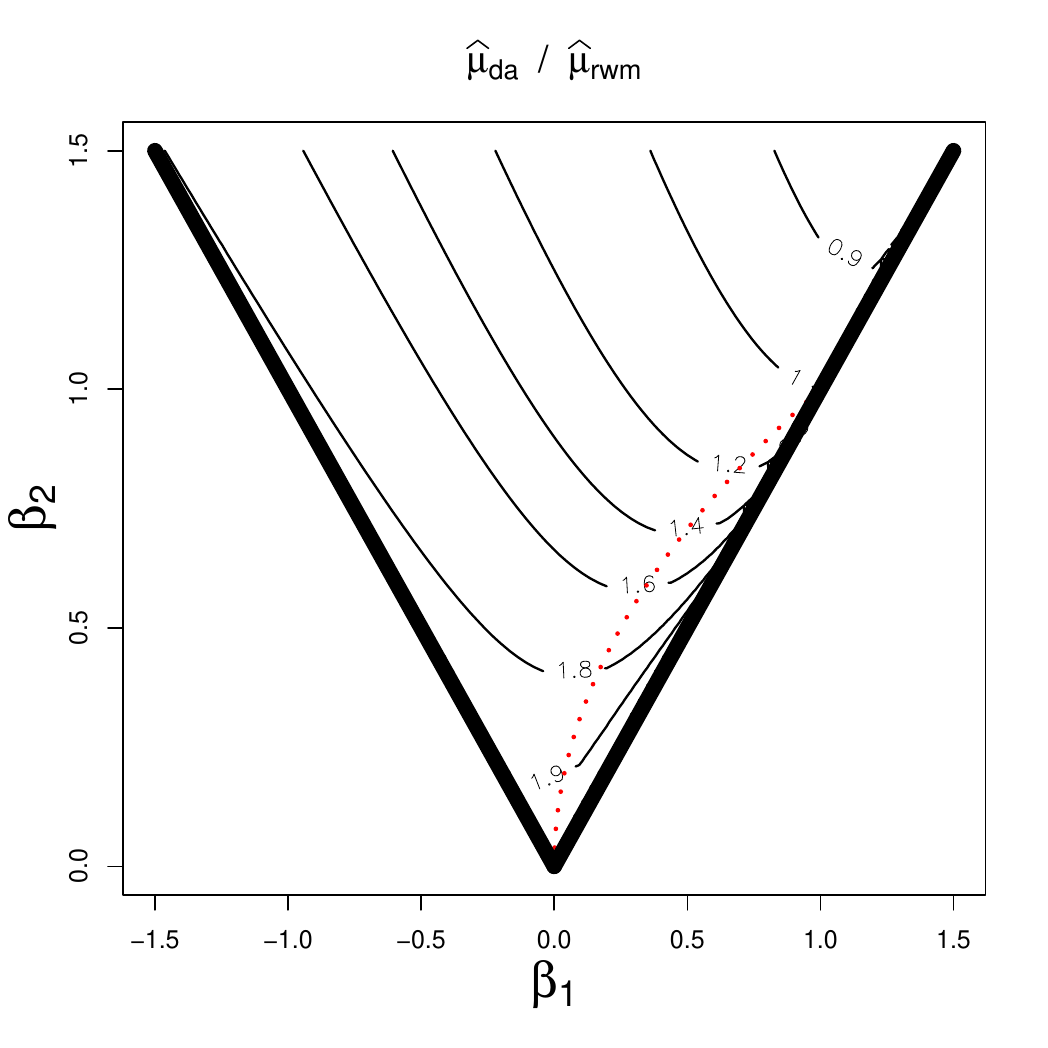}
}
\caption{
Contour plots of $\alpha_{2|1}(\muhat_{\textrm{rwm}},0;\beta_1,\beta_2)$ (left) and  $\muhat_{\textrm{da}}/\muhat_{\textrm{rwm}}$ (right), 
 as a
function of $\beta_1$ and $\beta_2$ for $\eta=0.01$. The red, dotted line satisfies $\beta_1=\beta_2^2$.
\label{fig.effRWM.betas}
}
\end{center}
\end{figure}

Figure
\ref{fig.da.optimal.lambda} (left)
is in fact a plot of $\muhat_{\textrm{da}}/\muhat_{\textrm{rwm}}$  
vs $\alpha_{2|1}(\muhat_{\textrm{rwm}})/\alpha_{\textrm{rwm}}(\muhat_{\textrm{rwm}})$ over the fine grid of values of $(\beta_1,\beta_2)$ used to create Figure \ref{fig.effRWM.betas}. Since $\mu\propto\lambda$, from \eqref{eq.proposal.high.dim}, 
 $\muhat_{\textrm{da}}/\muhat_{\textrm{rwm}}\equiv\lambdahat_{\textrm{da}}/\lambdahat_{\textrm{rwm}}$; the constant of proportionality is unknown, which is why we provide a graph for the ratio.
It suggests that $\alpha_{2|1}(\muhat_{\textrm{rwm}})$ combined with $\eta$ does indeed provide information on the relative increase in scaling needed over $\muhat_{\textrm{rwm}}$.

\subsection{Delayed-acceptance pseudo-marginal RWM}
\label{sec.dapmrwm}
For the DAPsMRWM we define an evaluation of $\pih$ with $\sigma^2=1$ as taking one unit of time, 
and $\eta>0$ is defined to be the time for an evaluation of $\pi_a$ on this scale. Under Assumption \ref{ass.effort.vs.variance}, the average time needed to compute the stochastic approximation is inversely proportional to the variance, $\sigma^2$, of the estimate of the log-target,
which leads to an average computational time for a single iteration of the algorithm of:
$\eta + \alpha_1 / \sigma^{2}$.
%
%
As discussed in Section \ref{sec.stoch.approx}, 
Assumption \ref{ass.effort.vs.variance} is reasonable when using particle MCMC to perform inference on the parameters of a hidden-Markov model, or when analysing panel data using a product of importance sampling estimators. Hence, we simplify notation and refer to the resulting efficiency as that of a Delayed-Acceptance Particle-Marginal method. Our efficiency functional is, therefore,

%
%
\begin{align} \label{eq.efficiency.functional}
\eff_{\textrm{dapm}}(\mu, \sigma^2) = \frac{\mu^2 \, \sigma^2 \, \alpha_{12}(\mu, \sigma)}{\eta \, \sigma^2 + \alpha_1(\mu)}.
\end{align}
Theorem \ref{thrm.behave.eff}, which is proved in the {\it Supplementary Material}, shows that $\eff_{\textrm{dapm}}(\mu, \sigma^2)$ possesses intuitive limiting properties.
\begin{thm}
\label{thrm.behave.eff}
Let Assumptions \ref{ass.noise.gauss}, \ref{ass.effort.vs.variance} and \ref{ass.regularity} hold. Then:
\begin{enumerate}
\item
For a fixed variance $\sigma^2>0$, $\eff_{\textrm{dapm}}(\mu, \sigma^2) \to 0$ as $\mu \to 0$ or $\mu \to \infty$.
\item
For a fixed jump size $\mu>0$, $\eff_{\textrm{dapm}}(\mu, \sigma^2) \to 0$ as $\sigma^2 \to 0$ or $\sigma^2 \to \infty$.
\end{enumerate}
\end{thm}
Using the same time scale as in \eqref{eq.efficiency.functional}, the
equivalent efficiency function for the Particle-Marginal RWM is
$\eff_{\textrm{pm}}(\mu,\sigma^2)\equiv2\mu^2\sigma^2\Phi\left(-\frac{1}{2}\sqrt{\mu^2+2\sigma^2}\right)$,
and this
is maximised at $\muhat_{\textrm{pm}} \approx 2.562$ and
$\sigmahat_{\textrm{pm}}^2\approx 3.283$
\cite[]{SherlockThieryRobertsRosenthal:2013}. 
Thus, we  define the efficiency of the DAPsMRWM relative to the maximum
achievable efficiency of the Particle-Marginal RWM as:
\begin{align}
\label{eqn.eff.rel.PM}
\eff^{\textrm{rel}}_{\textrm{dapm}}(\mu,\sigma^2)\equiv
\frac{\eff_{\textrm{dapm}}(\mu,\sigma^2)}{\eff_{\textrm{pm}}(\muhat_{\textrm{pm}},\sigmahat_{\textrm{pm}}^2)}.
\end{align}
An argument analogous to the one used for analyzing the DARWM suggests that $\alpha_{2|1}(\muhat_{\textrm{pm}}\sigmahat^2_{\textrm{pm}})$ and $\eta>0$ together should be informative on $\muhat_{\textrm{dapm}}$ and 
$\sigmahat^2_{\textrm{dapm}}$. 
Analogous contour plots to those in Figure \ref{fig.effRWM.betas}, provided in the {\it Supplementary Material}, 
show the same key property. Scatter plots of $\muhat_{\textrm{dapm}}$ and $\sigmahat^2_{\textrm{dapm}}$ 
against $\alpha_{2|1}(\muhat_{\textrm{pm}},\widehat{\sigma^2}_{\textrm{pm}})$ partitioned by $\eta$, analogous to Figure \ref{fig.da.optimal.lambda} (left), are provided in the {\it Supplementary Material}. Again the combination of known quantities provides insight on the optimal relative tunings of the DA parameters compared with their non-DA optimal values. An equivalent  efficiency plot suggests that it is not worth implementing a DAPsMRWM algorithm if $\pi_a$ is only ten times faster to evaluate than $\pih$ is with $\sigma^2=1$.

As discussed in Section \ref{sect.lit.review}, an alternative tuning methodology 
relies on the property of the Particle-Marginal RWM algorithm that the optimal $\mu$ for a given $\sigma^2$, $\muhat(\sigma^2)$, is almost invariant to $\sigma^2$ \cite{SherlockThieryRobertsRosenthal:2013,Sherlock:2015}. This effectively reduces a two-dimensional optimisation 
problem to two one-dimensional problems. 
Contour plots in the \emph{Supplementary Material} of $\eff^{\textrm{rel}}_{\textrm{dapm}}$ as a function
of $\mu$ and $\sigma^2$ for specific combinations of
$\beta_2 \ge 0$, $\Abs{\beta_1}<\beta_2$ and $\eta>0$. 
all show a single mode and also show that for a particular variance, the optimal scaling 
$\muhat(\sigma)$ is insensitive to the value of $\sigma$, except 
when $\sigma\lesssim 1$, at which point the optimal scaling increases. 
Provided the noise variance is not made too small, therefore, 
$\mu$ and $\sigma$ may also be tuned independently for the DAPsMRWM.

\subsection{Tuning guidelines}
Theorem \ref{thrm.behave.eff} suggests that our goal of finding 
the optimal scaling $\muhat_{\textrm{da}}>0$ or, for the DAPsMRWM, $\muhat_{\textrm{dapm}}>0$ and $\sigmahat_{\textrm{dapm}}^2 > 0$, is sensible.
However, $\beta_1$, $\beta_2$ and $I$ arise from an idealisation of the form of the target distribution, and the dependence of quantities of interest on these parameters arises from a limiting argument as $d\rightarrow \infty$. In reality, the quantities $\beta_1$ and $\beta_2$ and $I$ might not exist. Even if they did exist, their values would not be known. This is why the tuning guidelines in Section \ref{sec.DARWMtune} use features that appear to be approximately independent of the specific values of $\beta_1$ and $\beta_2$, and for which $I$ is irrelevant. In addition to the two examples in Section \ref{sec.DARWMtune}, a wide-ranging simulation study in the \emph{Supplementary Material} provides further evidence for the appropriateness of these guidelines.

The \emph{Supplementary Material} also details tuning guidelines for the both the scaling and variance of the DAPsRWM, and tests these via an extensive simulation study on a discretely observed Lotka-Volterra model \cite{boys08}.

\section{Discussion}
\label{sec.Discussion}
We have analysed the delayed-acceptance
random walk Metropolis (DARWM) and delayed-acceptance
pseudo-marginal random walk Metropolis algorithm
(DAPsMRWM) in the limit as the dimension of the parameter space
tends to infinity. The theory leads to tuning guidelines which we have verified empirically across a wide variety of scenarios.

The theoretical work also supports the intuition that, provided the cheap
deterministic approximation is fast and reasonably accurate, the DAPsMRWM and DARWM algorithms
should be optimally efficient when $\mu$ is much larger than (and the overall acceptance rate is much
lower than) that of the equivalent
(pseudo-marginal) RWM algorithm.

In some DAPsMRWM scenarios, even the cheap approximation may only be evaluated approximately, \emph{e.g.}, via a particle filter. The expectation of this approximation would then be treated as we have treated $\pi_a$. If the noise in the logarithm of the cheap approximation is additive, then an alternative to Proposition \ref{prop.limi.accept.proba} can account for this via a term akin to $W_{\Delta}$. In the \emph{Supplementary Material} we describe this and investigate some consequences.

The surrogate transition method of \cite{liu2001monte} is a generalisation of delayed acceptance where, from the current position, $\bmx$, multiple sub-iterations of a  Metropolis-Hastings kernel targeting $\pi_a$, are made. The final position of this sub-chain, $\bmx^*$, is a proposal that is accepted  with a probability analogous to the Stage 2 acceptance probability in delayed-acceptance. For a complex target with a very cheap yet accurate approximation this offers the possibility of even greater efficiency, with a large number of sub-iterations leading to an approximation of an independence sampler proposing directly from a complex $\pi_a$. Our theory is based on a limiting diffusion obtained through Taylor expansion about the current point, $\bmx$, and so does not apply here, where $\| \bmx^*-\bmx \|$ can be large and the limiting process might not even be a diffusion; moreover, our theory also relies on independence between  the components of $\bmx^*-\bmx$, which does not hold after multiple sub-iterations. When $\pi_a$ is both cheap and accurate, intuition and our experimentation suggest that the optimal scaling for $\pi$ is also a sensible scaling for $\pi_a$, whilst the number of sub-iterations should be no more than is required for approximate convergence, and fewer if $\pi_a$ is relatively expensive to evaluate. In Example 2.1 an optimally scaled and iterated surrogate transition kernel was  $\approx 1.7$ times more efficient than an optimally scaled DARWM because a single jump proposal of the DARWM with its  larger optimal scaling  was  equivalent to approximately $30$ sub-iterations. Once $\pi_a$ is available, however, the surrogate transition method requires little extra coding effort, and, hence, could be worthwhile even for such relatively small gains in efficiency. Just as with PsMRWM, we found that the optimal scaling for the DAPsRWM was insensitive to the choice of the variance of the estimator, turning a two-dimensional optimisation into two one-dimensional optimisations. It is plausible that this insensitivity might hold for the surrogate transition method, offering one possible tuning simplification in the pseudo-marginal setting.

\section*{Acknowledgements}
AHT acknowledges support from the Singapore Ministry of Education Tier 2 (MOE2016-T2-2-135) and a
Young Investigator Award Grant (NUSYIA FY16 P16; R-155-000-180-133).

\bibliographystyle{plain}
\bibliography{../dapmrwm}

\begin{thebibliography}{10}

\bibitem{AndrieuDoucetHolenstein:2010}
C.~Andrieu, A.~Doucet, and R.~Holenstein.
\newblock Particle {M}arkov chain {M}onte {C}arlo methods.
\newblock {\em J. R. Stat. Soc. Ser. B Stat. Methodol.}, 72(3):269--342, 2010.

\bibitem{AndrieuRoberts:2009}
C.~Andrieu and G.~O. Roberts.
\newblock The pseudo-marginal approach for efficient {M}onte {C}arlo
  computations.
\newblock {\em Ann. Statist.}, 37(2):697--725, 2009.

\bibitem{AuBeck2001}
Siu-Kui Au and James~L. Beck.
\newblock Estimation of small failure probabilities in high dimensions by
  subset simulation.
\newblock {\em Probabilistic Engineering Mechanics}, 16(4):263 -- 277, 2001.

\bibitem{Banterle2015}
Marco Banterle, Clara Grazian, Anthony Lee, and Christian~P. Robert.
\newblock Accelerating {M}etropolis-{H}astings algorithms by delayed
  acceptance.
\newblock {\em Foundations of Data Science}, 1(2):103--128, 2019.

\bibitem{beaumont03}
M.~A. Beaumont.
\newblock Estimation of population growth or decline in genetically monitored
  populations.
\newblock {\em Genetics}, 164:1139--1160, 2003.

\bibitem{BedardA:2007}
M.~B{\'e}dard.
\newblock Weak convergence of {M}etropolis algorithms for non-i.i.d. target
  distributions.
\newblock {\em Ann. Appl. Probab.}, 17(4):1222--1244, 2007.

\bibitem{bedard2012scaling}
M.~B{\'e}dard, R.~Douc, and E.~Moulines.
\newblock Scaling analysis of multiple-try {MCMC} methods.
\newblock {\em Stoch. Proc. Appl.}, 122(3):758--786, 2012.

\bibitem{BedardRosenthal:2008}
Myl{\`e}ne B{\'e}dard and Jeffrey~S. Rosenthal.
\newblock Optimal scaling of {M}etropolis algorithms: heading toward general
  target distributions.
\newblock {\em Canad. J. Stat.}, 36:483--503, 2008.

\bibitem{berard2013lognormal}
J.~B{\'e}rard, P.~Del~Moral, and A.~Doucet.
\newblock A lognormal central limit theorem for particle approximations of
  normalizing constants.
\newblock {\em arXiv preprint arXiv:1307.0181}, 2013.

\bibitem{BeskosRobertsStuart:2009}
A.~Beskos, G.~O. Roberts, and A.~Stuart.
\newblock Optimal scalings for local {M}etropolis-{H}astings chains on
  nonproduct targets in high dimensions.
\newblock {\em Ann. Appl. Probab.}, 19(3):863--898, 2009.

\bibitem{boys08}
R.~J. Boys, D.~J. Wilkinson, and T.~B.~L. Kirkwood.
\newblock Bayesian inference for a discretely observed stochastic-kinetic
  model.
\newblock {\em Stat. Comput.}, 18:125--135, 2008.

\bibitem{breyer2004optimal}
L.~A. Breyer, M.~Piccioni, and S.~Scarlatti.
\newblock Optimal scaling of {MALA} for nonlinear regression.
\newblock {\em Ann. Appl. Probab.}, 14(3):1479--1505, 2004.

\bibitem{MCMCHandbook}
S.~Brooks, A.~Gelman, G.~L. Jones, and X.-L. Meng, editors.
\newblock {\em Handbook of {M}arkov chain {M}onte {C}arlo}.
\newblock Chapman \& Hall/CRC Handbooks of Modern Statistical Methods. CRC
  Press, Boca Raton, FL, 2011.

\bibitem{CatanachBeck2018}
Thomas~A. {Catanach} and James~L. {Beck}.
\newblock {Bayesian Updating and Uncertainty Quantification using Sequential
  Tempered MCMC with the Rank-One Modified Metropolis Algorithm}.
\newblock {\em arXiv e-prints}, page arXiv:1804.08738, Apr 2018.

\bibitem{ChristenFox:2005}
J.~A. Christen and C.~Fox.
\newblock {M}arkov chain {M}onte {C}arlo using an approximation.
\newblock {\em J. Comp. Graph. Stat.}, 14(4):795--810, 2005.

\bibitem{cui2011bayesian}
T~Cui, C~Fox, and MJ~O'Sullivan.
\newblock Bayesian calibration of a large-scale geothermal reservoir model by a
  new adaptive delayed acceptance metropolis hastings algorithm.
\newblock {\em Water Resources Research}, 47(10), 2011.

\bibitem{DahlinSchon2019}
Johan {Dahlin} and Thomas~B. {Sch{\"o}n}.
\newblock {Getting Started with Particle Metropolis-Hastings for Inference in
  Nonlinear Dynamical Models}.
\newblock {\em arXiv e-prints}, page arXiv:1511.01707, Nov 2019.

\bibitem{DelMoral:2004}
P.~Del~Moral.
\newblock {\em Feynman-{K}ac formulae}.
\newblock Probability and its Applications (New York). Springer-Verlag, New
  York, 2004.
\newblock Genealogical and interacting particle systems with applications.

\bibitem{Doucetetal:2013}
A.~Doucet, M.~K. Pitt, G.~Deligiannidis, and R.~Kohn.
\newblock {Efficient implementation of Markov chain Monte Carlo when using an
  unbiased likelihood estimator}.
\newblock {\em Biometrika}, 102(2):295--313, 03 2015.

\bibitem{efendiev2005efficient}
Y~Efendiev, A~Datta-Gupta, V~Ginting, X~Ma, and B~Mallick.
\newblock An efficient two-stage {M}arkov chain {M}onte {C}arlo method for
  dynamic data integration.
\newblock {\em Water Resources Research}, 41(12), 2005.

\bibitem{EffendievHouLuo2006}
Y.~Efendiev, T.~Hou, and W.~Luo.
\newblock Preconditioning {M}arkov chain {M}onte {C}arlo simulations using
  coarse-scale models.
\newblock {\em SIAM Journal on Scientific Computing}, 28(2):776--803, 2006.

\bibitem{ethier1986markov}
S.~N. Ethier and T.~G. Kurtz.
\newblock {\em Markov processes: Characterization and convergence}, volume~6.
\newblock Wiley New York, 1986.

\bibitem{EvRow2017}
Richard~G. {Everitt} and Paulina~A. {Rowi{\'n}ska}.
\newblock {Delayed acceptance ABC-SMC}.
\newblock {\em arXiv e-prints}, page arXiv:1708.02230, Aug 2017.

\bibitem{fearnhead12}
P.~Fearnhead, V.~Giagos, and C.~Sherlock.
\newblock Inference for reaction networks using the {L}inear {N}oise
  {A}pproximation.
\newblock {\em Biometrics}, 70:457--466, 2014.

\bibitem{FearnheadSherlock2006}
Paul Fearnhead and Chris Sherlock.
\newblock An exact {G}ibbs sampler for the {M}arkov-modulated {P}oisson
  process.
\newblock {\em J. R. Stat. Soc. Ser. B Stat. Methodol.}, 68(5):767--784, 2006.

\bibitem{FilipponeGirolami:2014}
Maurizio Filippone and Mark Girolami.
\newblock Pseudo-marginal {B}ayesian inference for {G}aussian processes.
\newblock {\em IEEE Tran. Pattern Anal. Mach. Intell.}, 36(11):2214--2226,
  2014.

\bibitem{flury2011bayesian}
T.~Flury and N.~Shephard.
\newblock Bayesian inference based only on simulated likelihood: particle
  filter analysis of dynamic economic models.
\newblock {\em Econometric Theory}, 27(05):933--956, 2011.

\bibitem{FranksVihola2017}
Jordan Franks and Matti Vihola.
\newblock Importance sampling correction versus standard averages of reversible
  {MCMC}s in terms of the asymptotic variance.
\newblock {\em Stochastic Processes and their Applications}, 2020.

\bibitem{Gilks/Richardson/Spiegelhalter:1996}
W.~R. Gilks, S.~Richardson, and D.~J. Spiegelhalter.
\newblock {\em {M}arkov Chain {M}onte {C}arlo in practice}.
\newblock Chapman and Hall, London, UK, 1996.

\bibitem{Gillespie77}
D.~T. Gillespie.
\newblock Exact stochastic simulation of coupled chemical reactions.
\newblock {\em J. Phys. Chem.}, 81:2340--2361, 1977.

\bibitem{GolightlyWilkinson:2011}
A.~Golightly and D.~J. Wilkinson.
\newblock Bayesian parameter inference for stochastic biochemical network
  models using particle {M}arkov chain {M}onte {C}arlo.
\newblock {\em Interface Focus}, 1(6):807--820, 2011.

\bibitem{GolightlyHendersonSherlock:2013}
Andrew Golightly, Daniel~A. Henderson, and Chris Sherlock.
\newblock Delayed acceptance particle {MCMC} for exact inference in stochastic
  kinetic models.
\newblock {\em Statistics and Computing}, 25(5):1039--1055, Sep 2015.

\bibitem{MCMCHandbook16}
D.~C. Higdon, S.~J. Reese, D.~Moulton, J.~A. Vrugt, and C.~Fox.
\newblock Posterior exploration for computationally intensive forward models.
\newblock In S.~Brooks, A.~Gelman, G.~L. Jones, and X.-L. Meng, editors, {\em
  Handbook of {M}arkov chain {M}onte {C}arlo}, chapter~16, pages 401--418. CRC
  Press, Boca Raton, FL, 2011.

\bibitem{kaipio2006statistical}
Jari Kaipio and Erkki Somersalo.
\newblock {\em Statistical and computational inverse problems}, volume 160.
\newblock Springer Science \& Business Media, 2006.

\bibitem{Knape/deValpine:2012}
J.~Knape and P.~de~Valpine.
\newblock Fitting complex population models by combining particle filters with
  {M}arkov chain {M}onte {C}arlo.
\newblock {\em Ecology}, 93(2):256--263, 2012.

\bibitem{liu2001monte}
J.~S. Liu.
\newblock {\em Monte {C}arlo Strategies In Scientific Computing}.
\newblock Springer, 2001.

\bibitem{mattingly2012diffusion}
J.~C. Mattingly, N.~S. Pillai, and A.~M. Stuart.
\newblock Diffusion limits of the random walk {M}etropolis algorithm in high
  dimensions.
\newblock {\em Ann. Appl. Probab.}, 22(3):881--930, 2012.

\bibitem{Moulton/etal:2008}
J.~D. Moulton, C.~Fox, and D.~Svyatskiy.
\newblock Multilevel approximations in sample-based inversion from the
  {D}irichlet-to-{N}eumann map.
\newblock {\em J. Phys.: Conf. Ser.}, 124(1), 2008.

\bibitem{pasarica2010adaptively}
C.~Pasarica and A.~Gelman.
\newblock Adaptively scaling the {M}etropolis algorithm using expected squared
  jumped distance.
\newblock {\em Statistica Sinica}, 20(1):343, 2010.

\bibitem{PST12}
N.~S. Pillai, A.~M. Stuart, and A.~H. Thiery.
\newblock Optimal scaling and diffusion limits for the {L}angevin algorithm in
  high dimensions.
\newblock {\em Ann. Appl. Probab.}, 22(6):2320--2356, 2012.

\bibitem{Pittetal:2012}
M.~K. Pitt, R.~dos Santos~Silva, P.~Giordani, and R.~Kohn.
\newblock On some properties of {M}arkov chain {M}onte {C}arlo simulation
  methods based on the particle filter.
\newblock {\em Journal of Econometrics}, 171(2):134 -- 151, 2012.

\bibitem{CODA}
Martyn Plummer, Nicky Best, Kate Cowles, and Karen Vines.
\newblock Coda: Convergence diagnosis and output analysis for mcmc.
\newblock {\em R News}, 6(1):7--11, 2006.

\bibitem{Quiroz2018}
Matias Quiroz, Minh-Ngoc Tran, Mattias Villani, and Robert Kohn.
\newblock Speeding up {MCMC} by delayed acceptance and data subsampling.
\newblock {\em Journal of Computational and Graphical Statistics},
  27(1):12--22, 2018.

\bibitem{Roberts/Gelman/Gilks:1997}
G.~O. Roberts, A.~Gelman, and W.~R. Gilks.
\newblock Weak convergence and optimal scaling of random walk {M}etropolis
  algorithms.
\newblock {\em Ann. Appl. Probab.}, 7:110--120, 1997.

\bibitem{RobertsRosenthal:1998}
G.~O. Roberts and J.~S. Rosenthal.
\newblock Optimal scaling of discrete approximations to {L}angevin diffusions.
\newblock {\em J. R. Stat. Soc. Ser. B Stat. Methodol.}, 60(1):255--268, 1998.

\bibitem{Roberts/Rosenthal:2001}
G.~O. Roberts and J.~S. Rosenthal.
\newblock Optimal scaling for various {M}etropolis-{H}astings algorithms.
\newblock {\em Statistical Science}, 16:351--367, 2001.

\bibitem{roberts2014complexity}
G.~O. Roberts and J.~S. Rosenthal.
\newblock Complexity bounds for {MCMC} via diffusion limits.
\newblock {\em arXiv preprint arXiv:1411.0712}, 2014.

\bibitem{roberts2014minimising}
G.~O. Roberts and J.~S. Rosenthal.
\newblock Minimising {MCMC} variance via diffusion limits, with an application
  to simulated tempering.
\newblock {\em Ann. Appl. Probab.}, 24(1):131--149, 2014.

\bibitem{SDDP2018}
Sebastian~M. {Schmon}, George {Deligiannidis}, Arnaud {Doucet}, and Michael~K.
  {Pitt}.
\newblock {Large Sample Asymptotics of the Pseudo-Marginal Method}.
\newblock {\em arXiv e-prints}, page arXiv:1806.10060, Jun 2018.

\bibitem{Sherlock:2013}
C.~Sherlock.
\newblock Optimal scaling of the random walk {M}etropolis: general criteria for
  the $0.234$ acceptance rule.
\newblock {\em J. App. Prob.}, 50(1):1--15, 2013.

\bibitem{Sherlock/Fearnhead/Roberts:2010}
C.~Sherlock, P.~Fearnhead, and G.~O. Roberts.
\newblock The random walk {M}etropolis: linking theory and practice through a
  case study.
\newblock {\em Statist. Sci.}, 25(2):172--190, 2010.

\bibitem{Sherlock/Roberts:2009}
C.~Sherlock and G.~O. Roberts.
\newblock Optimal scaling of the random walk {M}etropolis on elliptically
  symmetric unimodal targets.
\newblock {\em Bernoulli}, 15(3):774--798, 2009.

\bibitem{Sherlock:2015}
Chris Sherlock.
\newblock Optimal scaling for the pseudo-marginal random walk metropolis:
  Insensitivity to the noise generating mechanism.
\newblock {\em Methodology and Computing in Applied Probability},
  18(3):869--884, Sep 2016.

\bibitem{SGH2017}
Chris Sherlock, Andrew Golightly, and Daniel~A. Henderson.
\newblock Adaptive, delayed-acceptance {MCMC} for targets with expensive
  likelihoods.
\newblock {\em Journal of Computational and Graphical Statistics},
  26(2):434--444, 2017.

\bibitem{SherlockLee2017}
Chris {Sherlock} and Anthony {Lee}.
\newblock {Variance bounding of delayed-acceptance kernels}.
\newblock {\em arXiv e-prints}, page arXiv:1706.02142, Jun 2017.

\bibitem{sherlock2017pseudo}
Chris Sherlock, Alexandre~H Thiery, and Anthony Lee.
\newblock Pseudo-marginal {M}etropolis--{H}astings sampling using averages of
  unbiased estimators.
\newblock {\em Biometrika}, 104(3):727--734, 2017.

\bibitem{SherlockThieryRobertsRosenthal:2013}
Chris Sherlock, Alexandre~H. Thiery, Gareth~O. Roberts, and Jeffrey~S.
  Rosenthal.
\newblock On the efficiency of pseudo-marginal random walk {M}etropolis
  algorithms.
\newblock {\em Ann. Statist.}, 43(1):238--275, 02 2015.

\bibitem{Smith:2011}
M.~E. Smith.
\newblock Estimating nonlinear economic models using surrogate transitions.
\newblock Available from
  \texttt{https://files.nyu.edu/mes473/public/Smith\_Surrogate.pdf}, 2011.

\bibitem{stuart2010inverse}
Andrew~M Stuart.
\newblock Inverse problems: a {B}ayesian perspective.
\newblock {\em Acta Numerica}, 19:451--559, 2010.

\bibitem{kampen2001}
N.~G. van Kampen.
\newblock {\em Stochastic Processes in Physics and Chemistry}.
\newblock North-Holland, 2001.

\bibitem{ViholaFranks2016}
Matti {Vihola}, Jouni {Helske}, and Jordan {Franks}.
\newblock {Importance sampling type estimators based on approximate marginal
  MCMC}.
\newblock {\em arXiv e-prints}, page arXiv:1609.02541, Sep 2016.

\bibitem{yang2019optimal}
Jun Yang, Gareth~O Roberts, and Jeffrey~S Rosenthal.
\newblock Optimal scaling of metropolis algorithms on general target
  distributions.
\newblock {\em arXiv preprint arXiv:1904.12157}, 2019.

\bibitem{zanella2017dirichlet}
Giacomo Zanella, Myl{\`e}ne B{\'e}dard, and Wilfrid~S Kendall.
\newblock A {D}irichlet form approach to {MCMC} optimal scaling.
\newblock {\em Stochastic Processes and their Applications},
  127(12):4053--4082, 2017.

\end{thebibliography}

\appendix

\section*{Supplementary Material}
This document contains the following material. Section \ref{sec.example.targets} provides further details of the example targets used in Section \ref{sec.DARWMtune}. Section \ref{sec.explicit.alphas} provides explicit expressions for the theoretical acceptance probabilities derived in Section \ref{sec.limiting.acc.proba}. Section \ref{app.additional.contours} provides further plots derived from our theory, including enlargements of the plots required for tuning. Section \ref{sec.valid.both} provides numerical verification of Lemma \ref{lem.pivotal} in a DARWM (\ref{sec.valid.ODE}) and DAPsMRWM (\ref{sec.valid.lotka}) setting. Section \ref{sec.proof} provides proofs of important results, as well as of the diffusion limit. Section \ref{sec.technical} gives the proofs of several technical results, whilst Section \ref{app.pmstageOne} describes key results parallel to those in Section \ref{sec.asymp.analysis} but for the case when the error at Stage One is noisy and biased. The method typically recommends a range of tuning parameter values rather than one specific value; Section \ref{app.envelopeWidth} contains an investigation of how to choose from the values in the envelope. Section \ref{app.DARWMbetas} details a broad simulation study, further validating the tuning advice for the DARWM. Section \ref{sec.DAPsRWMtuneandsim} describes the tuning strategies for the DAPsRWM and details a simulation study based on noisy observations from the Lotka-Volterra model where the  Linear Noise Approximation is used as $\pi_a$.


\section{Example targets}
\label{sec.example.targets}
In this section we provide further details on the target distributions used in Section \ref{sec.DARWMtune}.\\

\noindent
{\bf The Markov modulated Poisson process:} consider a $k$-state, continuous-time Markov chain $Z_t$ started from state $1$, and a Poisson process $N_t$ whose rate $\lambda_t$ is a fixed function of $Z_t$. The doubly-stochastic process is parameterised by the rate matrix for the Markov chain, $Q$, and a vector of rates for the Poisson process, $\lambda$, where $\lambda_i,~(i=1,\dots,k)$ is the rate of $N_t$ when $Z_t=i$.

The event times of $N_t$ are observed over a time window
$[0,t_{\mathrm{end}}]$, but the behaviour of $Z_t$ is unknown, and we
wish to perform inference on $(Q,\lambda)$. Setting
$\Lambda=\mbox{diag}(\lambda)$, the likelihood for the number of
events $n$ and the event times $t_1,\dots,t_n$ is  \cite[e.g.][]{FearnheadSherlock2006}:
  \[
  L(Q,\lambda;t)
  =
  e' \exp[(Q-\Lambda)t_1]\Lambda \exp[(Q-\Lambda)(t_2-t_1)]\Lambda \dots
  \Lambda\exp[(Q-\Lambda)(t_{\mathrm{end}}-t_n)] 1,
  \]
  where $1$ is the $k$-vector of ones and $e'=(1,0,\dots,0)$.
  We simulated a dataset using a cyclic four-state Markov chain for a
  $200$-second time window with $Q$ parameters of:
  $Q_{12}=Q_{23}=Q_{41}=1.0,~Q_{34}=0.25$ and all other
  off-diagonal rates set to zero. The rate parameters were
  $\lambda_1=20.0, \lambda_2=5.0,~\lambda_3=1.0$ and $\lambda_4=10.0$. We
  then conducted inference on the natural logarithm of each parameter
  that was not systematically zero,
  placing independent $N(0,2^2)$ priors on each of these. \\

\noindent
    {\bf The ODE model:}
    We set $x_0=(1,\ldots,1)\in \RR^5$ and assume independent centred Gaussian priors with standard deviations of $\sigma_0 = 10$ on the upper-triangular part of $A \in \RR^{5,5}$.
\section{Explicit expressions for the acceptance probabilities}
\label{sec.explicit.alphas}
Define $G(a,b):=\Expect{1\wedge \exp(\Normal{a,b^2})}=\Phi(a / b) + \exp \left(a + b^2/2 \right) \, \Phi(-b - a/b)$ with $\Phi:\RR \to [0,1]$ the standard Gaussian cumulative
distribution function. 
Then
\begin{align}
\label{eq.acc.1.gaussb}
\alpha_1(\mu, \sigma;& \beta_1,\beta_2) 
= G\BK{-\frac{\mu^2}{2}(1-\beta_1), \, \mu^2 \, \BK{1+\beta^2_2 - 2 \beta_1}}.
\end{align}
Further, we may rewrite
\begin{align*}
Q^\infty_\Delta =-\frac{1}{2}\mu^2+\mu \frac{\beta_1}{\beta_2}\xi +\Normal{0,\mu^2-\mu^2\frac{\beta_1^2}{\beta_2^2}}
\qquad \textrm{and} \qquad
S^\infty_\Delta =\frac{\beta_1}{2}\mu^2-\mu\beta_2\xi,
\end{align*}
where $\xi\sim\Normal{0,1}$ is independent of any other source of variability. Thus, the quantity $\alpha_{12}(\mu,\sigma^2;\beta_1,\beta_2)$ can also be expressed as
\begin{align}
\label{eq.acc.2.gaussb}
\Expect{G\left(-\frac{\mu^2}{2}(1-\beta_1)
+
\mu\left(\frac{\beta_1}{\beta_2}-\beta_2\right)\xi,\mu^2-\mu^2\frac{\beta_1^2}{\beta_2^2}\right)
G\left(-\frac{\beta_1}{2}\mu^2-\sigma^2+\mu\beta_2\xi,2\sigma^2\right)}.
\end{align}

\section{Further plots derived from the theory}
\label{app.additional.contours}
Figure \ref{fig.effPMRWM.betas} shows contour plots of
$\alpha_{2|1}(\muhat_{\textrm{pm}},\sigmahat_{\textrm{pm}}^2)$, $\muhat_{\textrm{dapm}}/\muhat_{\textrm{pm}}$, $\sigmahat^2_{\textrm{dapm}}/\sigmahat^2_{\textrm{pm}}$ and $\eff_{\textrm{rel}}(\muhat_{\textrm{dapm}},\sigmahat^2_{\textrm{dapm}})$, as a
function of $\beta_1$ and $\beta_2$ for $\eta=0.01$. The contours all have a very similar shape, which suggests that $\alpha_{2|1}(\muhat_{\textrm{pm}},\sigmahat_{\textrm{pm}}^2)$, $\muhat_{\textrm{dapm}}/\muhat_{\textrm{pm}}$ together with $\eta$ should provide information about the other three quantities.

Figure \ref{fig.musigma2scatterpm.app} shows an enlarged scatter plot of $\lambdahat_{\textrm{da}}/\lambdahat_{\textrm{rwm}}$ as a function of $\alpha_{2|1}(\lambdahat_{\textrm{rwm}})/\alpha_{\textrm{rwm}}(\lambdahat_{\textrm{rwm}})$. It also shows   
 $\muhat_{\textrm{dapm}}/\muhat_{\textrm{pm}}$,  $\sigmahat^2_{\textrm{dapm}}/\sigmahat^2_{\textrm{pm}}$ and $\eff^{\textrm{rel}}_{\textrm{dapm}}(\muhat_{\textrm{dapm}},\sigmahat^2_{\textrm{dapm}})$, all vs $\alpha_{2|1}(\muhat_{\textrm{pm}},\sigmahat^2_{\textrm{pm}})$ , with all plots partitioned by $\eta$.
\begin{figure}
\begin{center}
\subfigure{
  \includegraphics[width=0.4 \textwidth,angle=0]{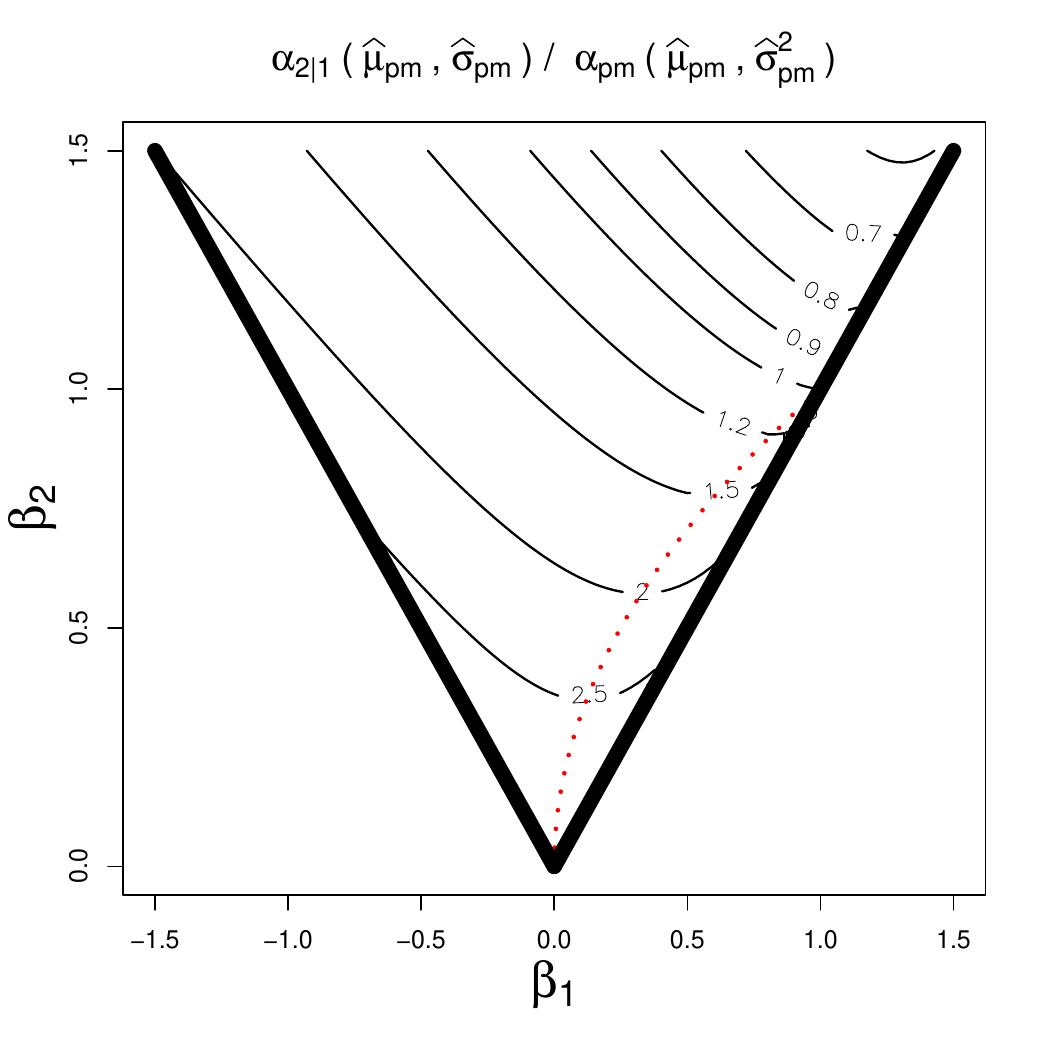}
}
\subfigure{
  \includegraphics[width=0.4 \textwidth,angle=0]{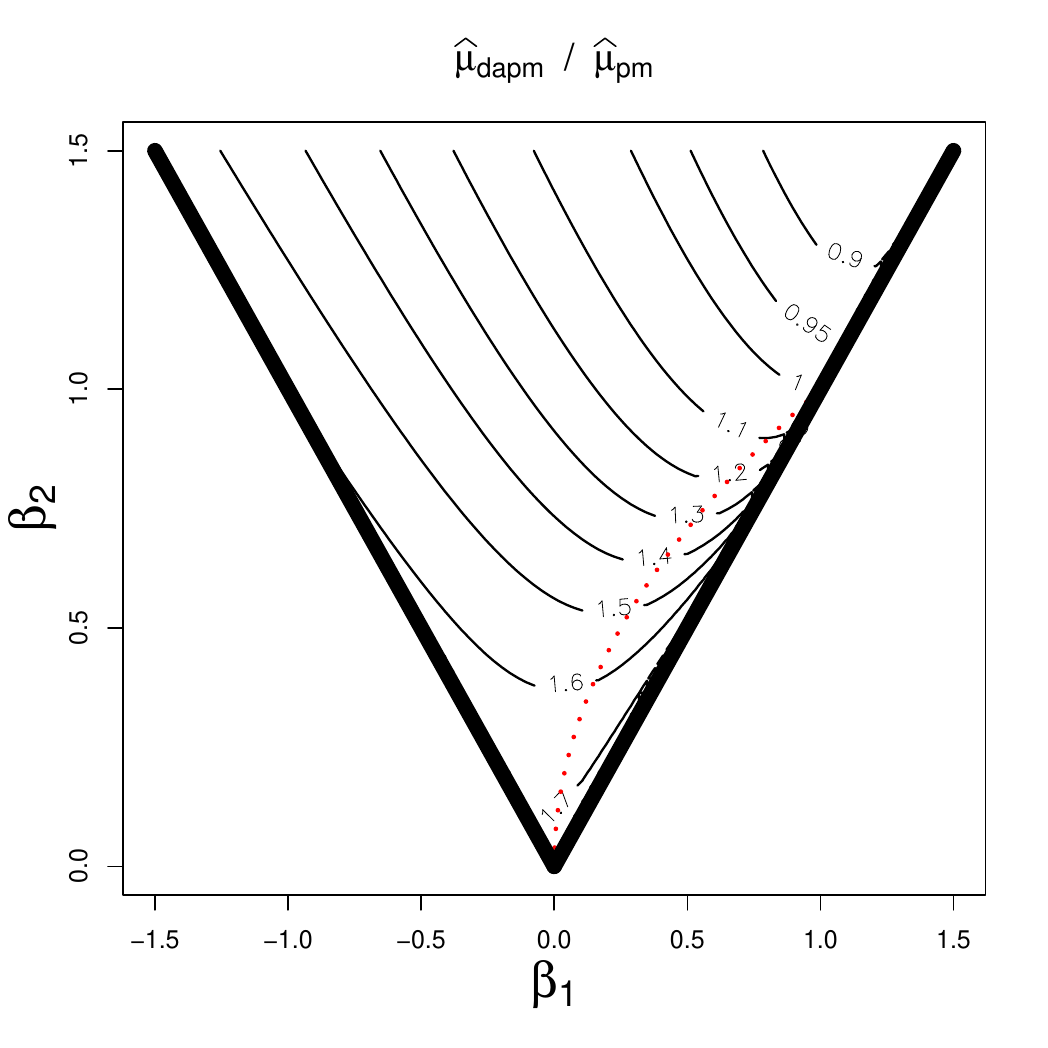}
}\\
\subfigure{
  \includegraphics[width=0.4 \textwidth,angle=0]{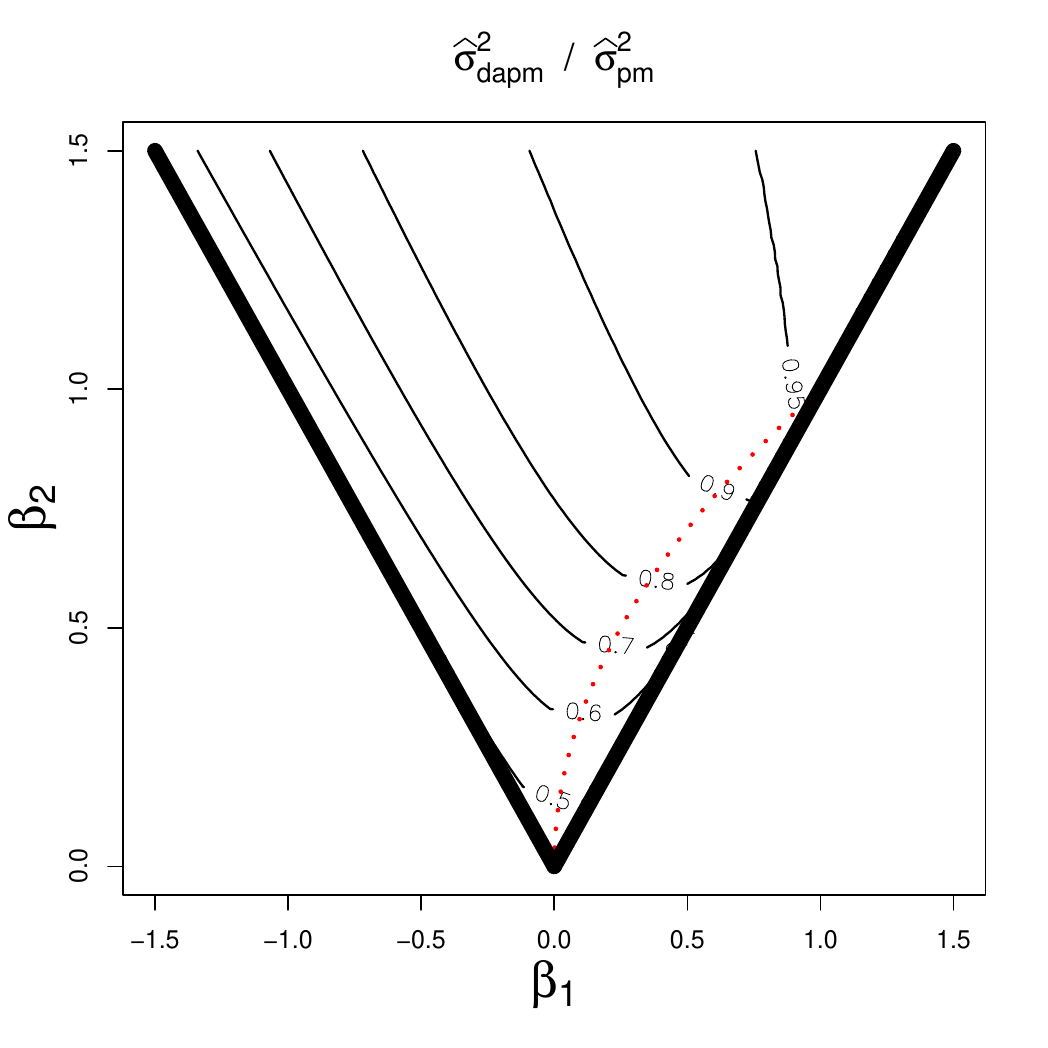}
}
\subfigure{
  \includegraphics[width=0.4 \textwidth,angle=0]{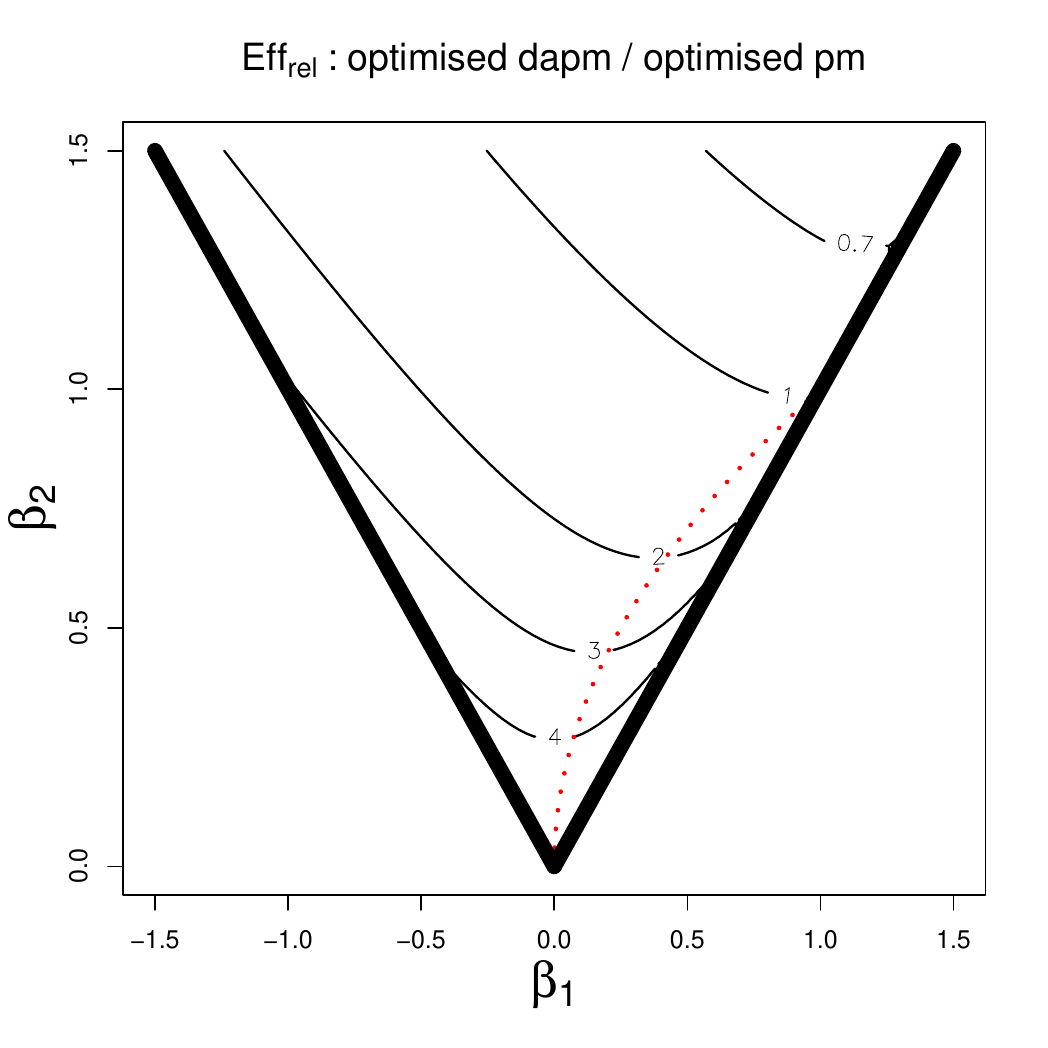}
}
\caption{
Contour plots of
$\alpha_{2|1}(\muhat_{\textrm{pm}},\sigmahat^2_{\textrm{pm}};\beta_1,\beta_2)$ (left),  $\muhat_{\textrm{dapm}}/\muhat_{\textrm{pm}}$, $\sigmahat^2_{\textrm{dapm}}/\sigmahat^2_{\textrm{pm}}$ and $\eff_{\textrm{rel}}(\muhat_{\textrm{dapm}},\sigmahat^2_{\textrm{dapm}})$ (right), as a
function of $\beta_1$ and $\beta_2$ for $\eta=0.01$. The red, dotted line satisfies $\beta_1=\beta_2^2$.
\label{fig.effPMRWM.betas}
}
\end{center}
\end{figure}
\begin{figure}
\begin{center}
\subfigure{
  \includegraphics[width=0.4 \textwidth,angle=0]{DARWM_muvsalpha_xtra.pdf}
  }
\subfigure{
  \includegraphics[width=0.4 \textwidth,angle=0]{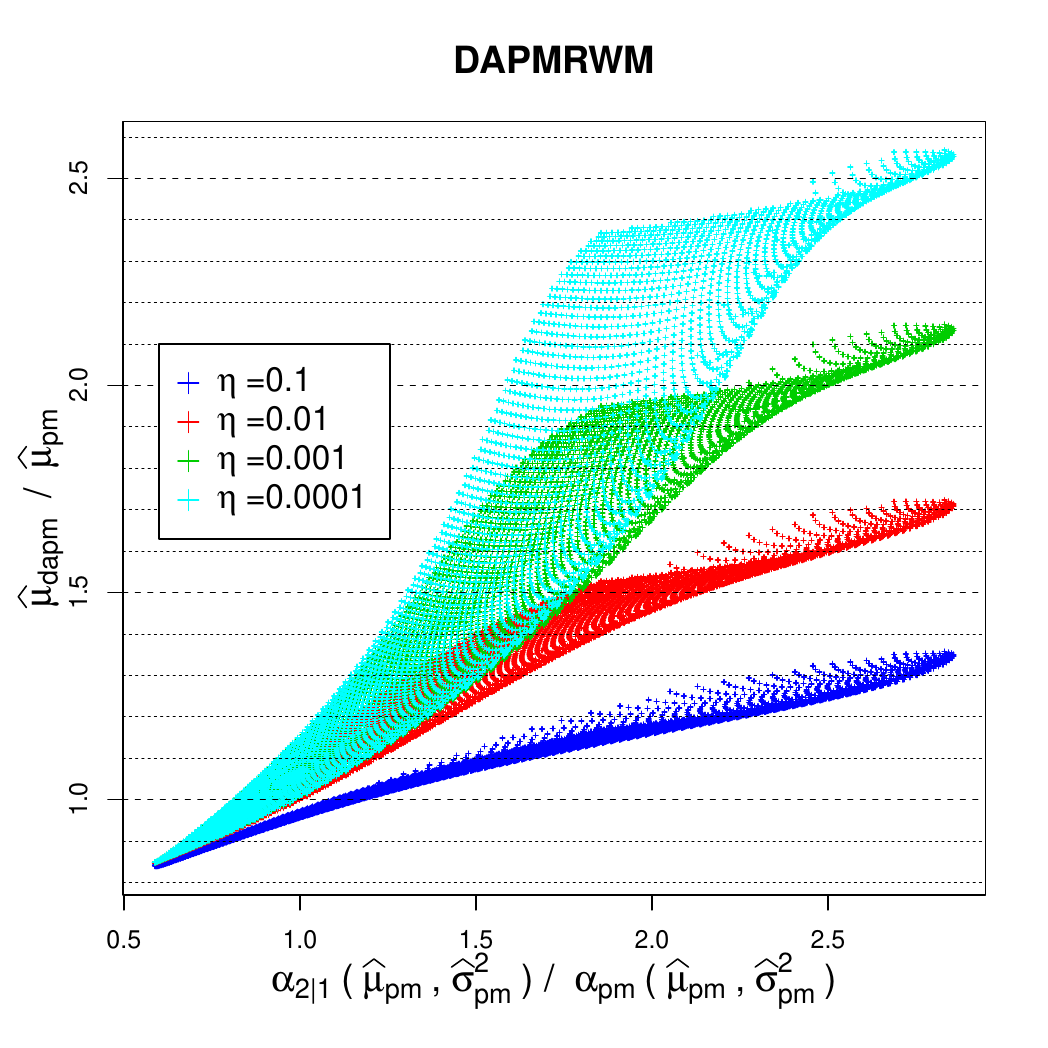}
}
\subfigure{
  \includegraphics[width=0.4 \textwidth,angle=0]{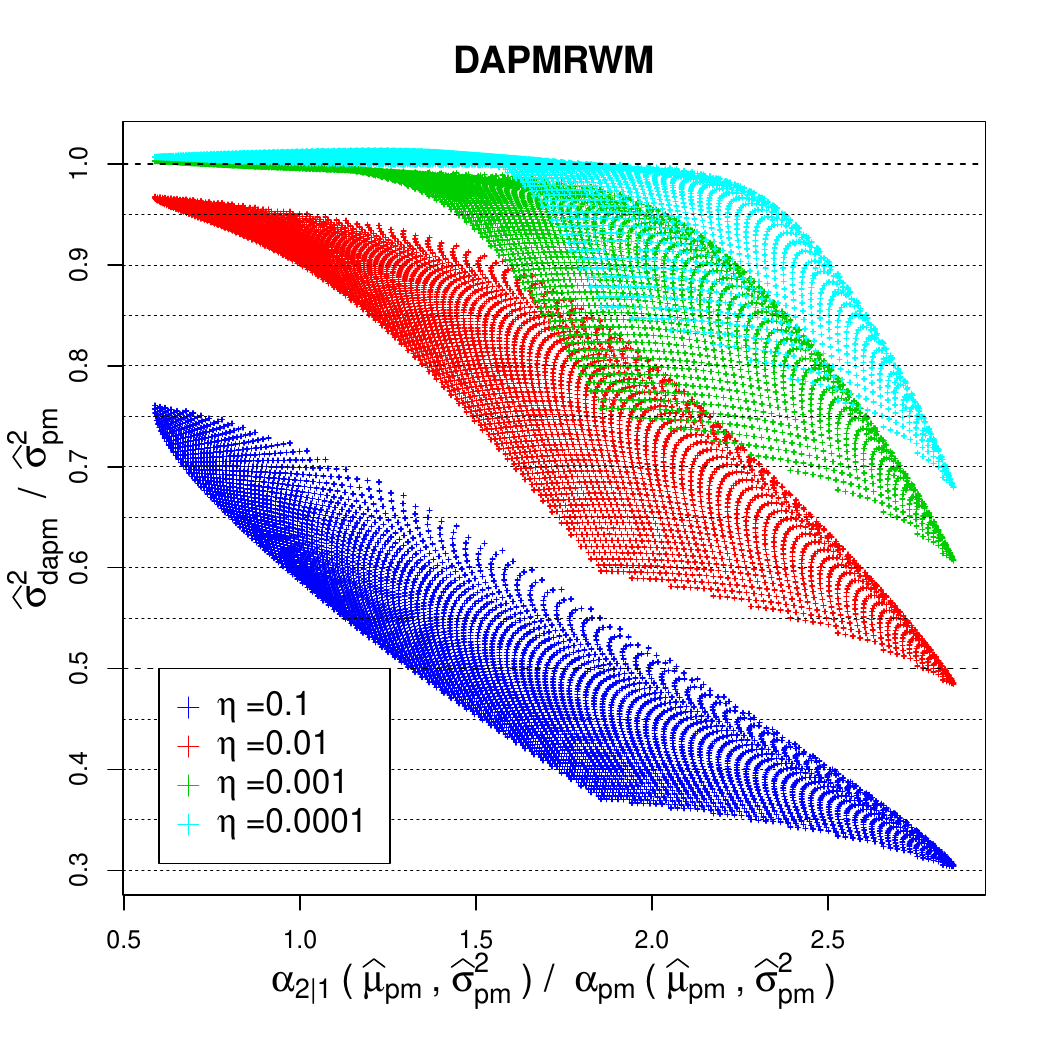}
}
\subfigure{
  \includegraphics[width=0.4 \textwidth,angle=0]{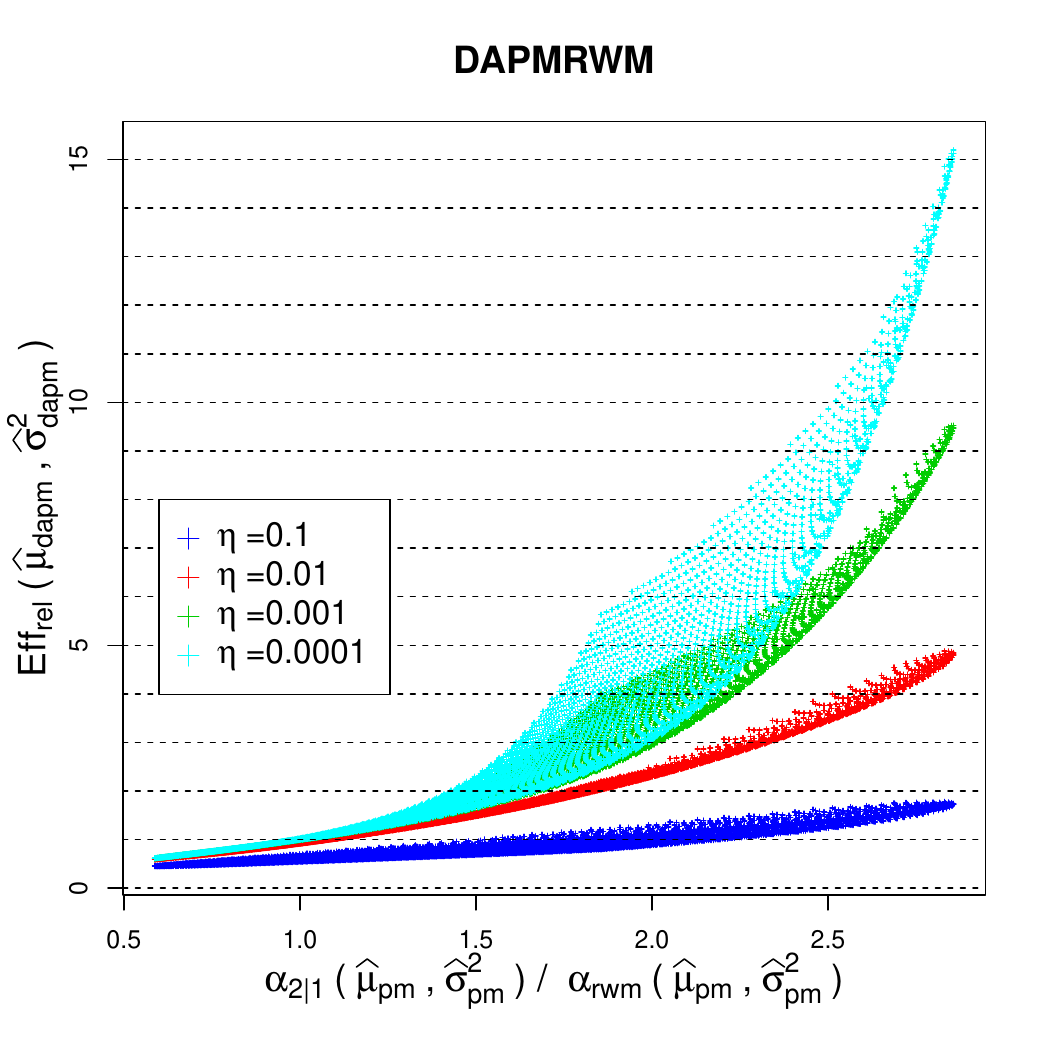}
}
\caption{
Scatter plots of
 $\lambdahat_{\textrm{da}}/\lambdahat_{\textrm{rwm}}$ as a function of $\alpha_{2|1}(\lambdahat_{\textrm{rwm}})/\alpha_{\textrm{rwm}}(\lambdahat_{\textrm{rwm}})$ for different values of $\eta$ (top left), then $\muhat_{\textrm{dapm}}/\muhat_{\textrm{pm}}$ (top right),  $\sigmahat^2_{\textrm{dapm}}/\sigmahat^2_{\textrm{pm}}$ (bottom left) and $\eff^{\textrm{rel}}_{\textrm{dapm}}(\muhat_{\textrm{dapm}},\sigmahat^2_{\textrm{dapm}})$ (bottom right), vs $\alpha_{2|1}(\muhat_{\textrm{pm}},\sigmahat^2_{\textrm{pm}})$ (right), partitioned by $\eta$.
\label{fig.musigma2scatterpm.app} 
}
\end{center}
\end{figure}

Figure \ref{fig.eff.mu.sigma2}, which is typical of many other 
such figures that we produced, shows contour plots of $\eff^{\textrm{rel}}_{\textrm{dapm}}$ as a function
of $\mu$ and $\sigma^2$ for specific combinations of
$\beta_2 \ge 0$, $\Abs{\beta_1}<\beta_2$ and $\eta>0$. 
Each plot shows a single mode and also shows that for a particular variance, the optimal scaling 
$\muhat(\sigma)$ is insensitive to the value of $\sigma$, except 
when $\sigma\lesssim 1$, at which point the optimal scaling increases. 
Provided the noise variance is not made too small, therefore, 
$\mu$ and $\sigma$ may also be tuned independently for the DAPsMRWM.

\begin{figure}
\begin{center}
\subfigure{
  \includegraphics[width=0.3\textwidth,angle=0]{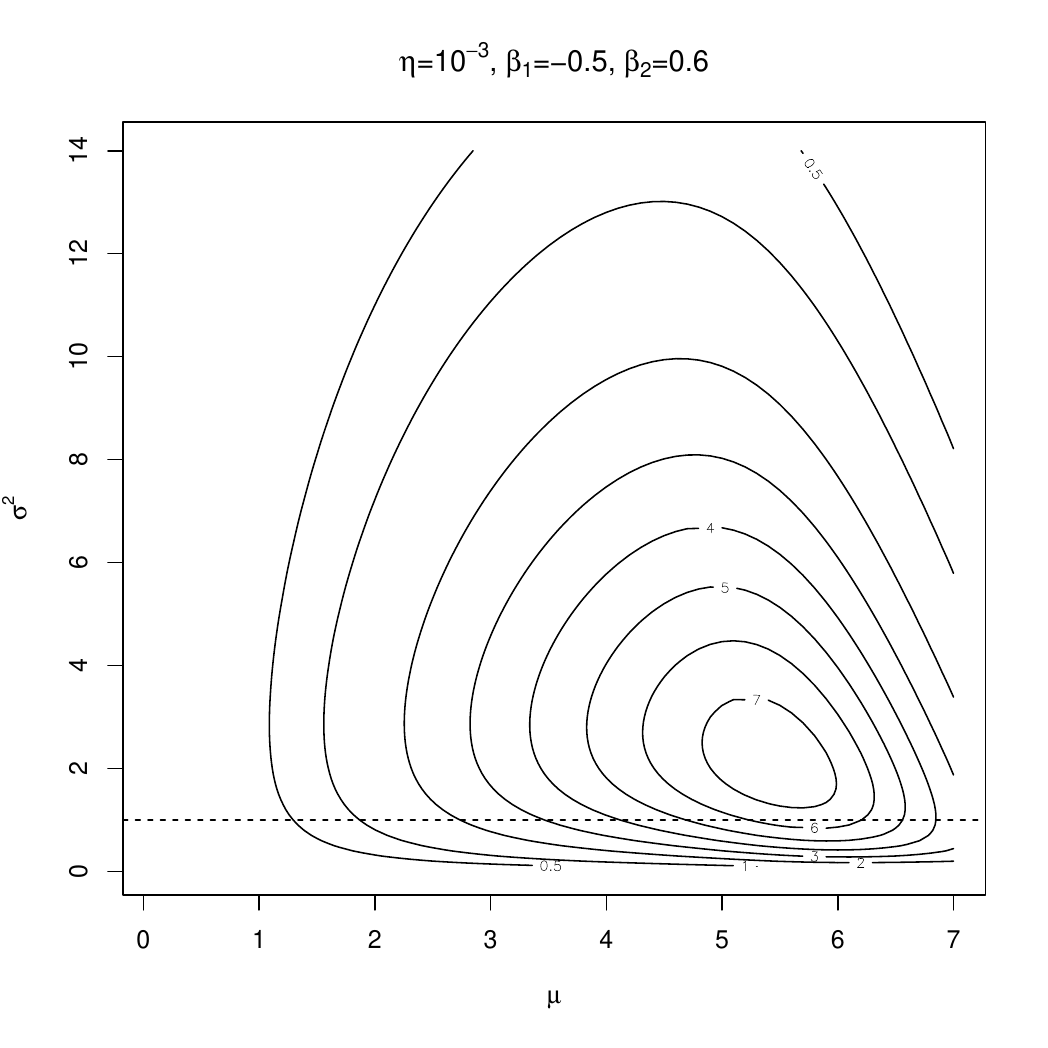}
}
\subfigure{
  \includegraphics[width=0.3\textwidth,angle=0]{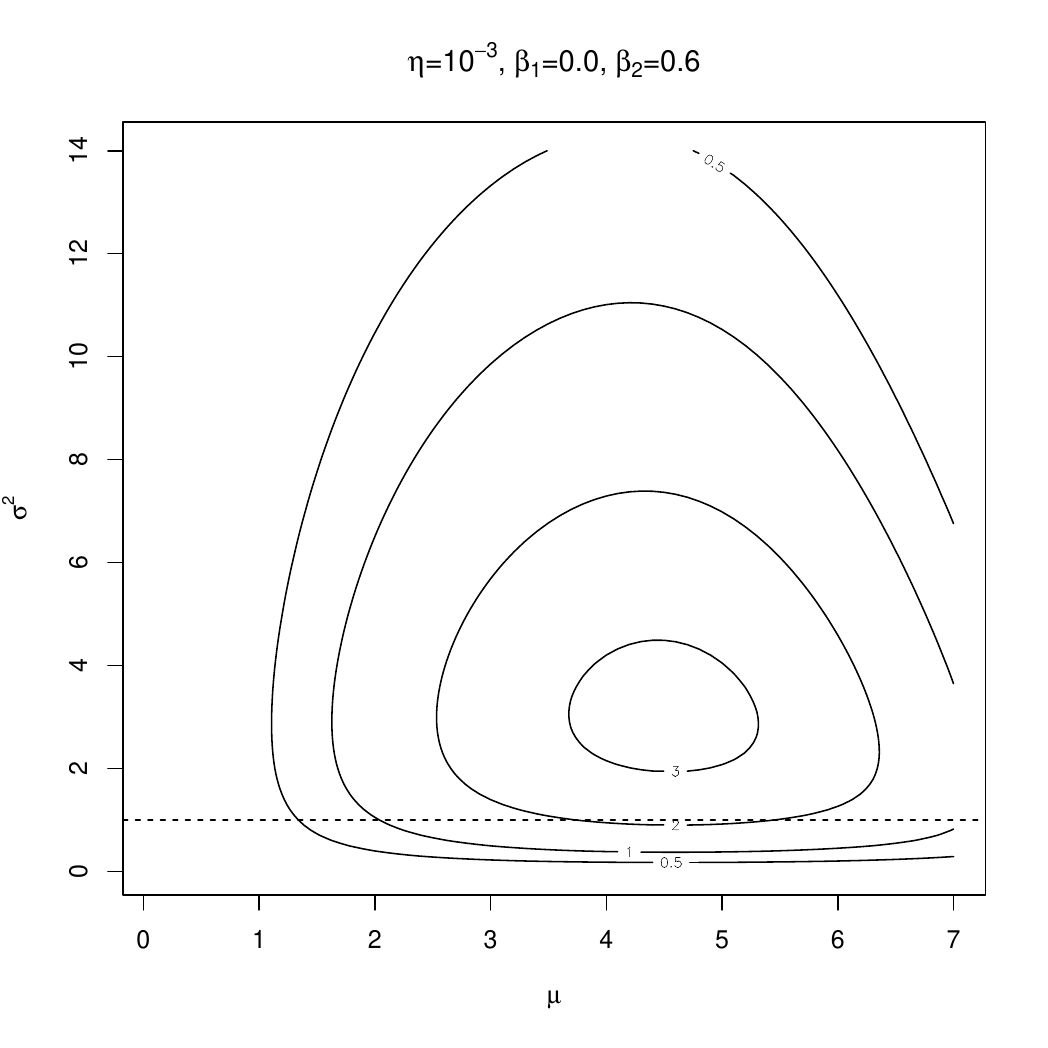}
}
\subfigure{
 \includegraphics[width=0.3\textwidth,angle=0]{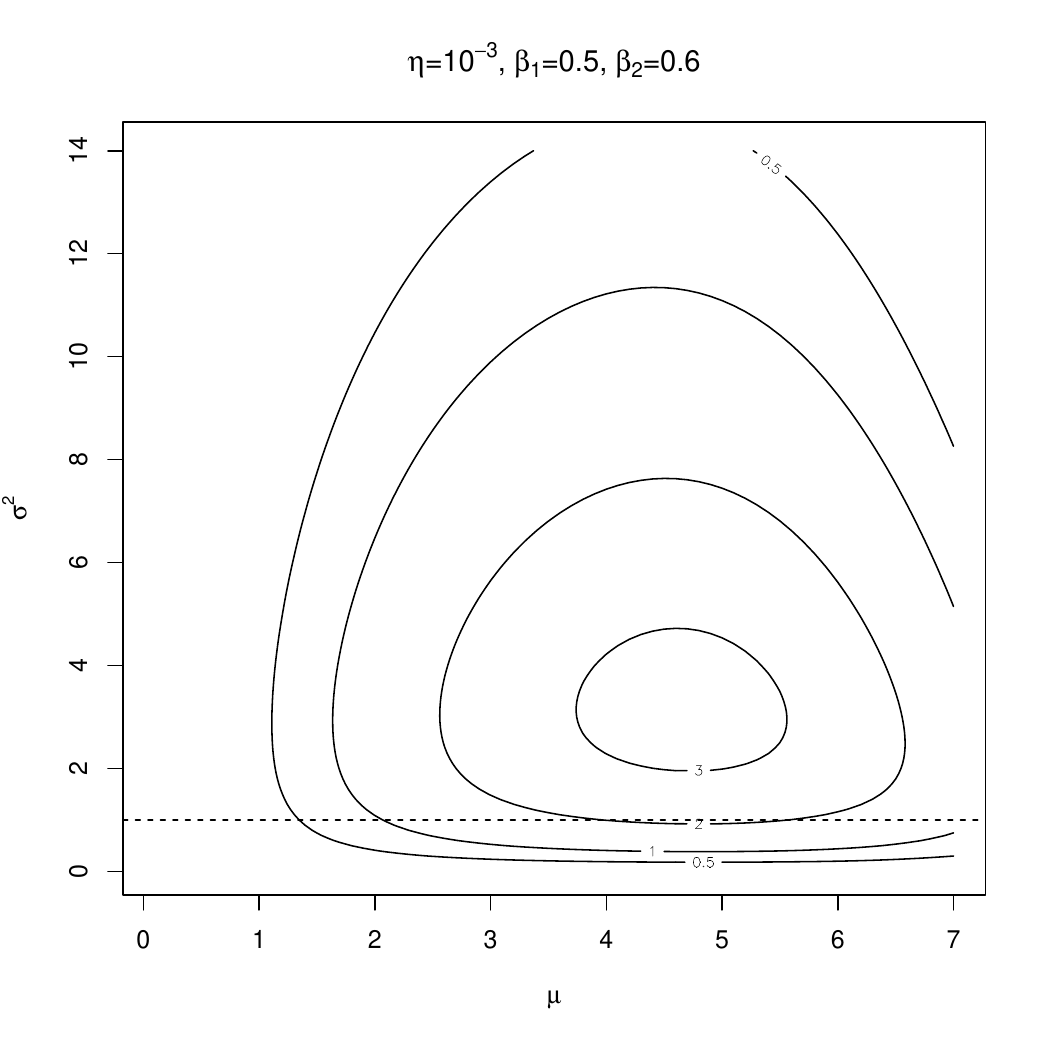}
 }
 
\caption{
Contour plots of the asymptotic efficiency relative to the optimal
efficiency of the equivalent pseudo-marginal RWM algorithm,
$\eff^{\textrm{rel}}_{\textrm{dapm}}$, as a
function of the scaling, $\mu$, and the variance of the noise in the
log-target, $\sigma^2$, for different choices of $\beta_1,~\beta_2$
at $\eta=10^{-3}$. For comparability, all contours are at
$0.5,1,2,3,4,5,6,7,8,10,12,15$. The horizontal dashed line denotes $\sigma^2=1$.
\label{fig.eff.mu.sigma2}
}
\end{center}
\end{figure}

\section{Validation of Lemma \ref{lem.pivotal}}
\label{sec.valid.both}
The product form Assumptions \eqref{eq.target.distribution} and
\eqref{eq.deterministic.log.error} from which we derive the bivariate Gaussian distribution in Lemma
\ref{lem.pivotal} are chosen for convenience. We expect the same
conclusions to hold, at least approximately, in much broader settings; for example, we believe that extensions of Lemma \ref{lem.pivotal} to non i.i.d target distributions similar to those discussed in
\cite{breyer2004optimal,BedardA:2007,BedardRosenthal:2008,Sherlock/Roberts:2009,BeskosRobertsStuart:2009,PST12} are possible, at the cost of much less transparent proofs. Here we verify the conclusions of Lemma \ref{lem.pivotal} numerically in two examples from the main paper.

\subsection{ODE model}
\label{sec.valid.ODE}
First we investigate the ODE Example \ref{example.ode} described in Equation \eqref{eq.ode.model}. For values of the jump scaling parameter $\lambda$ associated with acceptance rates in the range $10 \% \leq \alpha_{\textrm{rwm}}(\lambda) \leq 95\%$, we display the quantities $\loc(S_{\Delta}) / \lambda^2$, $\scale(S_{\Delta}) / \lambda$, $\loc(S_{\Delta}+Q_{\Delta}) / \lambda^2$ and $\scale(S_{\Delta}+Q_{\Delta}) / \lambda$ where $S_{\Delta} \equiv S(\bmx_0 + \lambda \, \xi) - S(\bmx_0)$ and $S(\bmx) \equiv \log[\pi(\bmx) / \pi_a(\bmx)]$ and $Q_{\Delta} \equiv \log[\pi(\bmx_0 + \lambda \, \xi) / \pi(\bmx_0)]$. Here, $\bmx_0 \in \RR^{10}$ is chosen in the bulk of the distribution, the perturbation $\xi$ is a centred Gaussian random variable with covariance approximately matching that of $\pi$. For a random variable $V$, $\loc(V)$ denotes its median, and $\scale(V) \equiv q_{75} - q_{25}$ is the inter-quartile range. We report the location and scale parameters instead of the mean and standard-deviation for increased robustness. Lemma \ref{lem.pivotal} predicts that the reported quantities are insensitive to the value of the scale parameter $\lambda$, and hence of the acceptance rate $\alpha_{\textrm{rwm}}(\lambda)$. Figure \ref{fig.pivotal.lemma} (right) shows that this property approximately holds true, although departure from the theory is noticeable for low acceptance rate, i.e. large Gaussian perturbations. Figure \ref{fig.pivotal.lemma} (Left) shows that the pair $(Q_\Delta, S_{\Delta})$ is approximately jointly Gaussian, as predicted by Lemma \ref{lem.pivotal}.

\begin{figure}[h]
\begin{center}
  \includegraphics[height=0.3\textheight, width=0.45\textwidth]{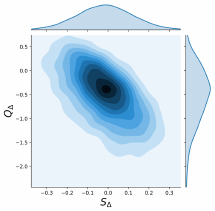}
  \includegraphics[height=0.3\textheight, width=0.45\textwidth]{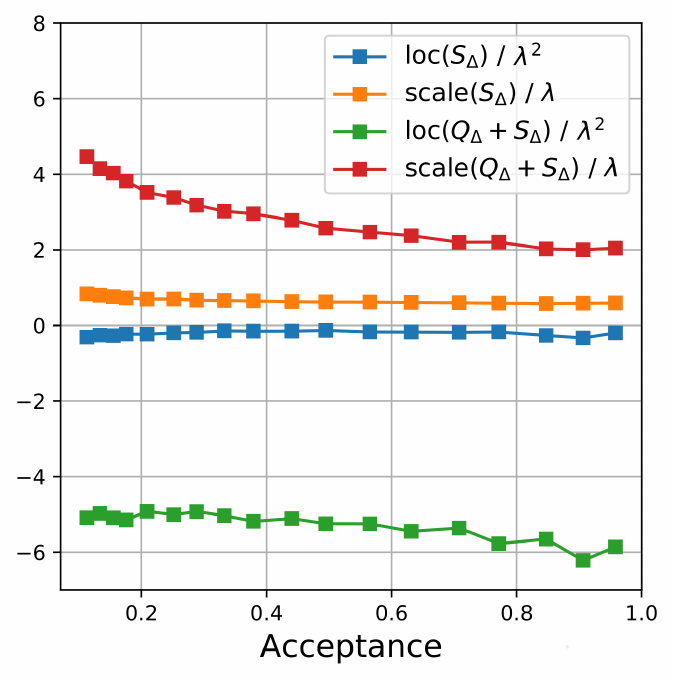}
\caption{Empirical verification of Lemma \ref{lem.pivotal} on the ODE Example \ref{example.ode} defined in Equation \eqref{eq.ode.model}.
\label{fig.pivotal.lemma}
}
\end{center}
\end{figure}

\subsection{Lotka-Volterra model}

\label{sec.valid.lotka}

\begin{figure}[h]
\begin{center}
  \includegraphics[width = 1. \textwidth]{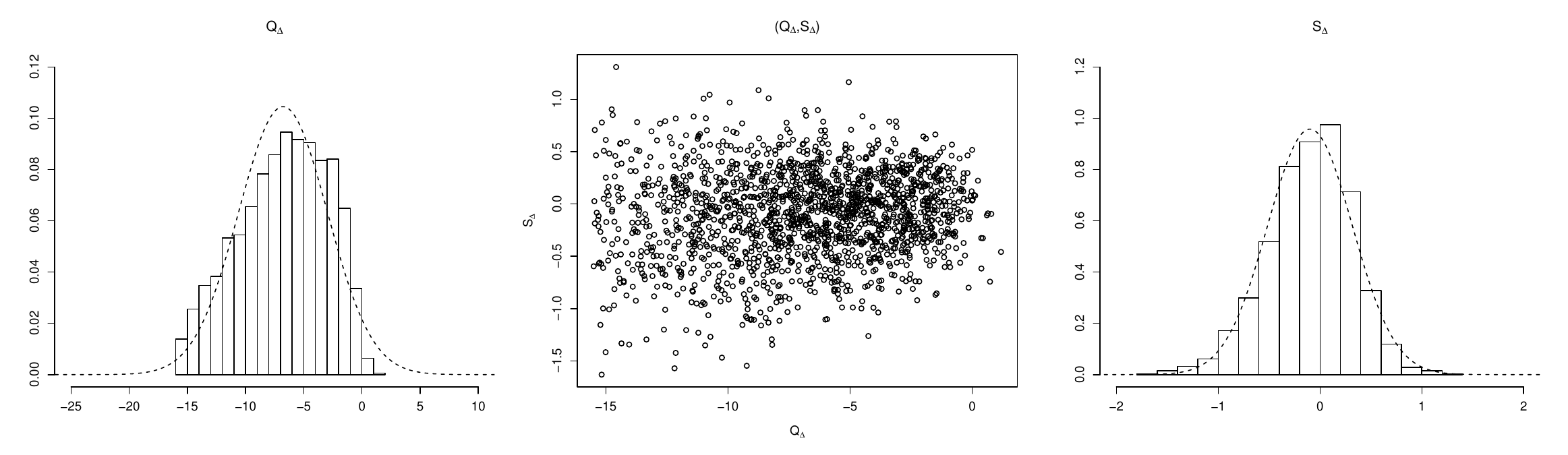}
\caption{
Empirical distribution, for the Lotka-Volterra model with LNA approximation, of $q^{(d)}_{\Delta}(\bmx^{(d)}, \bmX^{(d),*})$ and 
$s^{(d)}_{\Delta}(\bmx^{(d)}, \bmX^{(d),*})$ evaluated from a fixed
current point in the bulk of the target distribution. The dashed lines in the
left and right panels show the densities of Gaussian fits
to the empirical marginal distributions of $q^{(d)}_{\Delta}(\bmx^{(d)}, \bmX^{(d),*})$ and 
$s^{(d)}_{\Delta}(\bmx^{(d)}, \bmX^{(d),*})$ respectively.
\label{fig.QS.distributions.LV}
}
\end{center}
\end{figure}

To validate Lemma \ref{lem.pivotal} in the case of the five-dimensional Lotka-Volterra example we picked a point from the bulk of the posterior and used a very large number of particles to obtain an accurate estimate of the posterior, $\pi$, then ran the LNA to obtain $\pi_a$. From this point we repeatedly proposed jumps, and for each proposed point we also  obtained $\pi_a$ and an accurate estimate of $\pi$. 

Figure \ref{fig.QS.distributions.LV} shows histograms of $Q_{\Delta}$ and $S_{\Delta}$ and a scatter plot of the two quantities. Because of the low dimension ($d=5$), $Q_{\Delta}$ is only approximately Gaussian, although $\mathbb{E}[Q_{\Delta}]\approx-6.7 \approx -0.5\times 13.7\approx -0.5 \times  \mbox{var}[Q_{\Delta}]$, which fits with the theory. The Gaussian approximation to $S_\Delta$ is much more accurate, and the further three quantities of $\mathbb{E}[S_{\Delta}]\approx-0.096$, $\mbox{Var}[S_\Delta]\approx 0.17$ and $\mbox{Cor}[Q_{\Delta},S_{\Delta}]\approx 0.11$ are consistent with the two values $\beta_1\approx -0.013$ and $\beta_2\approx 0.11$.

%
%
\section{Proofs}
\label{sec.proof}
It will be helpful to introduce i.i.d sequences $\{X_i\}_{i \geq 1}$
and $\{\Aux_i\}_{i \geq 1}$ respectively marginally distributed as
$\pi$ and $\pi_{\Aux}$, and corresponding realisations of them, $\{x_i\}_{i \geq 1}$ and $\{\aux_i\}_{i \geq 1}$. Similarly, we consider an i.i.d sequence $\{Z_{i,k}\}_{i,k \geq 0}$ of standard Gaussian $\Normal{0,1}$ random variables, $\{U_k\}_{k \geq 0}$ an i.i.d sequence of random variables uniformly distributed on $[0,1]$, $W$ a random variable distributed as $\pi_W$ and $\{W^*_{k}\}_{k \geq 0}$ an i.i.d sequence distributed as $\pi_{W^*}$. For any dimension $d \geq 1$ we set $\bmX^{(d)}_0 = (X_1, \ldots, X_d) \in \RR^d$ and  $W^{(d)}_0 = W$ and $X^{(d),*}_{k,j} = X^{(d)}_{k,j} + (\mu / I) \, d^{-1/2} \, Z_{k,j}$; we recursively define 
\begin{align*}
(\bmX^{(d)}_{k+1}, W^{(d)}_{k+1})
=
\left\{
\begin{array}{ll}
(\bmX^{(d),*}_{k}, W^*_k) & \mathrm{if} \quad U_k < \acc^{(d)}\left( \bmX^{(d)}_{k}, W^{(d)}_k ;\bmX^{(d),*}_{k}, W^*_k \right)\\
(\bmX^{(d)}_{k}, W^{(d)}_{k}) & \mathrm{otherwise},
\end{array}
\right.
\end{align*}
for a proposal $\bmX^{(d),*}_d = (X^{(d),*}_{k,1}, \ldots, X^{(d),*}_{k,d})$. Indeed, the process $(\bmX^{(d)}_k, W^{(d)}_k)$ is a DAPsMRWM Markov chain started at stationarity and targeting $\pi^{(d)} \otimes \pi_{W}$.
We denote by $\FF_k$ the $\sigma$-algebra generated by the family of random variables $\big\{ \bmX^{(d)}_{t},W^{(d)}_{t} \mid t \leq k \big\}$ and use the notation $\EE_k[ \, \cdot \, ]$ for designating the conditional expectation $\EE[ \, \cdot \mid \FF_k]$. Similarly, we use the notation $\EE_{\bmx,w}[ \, \cdot \, ]$ instead of $\EE[ \, \cdot \mid (\bmX^{(d)}_0, W^{(d)}_0)=(\bmx,w)]$. Finally, we set
\begin{align}
\label{eqn.QSWdefs}
\begin{array}{lll}
q^{(d)}_{\Delta} = q^{(d)}_{\Delta}(\bmx^{(d)}, \bmx^{(d),*}),&
s^{(d)}_{\Delta} =  s^{(d)}_{\Delta}(\bmx^{(d)}, \bmx^{(d),*}),&
w^{(d)}_{\Delta} =  w^{(d),*} - w^{(d)},\\
\mathsf{Q}^{(d)}_{\Delta} = q^{(d)}_{\Delta}(\bmx^{(d)}, \bmX^{(d),*}),&
\mathsf{S}^{(d)}_{\Delta} = s{(d)}_{\Delta}(\bmx^{(d)}, \bmX^{(d),*}),&
\mathsf{W}^{(d)}_{\Delta} =  W^{(d),*} - w^{(d)},\\
Q^{(d)}_{\Delta} = q{(d)}_{\Delta}(\bmX^{(d)}, \bmX^{(d),*}),&
S^{(d)}_{\Delta} =  s{(d)}_{\Delta}(\bmX^{(d)}, \bmX^{(d),*}),&
W^{(d)}_{\Delta} =  W^{(d),*} - W^{(d)}.\\
\end{array}
\end{align}
and use the shorthand notation $\ell(x) \equiv \log \pi(x)$.

\subsection{Proof of Lemma \ref{lem.pivotal}}
\label{sec.proof.lemma.pivotal}
The Law of Large Numbers and the separability of $L^1(\pi \otimes \pi_{\Aux})$ readily yield that for almost every realisations $\{x_i\}_{i \geq 1}$ and $\{\aux_i\}_{i \geq 1}$, the following holds,
\begin{align} \label{eq.ae.cv.dist}
\lim_{n \to \infty} \; n^{-1} \, \sum_{i=1}^n \phi(x_i, \aux_i) = \int \phi(x,\aux) \, \left(\pi \otimes \pi_{\Aux}\right)(dx, d\aux)
\qquad \textrm{for all} \quad 
\phi \in L^1(\pi \otimes \pi_{\Aux}).
\end{align}
We can thus safely assume in the remainder of this section that Equation
\eqref{eq.ae.cv.dist} holds for the realisation $\{\aux_i\}_{i \geq 1}$
of the auxiliary random variables used to describe the deterministic approximation \eqref{eq.deterministic.log.error} .
By the Cramer-Wold device, for proving Lemma \ref{lem.pivotal} it suffices to establish that for any coefficient $c_Q,c_S \in \RR$ the sequence 
$c_Q \,  Q_{\Delta}^{(d)}(\bmx^{(d)}) + c_S \,  S_{\Delta}^{(d)}(\bmx^{(d)})$ converges in law towards $c_Q \,  Q^{\infty}_{\Delta} + c_S \,  S^{\infty}_{\Delta}$; the boundedness assumption on the derivatives of the functions $x \mapsto \loglik(x)$ and $x \mapsto \SS(x,u)$ and a second order Taylor expansion show that this is equivalent to proving that
the sum
\begin{align}
\frac{\mu }{\sqrt{I^2 \, d}} 
\sum_{i=1}^d \Big\{ c_Q \, \loglik'(x_i) + c_S \, \partial_x \SS(x_i,\aux_i) \Big\} \, Z_i
+
\frac12 \, \frac{\mu^2}{I^2 \, d}
\sum_{i=1}^d \Big\{ c_Q \, \loglik{''}(x_i) + c_S \, \partial_{xx} \SS(x_i,\aux_i) \Big\}
\end{align}
converges in law towards $c_Q \,  Q^{\infty}_{\Delta} + c_S \,  S^{\infty}_{\Delta}$.
Definition \eqref{eq.betas} of the coefficient $\beta_1$ and $\beta_2$
yields that for almost every realisation $\{x_i\}_{i \geq 1}$ and $\{\aux_i\}_{i \geq 1}$ we have
\begin{align}
\label{eq.as.cv.betas}
\frac{1}{I^2 \, d} \sum_{i=1}^d
\Big(
\loglik{'}(x_i)^2, \, 
\loglik{''}(x_i), \, 
\partial_{x} \SS(x_i,\aux_i)^2, \,
\partial_{xx} \SS(x_i,\aux_i), \,
\loglik{'}(x_i) \, \partial_{x} \SS(x_i,\aux_i)
\Big)
\; \to \;
\Big(
1, \,
-1, \, 
\beta^2_2, \,
\beta_1, \,
-\beta_1
\Big),
\end{align}
from which the conclusion directly follows since
$c_Q \,  Q^{\infty}_{\Delta} + c_S \,  S^{\infty}_{\Delta}$ has a Gaussian distribution with mean $\mu^2 \, (c_S \beta_1 - c_Q)/2$ and variance $\mu^2 \, (c_Q^2 + c_S^2 \, \beta^2_2 - 2 \, c_Q \, c_S \, \beta_1)$. 



\subsection{Proof of Proposition \ref{prop.limiting.esjd}}
The quantity $\esjd^{(d)}$ can also be expressed as
\begin{align}
\esjd^{(d)}
&= \left( \frac{\mu^2 }{ I^2 \, d} \right) \, \sum_{j=1}^d \EE\left[  \left( Z_j^{(d)} \right)^2 \times F\left( Q^{(d)}_{\Delta}  + S^{(d)}_{\Delta}\right) \times F\left( W^{(d)}_{\Delta}  - S^{(d)}_{\Delta}\right) \right]\\
&= \frac{\mu^2 }{ I^2 }  \, \EE\left[ \left( Z_1^{(d)} \right)^2 \times F\left( Q^{(d)}_{\Delta}  + S^{(d)}_{\Delta}\right) \times F\left( W^{(d)}_{\Delta}  - S^{(d)}_{\Delta}\right) \right]
\end{align}
for $Q^{(d)}_{\Delta}$, $S^{(d)}_{\Delta}$, $W^{(d)}_{\Delta}$ defined in \eqref{eqn.QSWdefs}; the second equality follows from the exchangeability, at stationarity, of the $d$ coordinates of the Markov chain. One can decompose $Q^{(d)}_{\Delta}$ and $S^{(d)}_{\Delta}$ as a sum of a term that is independent of $Z_1^{(d)}$ and a negligible term; we have  $Q^{(d)}_{\Delta}=Q^{(d),\perp}_{\Delta}+\log \big[ \pi(\bmX^{(d),*}_1) / \pi(\bmX^{(d)}_1) \big]$
and
$S^{(d)}_{\Delta}=S^{(d),\perp}_{\Delta}+\SS(\bmX^{(d),*}_1, \aux_1) - \SS(\bmX^{(d)}_1, \aux_1)$ with
\begin{align}
Q^{(d), \perp}_{\Delta} 
=
\sum_{j=2}^d \log \big[ \pi(\bmX^{(d),*}_j) / \pi(\bmX^{(d)}_j) \big]
\qquad \textrm{and} \qquad
S^{(d)}_{\Delta,\perp}
=
\sum_{j=2}^d \SS(\bmX^{(d),*}_j, \aux_j) - \SS(\bmX^{(d)}_j, \aux_j).
\end{align}
Note that $Q^{(d), \perp}_{\Delta}$ and $S^{(d), \perp}_{\Delta}$ are
independent of $Z^{(d)}_1$. Under Assumption \ref{ass.regularity}, the
moments of order two of the differences $Q^{(d)}_{\Delta}-Q^{(d),
  \perp}_{\Delta}$ and $S^{(d)}_{\Delta}-S^{(d), \perp}_{\Delta}$ are
finite and converges to zero as $d \to \infty$. The Cauchy-Schwarz
inequality and the fact that $F$ is bounded and Lipschitz yield that $\esjd^{(d)} / (\mu / I)^2$ can also be expressed as
\begin{align}
\EE\Big[ &\left( Z_1^{(d)} \right)^2 \times F\left( Q^{(d),\perp}_{\Delta}  + S^{(d), \perp}_{\Delta}\right) \times F\left( W^{(d)}_{\Delta}  - S^{(d), \perp}_{\Delta}\right) \Big] \\
&\qquad + 
\Expect{ \left( Z_1^{(d)} \right)^2 \times F\left( Q^{(d),\perp}_{\Delta}  + S^{(d), \perp}_{\Delta}\right) \times \left\{ F\left( W^{(d)}_{\Delta}  - S^{(d)}_{\Delta}\right) - F\left( W^{(d)}_{\Delta}  - S^{(d), \perp}_{\Delta}\right) \right\}} \\
&\qquad + 
\Expect{ \left( Z_1^{(d)} \right)^2 \times \left\{ F\left( Q^{(d)}_{\Delta}  + S^{(d)}_{\Delta}\right) - F\left( Q^{(d),\perp}_{\Delta}  + S^{(d), \perp}_{\Delta}\right) \right\} \times F\left( W^{(d)}_{\Delta}  - S^{(d)}_{\Delta}\right) } \\
&=
\Expect{ F\left( Q^{(d),\perp}_{\Delta}  + S^{(d), \perp}_{\Delta}\right) \times F\left( W^{(d)}_{\Delta}  - S^{(d), \perp}_{\Delta}\right) } + o(1)\\
&=
\Expect{ F\left( Q^{\infty}_{\Delta}  + S^{\infty}_{\Delta}\right) \times F\left( W_{\Delta}  - S^{\infty}_{\Delta}\right) } + o(1)
= \alpha_{12} + o(1),
\end{align}
as required. We have used the fact that for almost every realisation of the auxiliary random variable $\{\Aux_j\}_{j \geq 1}$ the sequence $\left(Q^{(d),\perp}_{\Delta},S^{(d),\perp}_{\Delta} \right)$ converges in distribution to $\left(Q^{\infty}_{\Delta}, S^{\infty}_{\Delta} \right)$, which readily follows from Lemma \ref{lem.pivotal}.

%
%
\subsection{Proof of Theorem \ref{thm.diff.lim}}
\label{sec.proof.diff.lim}
The proof is a generalisation of the generator approach of \cite{Roberts/Gelman/Gilks:1997,BedardA:2007} coupled with an homogenization argument.
We introduce the subsampled processes 
$\wtilde{\bmX}^{(d)}$ and $\wtilde{W}^{(d)}$ defined by
\begin{align}
\wtilde{\bmX}^{(d)}_k
= \bmX^{(d)}_{k \times T^{(d)}}
\qquad \textrm{and} \qquad
\wtilde{W}^{(d)}_k
= W^{(d)}_{k \times T^{(d)}}
\end{align}
for an intermediary time scale defined as $T^{(d)}=\floor{ d^{\gamma} }$ where $\gamma$ is an arbitrary exponent such that $\gamma \in (0, 1/4)$.
One step of the process $\wtilde{\bmX}^{(d)}$ (resp. $\wtilde{W}^{(d)}$) corresponds to $T^{(d)}$ steps of the process $\bmX^{(d)}$ (resp. $W^{(d)}$). 
We then define an accelerated version $\wtilde{V}^{(d)}$ of the subsampled process $\wtilde{X}^{(d)}$. In order to prove a diffusion limit 
for the process $X^{(d)}$, one needs to accelerate time by a factor of $d$; consequently, in order to prove a diffusion limit for the process 
$\wtilde{X}^{(d)}$, one needs to accelerate time by a factor $d / T^{(d)}$ and thus define $\wtilde{V}^{(d)}$ by
\begin{equation*}
\wtilde{V}^{(d)}(t) := \wtilde{X}^{(d)}_{\floor{td/T^{(d)}},1}.
\end{equation*}
The proof then consists of showing that the sequence $\wtilde{V}^{(d)}$ converges weakly in the Skorohod topology towards the 
limiting diffusion \eqref{e.limiting.diffusion} and verifying that $\| \wtilde{V}^{(d)} - V^{(d)}\|_{\infty, [0,T]}$ converges to zero in probability; 
this is enough to prove that the sequence  $V^{(d)}$ converges weakly in the Skorohod topology towards the 
limiting diffusion \eqref{e.limiting.diffusion}.
We denote by $\LL$ the generator 
of the limiting diffusion $\eqref{e.limiting.diffusion}$. Similarly,
we define $\LL^{(d)}$ and $\wtilde{\LL}^{(d)}$ the approximate generators 
of the first coordinate processes $X^{(d)}_1$ and $\wtilde{X}^{(d)}_1$; for
any smooth and compactly supported test function $\phi:\RR \to \RR$,
vector $\bmx=(x_1, \ldots,x_d) \in \RR^d$ and scalar $x,w \in \RR$ we have 
\begin{align}
\left\{
\begin{array}{ll}
\LL^{(d)} \phi(\bmx,w)
&=
\EE_{\bmx,w}[ \phi(X^{(d)}_{1,1})-\phi(X^{(d)}_{0,1})] /  \delta\\
\wtilde{\LL}^{(d)} \phi(\bmx,w)
&=
\EE_{\bmx,w}[ \phi(\wtilde{X}^{(d)}_{1,1})-\phi(\wtilde{X}^{(d)}_{0,1})] / (T^{(d)} \times \delta)
\qquad \textrm{for} \qquad \delta \equiv 1/d\\
\LL \phi(x) 
&= 
\frac{1}{2}\, J(\mu) \times \left( \loglik'(x) \phi'(x) + \phi''(x) \right).
\end{array}
\right.
\end{align}
Note that although $\phi$ is a scalar function, the functions $\LL^{(d)} \phi$ and $\wtilde{\LL}^{(d)}\phi$ are defined on $\RR^d \times \RR$.
The law of iterated conditional expectation yields the important identity between the generators $\LL^{(d)}$ and $\wtilde{\LL}^{(d)}$,
\begin{align}
\label{eq.telescoping}
\wtilde{\LL}^{(d)}\phi(\bmx,w)
=
\frac{1}{T^{(d)}} \, \EE_{\bmx,w}\sqBK{ \sum_{k=0}^{T^{(d)}-1} \LL^{(d)}\phi \left( \bmX^{(d)}_k, W^{(d)}_k \right) }.
\end{align}
For clarity, the proof of Theorem \ref{thm.diff.lim} is divided into several steps.\\

\subsubsection{The finite dimensional marginals of $\wtilde{V}^{d}$ converge to those of the diffusion \eqref{e.limiting.diffusion}}
Since the limiting process is a scalar diffusion, the set of smooth and compactly supported functions
is a core for the generator of the limiting diffusion (\cite{ethier1986markov},Theorem $2.1$, Chapter $8$); in the sequel, one can thus work with test functions belonging to this core only.
Because the processes are started at stationarity, it suffices to show 
(\cite{ethier1986markov},Chapter $4$, Theorem $8.2$, Corollary $8.4$) that 
for any smooth and compactly supported function $\phi:\RR \to \RR$ the following limit holds,
\begin{align} \label{Lp.convergence.generator}
\lim_{d \to \infty} \; 
\EE \left[ \left| \wtilde{\LL}^{(d)}\phi(X_1, \ldots, X_d,W) - \LL \phi(X_1) \right|^2 \right]
= 0.
\end{align}
The proof of Equation \eqref{Lp.convergence.generator} spans the remaining of this section and is based on an asymptotic expansion that we now describe.
For every $x,w \in \RR$ we define the approximated generator $\A \phi: \RR \times \RR \to \RR$ by
%
%
\begin{align}
\label{eq.approximate.generator.2}
\A \phi(x,w) = 
\left( \frac{\mu}{I} \right)^2 \left\{ A(w) \, \loglik'(x) + \left( \frac{1}{2} \alpha_{12} + [A(w) - B(w)] \, \partial_x \SS(x, \gamma_1) \right) \, \phi^{''}(x) \right\}
\end{align}
where $A,B: \RR \to (0;\infty)$ are two bounded and continuous functions defined by
\begin{align} \label{eq.A.B}
\left\{
\begin{array}{ll}
A(w) &= \EE \big[ F'(Q^{\infty}_{\Delta}+S^{\infty}_{\Delta}) \times F(W^*-w - S^{\infty}_{\Delta})\big]\\
B(w) &= \EE \big[ F(Q^{\infty}_{\Delta}+S^{\infty}_{\Delta}) \times F'(W^*-w - S^{\infty}_{\Delta})\big]
\end{array}
\right.
\end{align}
for $W^* \dist \pi_{W^*}$ 
and $F'(u) = e^u \, \mathbb{I}_{u<0}$
and $(Q^{\infty}_{\Delta},S^{\infty}_{\Delta})$ as defined in \eqref{eq.limiting.Q.S}. The functions $A,B: \RR \to \RR_+$ are such that
\begin{align}
\label{eq.chris.lemma}
\Expect{ A(W) } = \Expect{ B(W) } = \frac{1}{2} \, \alpha_{12}.
\end{align}
The proof of \eqref{eq.chris.lemma} can be found in Appendix \ref{appendix.chris.lemma}. It follows from \eqref{eq.chris.lemma} that for any fixed $x \in \RR$ we have
\begin{align}
\label{eq.averaging.A}
\EE\big[ \A \phi(x,W) \big] = \LL \phi(x)
\end{align}
for a random variable $W \dist \pi_W$.
%
%
\begin{lem}
\label{lem.generator.approx}
Let Assumptions \ref{ass.regularity} hold. We have
\begin{align}
\label{generator.approx}
\lim_{d \to \infty} \; 
\Expect{ \, \left| \LL^{(d)} \phi(X_1, \ldots, X_d, W)
- \A \phi(X_1, W) \right|^2 \, } \; = \; 0.
\end{align}
\end{lem}
The proof of Lemma \eqref{lem.generator.approx} consists in second order Taylor expansion and an averaging argument; details are in Section \ref{sec.proof.lem.generator.approx}. For proving Equation \eqref{Lp.convergence.generator}, note that identity \eqref{eq.telescoping} and Jensen's inequality yield 
the quantity inside the limit described in Equation \eqref{Lp.convergence.generator} is less than two times the expectation of
\begin{align}
\left\{ \frac{ \sum_{k=0}^{T^{(d)}} \LL^{(d)}\phi\left(\bmX^{(d)}_k,W^{(d)}_k\right) - \A\phi\left( \bmX^{(d)}_{k,1},W^{(d)}_k \right)  }{  T^{(d)} }\right\}^2 + 
\left\{\frac{ \sum_{k=0}^{T^{(d)}} \A\phi \left(\bmX^{(d)}_{k,1},W^{(d)}_k\right) - \LL\phi \left(\bmX^{(d)}_{0,1}\right) }{ T^{(d)}}\right\}^2.
\end{align}
The expectation of the first term is less than $\Expect{ \big| \LL^{(d)} \phi(X_1, \ldots, X_d, W)- \A \phi(X_1, W) \big|^2 }$ and Lemma \eqref{lem.generator.approx} shows that this quantity goes to zero as $d \to \infty$. To finish the proof it thus remains to verify that the expectation of the second term also converges to zero; to prove so, note that the second term is less than two times
\begin{align}
\label{eq.A.L}
\frac{\sum_{k=0}^{T^{(d)}} \left| \A\phi(X^{(d)}_{k,1},W^{(d)}_k) - \A\phi(X^{(d)}_{0,1},W^{(d)}_k) \right|^2}{T^{(d)}}
+
\left\{
  \sum_{k=0}^{T^{(d)}} \A\phi\left( X^{(d)}_{0,1},W^{(d)}_k \right) - \LL\phi\left( X^{(d)}_{0,1} \right){T^{(d)}} \right\}^2.
\end{align}
Under the assumptions of Theorem \ref{thm.diff.lim}, it is straightforward to verify that the function $\A \phi$ is globally Lipschitz in the sense that there exists a constant $\|\A \phi\|_{\textrm{Lip}}$ such that for every $x_1, x_2, w \in \RR$ we have $|\A \phi(x_1, w) - \A \phi(x_2, w)| \leq \|\A \phi\|_{\textrm{Lip}} \times |x_1 - x_2|$; since $\EE[(X^{(d)}_{k,1} - X^{(d)}_{0,1})^2] \lesssim k^2 / d$, it follows that
\begin{align*}
\EE \BK{ \big| \A\phi(X^{(d)}_{k,1},W^{(d)}_k) - \A\phi(X^{(d)}_{0,1},W^{(d)}_k) \big|^2 }
\lesssim \frac{k^2}{d}.
\end{align*}
Consequently, the expectation of the first term in \eqref{eq.A.L} converges to zero. For proving that the second term also converges to zero, we make use of the following ergodic averaging Lemma whose proof can be found in Section \ref{sec.proof.lem.averaging}.
%
%
\begin{lem}\label{lem.averaging}
Let $h:\RR \to \RR$ be a bounded and measurable test function. We have
\begin{align*}
\lim_{d \to \infty} \;
\Expect{ \, \left|  \frac{ \sum_{k=0}^{T^{(d)}-1} h(W^{(d)}_k)}{T^{(d)}} - \Expect{h(W)} \right|^2 \, } = 0,
\end{align*}
for a random variable $W \dist \pi_W$ independent from any other sources of randomness. 
\end{lem}
Identity \eqref{eq.averaging.A}, a standard conditioning argument and Lemma \ref{lem.averaging} yield that the expectation of the second term in Equation \eqref{eq.A.L} also converges to zero; this finishes the proof of  the convergence of the finite dimensional marginals of $\wtilde{V}^{d}$ to those of the limiting diffusion \eqref{e.limiting.diffusion}.

\subsubsection{The sequence $\wtilde{V}^{d}$ converges weakly towards the diffusion \eqref{e.limiting.diffusion}}
The finite dimensional marginals of the sequence process
$\wtilde{V}^{d}$ converges to those of the diffusion
\eqref{e.limiting.diffusion}. To prove that the sequence
$\wtilde{V}^{d}$ actually converges to the diffusion
\eqref{e.limiting.diffusion}, it thus suffices to verify that the
sequence $\wtilde{V}^{d}$ is relatively weak compact in the  Skorohod
topology: since the process $\wtilde{V}^{(d)}$ is started at
stationarity and the space of smooth functions with compact support is
an algebra that strongly separates points, (\cite{ethier1986markov},
Chapter $4$, Corollary $8.6$)  states that it suffices to show that for any smooth and compactly supported test function $\phi$ the sequence $d \mapsto \EE \big| \wtilde{\LL}^{(d)} \phi(X_1, \ldots, X_d,W) \big|^2$ is bounded. Equation \eqref{Lp.convergence.generator} shows that it suffices to verify that $\EE \big| \LL \phi(X) \big|^2 < \infty$ for $X \dist \pi$, which is obvious since $\phi$ is assumed to be smooth with compact support.
\subsubsection{The sequence $V^{d}$ converges weakly towards the diffusion \eqref{e.limiting.diffusion}}
Because the sequence $\wtilde{V}^{d}$ converges weakly to the diffusion \eqref{e.limiting.diffusion}, it suffices to prove that the difference $\|V^{d}-\wtilde{V}^{d}\|_{\infty,[0,T]}$ goes to zero in probability. To this end, it suffices to prove that the supremum
\begin{equation} \label{eq.supremum}
\sup \left\{ \, 
\left|X^{(d)}_{kT^{(d)}+i,1} - X^{(d)}_{kT^{(d)},1}  \right| \; : \; k \times T^{(d)} \leq d
\times T, \; i \leq T^{(d)} \, \right\}
\end{equation}
converges to zero in probability. The quantity $\left|X^{(d)}_{kT^{(d)}+i,1} - X^{(d)}_{kT^{(d)},1}  \right|$ is less than a constant times 
\begin{equation*}
\frac{1}{d^{1/2}} \, \Big\{ \big|Z_{k T^{(d)},1}\big| + \ldots + \big|Z_{(k+1)T^{(d)}-1,1} \big| \Big\}.\end{equation*}
Therefore, for any $\epsilon > 0$ and integer $p \geq 1$ and an i.i.d sequence of standard Gaussian random variables $\{\xi_k\}_{k \geq 0}$, the union bound and Markov's inequality show the probability that the supremum in Equation \eqref{eq.supremum} is larger than $\epsilon$ is less than a constant multiple of
\begin{align*}
\frac{d}{T^{(d)}} \times \PP\sqBK{|\xi_1| + \ldots + |\xi_{T^{(d)}|} > \epsilon \, d^{1/2}}
&\leq
\frac{d}{T^{(d)}} \times \frac{ \EE\sqBK{ \BK{ |\xi_1| + \ldots + |\xi_{T^{(d)}}| }^p} }{ \epsilon^p \, d^{p/2}}
\end{align*}
Since $T^{(d)} = d^{\gamma}$ and for every integer $p \geq 1$ we have that $\EE\sqBK{ \BK{ |\xi_1| + \ldots + |\xi_{n}| }^p} \leq C(p) \, n^p$ for a constant $C(p)$ that only depends on $p$, it follows that
\begin{align*}
\frac{d}{T^{(d)}} \times \PP\sqBK{|\xi_1| + \ldots + |\xi_{T^{(d)}|} > \epsilon \, d^{1/2}}
\lesssim d^{1-\gamma + p [\gamma - 1/2]}.
\end{align*}
Since $\gamma < 1/2$, one can choose $p$ large enough such that $1-\gamma + p [\gamma - 1/2] < 0$. This concludes the proof of Theorem \ref{thm.diff.lim}.

\section{Proof of technical results}
\label{sec.technical}
In this section we denote by $\Phi(x) = \int_{-\infty}^x \phi(u) \, du$ the cumulative Gaussian function with $\phi(u) = e^{-u^2/2} / \sqrt{2 \pi}$. The bound $1-\Phi(x) < \phi(x)/x$ for $x>0$ is used in several places.

\subsection{Proof of Proposition \ref{prop.acc.rates.dec}}
\label{sec.prop.acc.rates.dec}
The only not entirely trivial parts of this proposition involve
establishing that $\alpha_1$ and $\alpha_{2|1}$ are decreasing in
$\mu$ and $\sigma$ respectively. For proving that
$\alpha_1=G\BK{-\frac{\mu^2}{2}(1-\beta_1), \, \mu^2 \,
  \BK{1+\beta^2_2 - 2 \beta_1}}$ is decreasing as a function of $\mu$
when $\beta_1<1$, note that since $|\beta_1|<\beta_2$, 
$1+\beta^2_2 - 2 \beta_1 \ge (1-\beta_1)^2$; hence  it suffices to
show that for any positive constant $c>0$ the function $h:\mu \mapsto G(-\mu^2, c^2\mu^2)$ is decreasing. Since $h(\mu) = \Expect{F(-\mu^2 + c \, \mu \, \xi)}$ for a random variable $\xi \dist \Normal{0,1}$ and $F'(x) = e^x \, \mathbb{I}(x<0)$ it follows that 
\begin{align*}
h'(\mu)=\int_{z \in \mathbb{R}} F'(-\mu^2 + c \, \mu \, z) \, (-2 \, \mu + c \, z)\, \phi(z) \, dz
&=
\int_{z < \mu/c} F'(-\mu^2 + c \, \mu \, z) \, (-2 \, \mu + c \, z) \, \phi(z) \, dz.
\end{align*}
This quantity is negative since $-2 \, \mu + c \, z < 0$ on the event $\{z : z<\mu/c\}$. Proving that $\alpha_{2|1}$ is decreasing as a function of $\sigma$ readily follows from the fact that
for any fixed $a \in \mathbb{R}$ the derivative of the function
$\sigma \mapsto G(-\sigma^2+a, 2\sigma^2)<-\sqrt{2}\phi(-\sigma/\sqrt{2}+a/(\sigma\sqrt{2})) < 0$  and differentiation under the integral sign.

\subsection{Proof of Theorem \ref{thrm.behave.eff}}
\label{sec.proof.thrm.behave.eff}
Since $\eff(\mu, \sigma^2) = \frac{\mu^2 \, \alpha_{12}(\mu, \sigma)}{\eta + \alpha_1(\mu) / \sigma^2}$, for a fixed value of scaling $\mu>0$ the efficiency functional goes to zero as $\sigma \to 0$ and $\sigma \to \infty$. Similarly, the fact that the efficiency goes to zero as $\mu \to 0$ for any fixed value of $\sigma > 0$ is straightforward; it remains to verify that the efficiency also converge to zero as $\mu \to \infty$. It suffices to show that $\mu^2 \, \alpha_{12}(\mu, \sigma) \to 0$; since for any $x,y \in \RR$ we have $\min\left( 1, e^x\right) \, \min\left( 1, e^y\right) \leq \min\left( 1, e^{x+y}\right)$, 
\begin{equation*}
\alpha_{12} \leq \Expect{F\left( Q^{\infty}_{\Delta} + W_{\Delta}\right)} = 2 \, \Phi\left\{ -\frac{(\mu^2 + 2\sigma^2)^{1/2}}{2}\right\}
\end{equation*}
and the conclusion readily follows.

\subsection{Proof of Equation \eqref{eq.chris.lemma}}
\label{appendix.chris.lemma}
Equation \ref{eq.change.W.W} yields that $R \equiv W^* - W$ for $(W^*,W) \sim \pi_{W^*} \otimes \pi_W$ has a density $\pi_R$ such that the function $r \mapsto e^{r/2} \, \pi_R(r)$ is symmetric i.e. $e^{r/2} \, \pi_R(r)=e^{-r/2} \, \pi_R(-r)$. 
Similarly, algebra reveals that the joint Gaussian density $\pi_{Q,S}(q,s)$ of the pair $\BK{Q^{\infty}_{\Delta}, S^{\infty}_{\Delta}}$ described in Lemma \ref{lem.pivotal} is such that
\begin{equation*}
e^{q/2} \, \pi_{Q,S}(q,s)
=
e^{-q/2} \, \pi_{Q,S}(-q,-s).
\end{equation*}
That is because $- \log \pi_{Q,S}(q,s) = a\, q^2 + b \, s^2 + c \, qs - q/2 + \textrm{(constant)}$ for some coefficients $a,b,c \in \RR$.
Consequently, since the accept reject function $F$ is such that $e^{-u}F(u)=F(-u)$ for any $u \in \RR$, the function 
\begin{align*}
g(q,r,s)
&= e^{(q+s)/2} \, F(r-s) \, \pi_{Q,S}(q,s) \, \pi_R(r)\\
&= e^{-(r-s)/2} \, F(r-s) \, \BK{ e^{q/2} \, \pi_{Q,S}(q,s)} \, \BK{ e^{r/2} \, \pi_R(r) }
\end{align*}
is such that $g(q,r,s)=g(-q,-r,-s)$. It follows that
\begin{align*}
\Expect{ A(W) } &= \Expect{ F'(Q^{\infty}_{\Delta}+S^{\infty}_{\Delta}) \times F(R - S^{\infty}_{\Delta}) }
=
\iiint_{\RR^3} F'(q+s) \, F(r-s) \, \pi_{Q,S}(q,s) \, \pi_R(r) \, dq \, dr \, ds\\
&=
\iiint_{\RR^3} e^{-(q+s)/2}F'(q+s)\, g(q,r,s) \, dq \, dr \, ds
=
\iiint_{\RR^3} e^{(q+s)/2}F'(-[q+s])\, g(q,r,s) \, dq \, dr \, ds\\
&=
\Expect{e^{Q^{\infty}_{\Delta}+S^{\infty}_{\Delta}}F'(-[Q^{\infty}_{\Delta}+S^{\infty}_{\Delta}]) F(R-S^{\infty}_{\Delta})}.
\end{align*}
Consequently, since $F'(u)+e^u \, F'(-u)=F(u)$ for $u \in \RR$, it follows that
\begin{align*}
2 \times \Expect{ A(W) } 
&= \Expect{F'(Q^{\infty}_{\Delta}+S^{\infty}_{\Delta}) \times F(R - S^{\infty}_{\Delta})+e^{Q^{\infty}_{\Delta}+S^{\infty}_{\Delta}}F'(-[Q^{\infty}_{\Delta}+S^{\infty}_{\Delta}]) F(R-S^{\infty}_{\Delta})}\\
&= \Expect{ F(Q^{\infty}_{\Delta}+S^{\infty}_{\Delta})} \equiv \alpha_{12}.
\end{align*}
The proof that $\Expect{B(W)} = \alpha_{12} / 2$ is similar and thus omitted.

\subsection{Proof of Lemma \ref{lem.generator.approx}}
\label{sec.proof.lem.generator.approx}
In this section we need to consider asymptotic expansions of the type $\EE_{\bmx,w}[\ldots] =\Psi(\bmx,w) + \textrm{(error term)}$, where $(\bmx,w) \in \RR^d \times \RR$ and $\textrm{(error term)} = \epsilon_d(\bmx,w)$ for a function $\epsilon_d: \RR^d \times \RR \to \RR$. We use the notation $\textrm{(error term)} = o_{L^2}(1)$ to indicates that, under the equilibrium distribution, the moment of order two of the error term is asymptotically negligible, $\Expect{\epsilon_d(\bmX^{(d)},W)^2} \to 0$ as $d \to \infty$ for $(\bmX^{(d)},W) \dist \pi^{(d)} \otimes \pi_W$.
Since $\phi$ is smooth with compact support, a second order Taylor expansion reveals that
\begin{align} \label{eq.gen.asymp}
\LL^{(d)} \phi(\bmx,w)
&=
\textrm{(drift term)} \, \phi'(\bmx)
+
(1/2) \, \textrm{(volatility term)} \, \phi^{''}(\bmx)
+ o_{L^2}(1)
\end{align}
where the drift and volatility terms are given by the following conditional expectations,
\begin{align}
\left\{
\begin{array}{ll}
\textrm{(drift term)}
&\;=\;
(1/\delta) \times \EE_{\bmx,w} \left[  \left(X^{(d),*}_{1,1} - x_1 \right) \, \alpha^{(d)}_{12}\left(\bmx,w,\bmX^{(d),*},W^{(d),*}\right) \right]  \\
\textrm{(volatility term)}
&\;=\;
(1/\delta) \times \EE_{\bmx,w} \left[  \left(X^{(d),*}_{1,1} - x_1\right)^2 \, \alpha^{(d)}_{12}\left(\bmx,w,\bmX^{(d),*},W^{(d),*}\right) \right]
\end{array}
\right.
\end{align}
with $\bmX^{(d),*}_1 = \bmx + (\mu / I) \,  \delta^{1/2} \, \bmZ^{(d)}$ and standard centred Gaussian random variable $\bmZ^{(d)} = (Z_1, \ldots, Z_d)$
\begin{itemize}
\item It readily follows from Lemma \ref{lem.pivotal} that for $\pi$-almost every $\bmx$ we have
\begin{equation} \label{eq.vol.asymp}
\textrm{(volatility term)} = \alpha_{12} \times (\mu / I)^2 = J(\mu) + o_{L^2}(1).
\end{equation}
\item
For the drift term, we make use of the following integration-by-part formula, also known as Stein's identity,
%
%
\begin{equation} \label{eq.stein}
\Expect{ Z \times g(Z) } = \Expect{ g'(Z) }
\qquad \textrm{for} \qquad Z \dist \Normal{ 0,1 },
\end{equation}
which holds for any continuous and piecewise continuously differentiable function $g: \RR \to \RR$ such that $x \mapsto \max\left( g(x),g'(x) \right)$ is polynomially bounded. In what follows, $F'(u) = e^u \, \mathbb{I}_{u<0}$. The expression for $\alpha_{12}^{(d)}\left( \bmx,w,\bmX^{(d),*},W^{(d),*} \right)$, identity \eqref{eq.stein} and standard algebraic manipulations yield that
\begin{align} \label{eq.drift.asymp}
\begin{aligned}
\textrm{(drift term)}
&\;=\;
\delta^{-1/2} \, (\mu / I) \, \EE_{\bmx,w}[ Z_1  \, \alpha^{(d)}_{12}(\bmx,w,\bmX^{(d),*},W^{(d),*})] \\
&\;=\;
(\mu / I)^2 \, \EE_{\bmx,w}[ 
F'(\mathsf{Q}^{(d)}_{\Delta}+\mathsf{S}^{(d)}_{\Delta}) \, F(\mathsf{W}^{(d)}_{\Delta}-\mathsf{S}^{(d)}_{\Delta}) \,  \big\{\loglik'(X^{(d),*}_{1,1}) + \partial_x \SS(X^{(d),*}_{1,1},\gamma_1)\big\}]\\
&\qquad -
(\mu / I)^2 \, \EE_{\bmx,w}[ 
F(\mathsf{Q}^{(d)}_{\Delta}+\mathsf{S}^{(d)}_{\Delta}) \, F'(\mathsf{W}^{(d)}_{\Delta}-\mathsf{S}^{(d)}_{\Delta}) \,  \partial_x \SS(X^{(d),*}_{1,1},\gamma_1)]\\
&\;=\;
(\mu / I)^2 \, A(w)  \, \loglik'(x_1)
+
(\mu / I)^2 \, [A(w) - B(w)] \, \partial_x \SS(x_1, \gamma_1)
+ o_{L^2}(1),
\end{aligned}
\end{align}
where the functions $A,B:\RR \to \RR^+$ are defined in Equation \eqref{eq.A.B} and the quantities $\mathsf{Q}^{(d)}_{\Delta},\mathsf{S}^{(d)}_{\Delta}$ and $\mathsf{W}^{(d)}_{\Delta}$ in Equation \eqref{eqn.QSWdefs}.
\end{itemize}
Plugging \eqref{eq.drift.asymp} and \eqref{eq.vol.asymp} into \eqref{eq.gen.asymp} shows that the limit
\begin{equation*}
\lim_{d \to \infty} \; \Expect{ \left| \LL^{(d)} \phi(\bmX^{(d)},W) - \A \phi(X^{(d)}_1, W) \right|^2} = 0
\end{equation*}
holds for $\left( \bmX^{(d)}, W \right) \sim \pi^{(d)} \otimes \pi_W$, as required.

\subsection{Proof of Lemma \ref{lem.averaging}}
\label{sec.proof.lem.averaging}
The strategy of the proof is as follows. We define three stochastic processes 
$\left\{ W^{(d)}_{\clubsuit,k} \right\}_{k \geq 0}$,
$\left\{ W^{(d)}_{\spadesuit,k} \right\}_{k \geq 0}$,
$\left\{ W_{\blacksquare,k} \right\}_{k \geq 0}$ such that
\begin{align} \label{eq.W.ergodic.strategy}
\left\{
\begin{array}{ll}
\quad & \lim_{d \to \infty} \; \Prob{W^{(d)}_{\clubsuit,k}=W^{(d)}_{k} \; : \; 0 \leq k \leq T^{(d)} } \; = \; 1, \\
\quad &\left( W^{(d)}_{\clubsuit,k}=W^{(d)}_{k} \; : \; 0 \leq k \leq T^{(d)} \right) \; {\overset{\mathrm{law}}{=}} \; 
\left( W^{(d)}_{\spadesuit,k}=W^{(d)}_{k} \; : \; 0 \leq k \leq T^{(d)} \right),\\
\quad & \lim_{d \to \infty} \; \Prob{W^{(d)}_{\spadesuit,k}=W_{\blacksquare,k} \; : \; 0 \leq k \leq T^{(d)} } \; = \; 1,\\
\quad &
\textrm{$\left\{ W_{\blacksquare,k} \right\}_{k \geq 0}$ is a Markov chain that is ergodic with respect to $\pi_W$}.
\end{array}
\right.
\end{align}
Once \eqref{eq.W.ergodic.strategy} is proved, Lemma \ref{lem.averaging} immediately follows. Let us now defines these three processes and verify that Equation \eqref{eq.W.ergodic.strategy} holds. To do so, let us consider i.i.d sequences $\{X_i\}_{i \geq 1}$ and $\{W^*_i\}_{i \geq 1}$ and $\{Z_{i,k}\}_{i,k \geq 1}$ and $\{U_k\}_{k \geq 0}$ respectively marginally distributed as $\pi$ and $\pi_{W^*}$ and $\Normal{0,1}$ and $\textrm{Uniform}([0,1])$. We consider $\{x_i\}_{i \geq 1}$ a realisation of $\{X_i\}_{i \geq 1}$ and for any index $d \geq 1$ we set $\bmX^{(d)}_0 = (x_1, \ldots, x_d)$ and $W^{(d)}_0 \dist \pi_W$ and recursively define $\left(\bmX^{(d)}_{k+1}, W^{(d)}_{k+1}\right) = \left(\bmX^{(d),*}_k,W^{*}_k\right)$, with $\bmX^{(d),*}_k = \bmX^{(d)}_k + (\mu/I) \, \delta^{1/2} \, \bmZ^{(d)}_k$ and $\bmZ^{(d)}_k=(Z_{1,k}, \ldots, Z_{d,k})$, if
\begin{align} \label{eq.event.W}
U_k \leq 
F\left( Q^{(d)}_{\Delta,k}  + S^{(d)}_{\Delta,k} \right) \times F\left( W^*_k - W^{(d)}_{k}- S^{(d)}_{\Delta,k}\right)
\end{align}
and  $\left( \bmX^{(d)}_{k+1}, W^{(d)}_{k+1}\right) = \left(\bmX^{(d)}_{k}, W^{(d)}_{k} \right)$ otherwise. In the above
\begin{align*}
\left\{
\begin{array}{ll}
Q^{(d)}_{\Delta,k}
&= \sum_{i=1}^d \ell\left( X^{(d),*}_{k,i} \right) - \ell\left( X^{(d)}_{k,i} \right)\\
S^{(d)}_{\Delta,k}
&= \sum_{i=1}^d \SS\left( X^{(d),*}_{k,i}, \gamma_i \right) - \SS\left( X^{(d)}_{k,i}, \gamma_i \right).
\end{array}
\right.
\end{align*}
Indeed, for any index $d \geq 1$ the process $\left\{ \left( \bmX^{(d)}_{k}, W^{(d)}_{k} \right) \right\}_{k \geq 0}$ is a DAPsMRWM Markov chain that targets $\pi^{(d)} \otimes \pi_W$. Let us now define the processes $W_{\clubsuit}$,$W_{\spadesuit}$,$W_{\blacksquare}$.
\begin{itemize}
%
%
\item We set $W^{(d)}_{\clubsuit,0} = W^{(d)}_0$ and recursively define $W^{(d)}_{\clubsuit,k+1} = W^{*}_{k}$ if
\begin{align} \label{eq.event.W.clubsuit}
U_k \leq F\left( Q^{(d)}_{\clubsuit,\Delta,k}  + S^{(d)}_{\clubsuit,\Delta,k} \right) \times F\left( W^*_k - W^{(d)}_{\clubsuit,k}- S^{(d)}_{\clubsuit,\Delta,k}\right)
\end{align}
and $W^{(d)}_{\clubsuit,k+1} = W^{(d)}_{\clubsuit,k}$ otherwise; we have used the notations
\begin{align*}
\left\{
\begin{array}{ll}
Q^{(d)}_{\clubsuit,\Delta,k}
&= 
(\mu \delta / I) \, \sum_{i=1}^d \ell'(x_i) Z_{i,k}
+
(\mu^2 \delta^2 / 2 \, I^2) \, \sum_{i=1}^d \ell''(x_i) \\
S^{(d)}_{\clubsuit,\Delta,k}
&= 
(\mu \delta / I) \, \sum_{i=1}^d \mathcal{S}'(x_i, \gamma_i) Z_{i,k}
+
(\mu^2 \delta^2 / 2 \, I^2) \, \sum_{i=1}^d \mathcal{S}''(x_i, \gamma_i).
\end{array}
\right.
\end{align*}
%
%
\item Similarly, we set $W^{(d)}_{\spadesuit,0} = W^{(d)}_0$ and recursively define $W^{(d)}_{\spadesuit,k+1} = W^{*}_{k}$ if
\begin{align} \label{eq.event.W.spadesuit}
U_k \leq F\left( Q^{(d)}_{\spadesuit,\Delta,k}  + S^{(d)}_{\spadesuit,\Delta,k} \right) \times F\left( W^*_k - W^{(d)}_{\spadesuit,k}- S^{(d)}_{\spadesuit,\Delta,k}\right)
\end{align}
and $W^{(d)}_{\spadesuit,k+1} = W^{(d)}_{\spadesuit,k}$ otherwise; we have used the notations $\left( Q^{(d)}_{\spadesuit,\Delta,k}, S^{(d)}_{\spadesuit,\Delta,k}\right)$ to designate a Gaussian random variable in $\RR^2$, independent from any other source of randomness, with same law as $\left( Q^{(d)}_{\clubsuit,\Delta,k}, S^{(d)}_{\clubsuit,\Delta,k}\right)$.
%
%
\item Finally, we set $W^{(d)}_{\blacksquare,0} = W^{(d)}_0$ and recursively define $W^{(d)}_{\blacksquare,k+1} = W^{*}_{k}$ if
\begin{align} \label{eq.event.W.infty}
U_k \leq F\left( Q^{(\infty)}_{\Delta,k}  + S^{(\infty)}_{\Delta,k} \right) \times F\left( W^*_k - W^{(d)}_{\blacksquare,k}- S^{(\infty)}_{\Delta,k} \right)
\end{align}
and $W^{(d)}_{\blacksquare,k+1} = W^{(d)}_{\blacksquare,k}$ otherwise; in the above $\left\{ \left( Q^{(\infty)}_{\Delta,k}, S^{(\infty)}_{\Delta,k} \right) \right\}_{k \geq 0}$ is an i.i.d sequence marginally distributed as $\left( Q^{(\infty)}_{\Delta}, S^{(\infty)}_{\Delta} \right)$; see Lemma \ref{lem.pivotal}.
\end{itemize}
It is obvious that $\left\{ W^{(d)}_{\clubsuit,k} \right\}_{k \geq 0}$ and $\left\{ W^{(d)}_{\spadesuit,k} \right\}_{k \geq 0}$ have the same law. The fact that $\left\{ W^{(d)}_{\blacksquare,k} \right\}_{k \geq 0}$ is a Markov chain ergodic with respect to $\pi_W$ readily follows from the fact that it is reversible with respect to $\pi_W$; it is a standard Gaussian computation. The proof of the first and third equation in \eqref{eq.W.ergodic.strategy} is based on the following basic remark. For convenience, let us denote by $\mathcal{E}^{(d)}_k$,$\mathcal{E}^{(d)}_{k, \clubsuit}$,$\mathcal{E}^{(d)}_{k, \spadesuit}$,$\mathcal{E}^{(d)}_{k, \infty}$ the Bernoulli random variables indicating whether or not the respective events \eqref{eq.event.W},\eqref{eq.event.W.clubsuit},\eqref{eq.event.W.spadesuit}, \eqref{eq.event.W.infty} are realised or not. We have
\begin{align} \label{eq.union.bound}
1-\Prob{W^{(d)}_{\clubsuit,k}=W^{(d)}_{k} \; : \; 0 \leq k \leq T^{(d)} }
&\leq \sum_{k=0}^{T^{(d)}-1} \Prob{ \left. \mathcal{E}^{(d)}_k \neq \mathcal{E}^{(d)}_{k, \clubsuit} \right| W^{(d)}_{\clubsuit,k}=W^{(d)}_{k}}
\end{align}
and the conditional probability $\Prob{ \left. \mathcal{E}^{(d)}_k \neq \mathcal{E}^{(d)}_{k, \clubsuit} \right| W^{(d)}_{\clubsuit,k}=W^{(d)}_{k}}$ is less than the expectation, conditioned upon the event $\left\{ W^{(d)}_{\clubsuit,k}=W^{(d)}_{k} \right\}$, of the absolute difference
\begin{align} \label{eq.absolute.difference}
\left| F\left( Q^{(d)}_{\Delta,k}  + S^{(d)}_{\Delta,k} \right) \, F\left( W^*_k - W^{(d)}_{k}- S^{(d)}_{\Delta,k}\right)
-
F\left( Q^{(d)}_{\clubsuit,\Delta,k}  + S^{(d)}_{\clubsuit,\Delta,k} \right) \, F\left( W^*_k - W^{(d)}_{\clubsuit,k}- S^{(d)}_{\clubsuit,\Delta,k}\right) 
\right|.
\end{align}
Because the $[0,1]$-valued function $F$ is assumed to be Lipschitz, if $W^{(d)}_{\clubsuit,k}=W^{(d)}_{k}$ the absolute difference in \eqref{eq.absolute.difference}
is less than $2 \times \|F\|_{\textrm{Lip}} \times \left\{ \left|Q^{(d)}_{\Delta,k} - Q^{(d)}_{\clubsuit,\Delta,k} \right| + \left| S^{(d)}_{\Delta,k} - S^{(d)}_{\clubsuit,\Delta,k}\right| \right\}$.
Because the second and third derivatives of the log-likelihood function $\ell$ are globally bounded, a third order Taylor expansion yield that
\begin{align*}
\EE&\left| Q^{(d)}_{\Delta,k} - Q^{(d)}_{\clubsuit,\Delta,k} \right|
\lesssim
d^{-1/2} \, \EE \left| \sum_{i=1}^d \left( \ell'(X^{(d)}_{k,i}) - \ell'(x_i) \right) \, Z_{i,k} \right|
+
d^{-1} \, \EE \left| \sum_{i=1}^d \left( \ell''(X^{(d)}_{k,i}) - \ell''(x_i) \right) \, Z^2_{i,k} \right|
+ \OO(d^{-1/2})\\
&\lesssim
d^{-1/2} \, \left\{ \sum_{i=1}^d \Expect{ \left( \ell'(X^{(d)}_{k,i}) - \ell'(x_i) \right)^2 } \right\}^{1/2}
+
d^{-1} \, \left\{ \sum_{i=1}^d \Expect{ \left( \ell''(X^{(d)}_{k,i}) - \ell''(x_i) \right)^2 } \right\}^{1/2}
+ \OO(d^{-1/2})\\
&= \OO(k \, d^{-1/2}).
\end{align*}
We have used the fact that for any exponent $p \geq 1$ we have $\Expect{ \left| X^{(d)}_{k,i} - x_i\right|^p}^{1/p} \lesssim k \, d^{-1/2}$, which readily follows from the triangular inequality. Similarly, we have that
$\EE\left| S^{(d)}_{\Delta,k} - S^{(d)}_{\clubsuit,\Delta,k} \right| \lesssim k \, d^{-1/2}$. Plugging these estimates in \eqref{eq.union.bound} shows that
\begin{align*}
1-\Prob{W^{(d)}_{\clubsuit,k}=W^{(d)}_{k} \; : \; 0 \leq k \leq T^{(d)} }
\lesssim
d^{-1/2} \, \sum_{k=0}^{T^{(d)}-1} k \; \to \; 0
\end{align*}
since $T^{(d)} = d^{\gamma}$ for some exponent $\gamma \in (0, 1/4)$; we have thus proved that $\Prob{W^{(d)}_{\clubsuit,k}=W^{(d)}_{k} \; : \; 0 \leq k \leq T^{(d)} }$ converges to one as $d \to \infty$. The proof of the estimate $\Prob{W^{(d)}_{\spadesuit,k}=W^{(d)}_{\blacksquare,k} \; : \; 0 \leq k \leq T^{(d)} } \to 1$ uses the same ingredients and is thus omitted.

\section{Pseudo-marginal estimate at Stage 1}
\label{app.pmstageOne}
Throughout the main article we have assumed that $\pi_a(x)$ is a computationally cheap, \emph{deterministic} approximation to $\pi(x)$. However, in the DAPsMRWM setting, where only a noisy but unbiased non-negative estimator, $\pihat(x,U)$, of $\pi(x)$ is available one might have a cheaper, noisy, non-negative unbiased estimator, $\pihat_a(x,U_a)$ to use in Stage One. Indeed, $\pihat_a(x,U_a)$ could even be biased, and the delayed-acceptance algorithm would still be valid \cite[e.g., see ][]{GolightlyHendersonSherlock:2013}. Here, we provide an equivalent to Lemma 4.1 and Proposition 4.1 for this general scenario, and then investigate the use of a noisy, unbiased estimator at Stage One; this suggests an upper bound on the efficiency achievable in the case of a noisy, biased estimator.

Analogously to Section 3.3, we may write
\[
\pihat_a^{(d)}(x,u_a)=\pi_a^{(d)}(x)e^{w_a}.
\]
Assumptions 1 and 2 lead to distributions $W_a^*\sim N(-\sigma^2_a/2,\sigma_a^2)$ for the noise in the density at the proposed value, and
$W_a\sim N(\sigma^2_a/2,\sigma_a^2)$ for the noise in the density at the current value. Similarly, Assumptions 3 gives the compute time for Stage One as being proportional to $\sigma^{-2}_a$.

Section \ref{sec.PMPM.proof.sketch} sketches the proof, by steps analogous to those in Section 4, that the following limits exist:
 $\alpha_1(\mu,\sigma_a^2;\beta_1,\beta_2)=\lim_{d\rightarrow \infty} \alpha_1^{(d)}(X^{(d)},W_a^{(d)})$ and $\alpha_{12}(\mu,\sigma^2,\sigma_a^2;\beta_1,\beta_2)=\lim_{d\rightarrow \infty} \alpha_{12}(X^{(d)},W_a^{(d)},W^{(d)})$.  We do not pursue a complete analysis of the joint tuning of three parameters over the range of possible values of $\beta_1$ and $\beta_2$; instead we investigate the limit as $\beta_1\rightarrow 0$ and $\beta_2\downarrow 0$. This describes the behaviour as the bias in the cheap approximation approaches $0$, and might provide an approximate upper bound on the possible improvements achievable when the Stage One approximation is, in fact, biased. In this limit may obtain the following values, analogous to the expressions in Appendix \ref{sec.explicit.alphas} (where $G$ is defined):
\begin{align}
  \label{eqn.PMPM.a}
  \alpha_1(\mu,\sigma_a^2;0,0)
  &=
  G\left(-\frac{1}{2}\mu^2-\sigma_a^2,\mu^2+2\sigma_a^2\right)\\
    \label{eqn.PMPM.aa}
  \alpha_{12}(\mu,\sigma_a^2,\sigma^2;0,0)
  &=
  \Expect{G\left(-\frac{1}{2}\mu^2-\sigma_a^2+\sqrt{2}\sigma_a\xi,\mu^2\right)
  G\left(-\sigma^2+\sigma_a^2-\sqrt{2}\sigma_a\xi,2\sigma^2\right)},
\end{align}
where $\xi\sim \mathsf{N}(0,1)$.
An analogous proof to that of Proposition 4.3 then gives
$\lim_{d\rightarrow \infty} ESJD^{(d)}\propto\alpha_{12}\mu^2$.

For the DAPsRWM, $\eta$ was the relative computational cost of the determistic approximation to the pseudo marginal approximation when $\sigma^2=1$. When the approximation at Stage One is also random we define $\eta$ to be the relative computational cost of $\pihat_a$ when $\sigma_a^2=1$ to the cost of $\pihat$  when $\sigma^2=1$. Thus, following Section 5.2, the total computational cost is proportional to $\eta/\sigma_a^2 + \alpha_1/\sigma^2$, and we obtain a limiting algorithm efficiency of
\begin{align}
  \label{eqn.effdapmpm}
  \eff_{\mathrm{dapmpm}}(\mu,\sigma^2,\sigma^2)
  &=
  \mu^2 \frac{\alpha_{12}(\mu,\sigma^2,\sigma_a^2)}
  {\eta/\sigma_a^2 + \alpha_1(\mu,\sigma_a^2)/\sigma^2}.
\end{align}

With $\eta=0.1$ the optimal efficiency in \eqref{eqn.effdapmpm} is only a factor of $\approx 2.1$ better than standard pseudo-marginal MCMC and we conclude that, as with the DAPsRWM algorithm (see Section 5.2), when only an order of magnitude cheaper, $\pihat_a$ is of dubious utility. Finally, at $\eta=0.01$ the optimal efficiency factor is $\approx 6.8$ and the algorithm could be worth implementing provided that the bias is small. In \cite{GolightlyHendersonSherlock:2013}, a biased, noisy estimator at Stage One (with $\eta\approx 0.1$) is found to lead to an efficiency of at most double that of the PsMRWM, agreeing with the above. In contrast, for the same target, a deterministic approximation at Stage One (with $\eta\approx 1/300$) leads to an efficiency increase of over an order of magnitude relative to the PsMRWM.

In \cite{GolightlyHendersonSherlock:2013} it is also suggested (though not implemented) that the noise in $\pihat_a$ could be made positively correlated with the noise in $\pihat$ since this should increase the Stage Two acceptance probability on average. For example, in particle MCMC-based inference on the parameters of an SDE, $\pihat_a$ could use a coarser Euler-Maruyama time step than $\pihat$ and both could use the same driving Brownian motion. Use of the same Brownian motion would produce the highest correlation possible between the two estimators; however, the particle-filter resampling steps would ensure this was below $1$.

The article \cite{berard2013lognormal} considers the output from a single particle filter, rather than the joint output from two correlated particle filters, and it is beyond the scope of this work to rigorously generalise this result. If, however, the joint noise vector $(W_{a,\Delta}^\infty,W_\Delta^{\infty})$ in Proposition \ref{prop.equiv.lem.pivotal} has a bivariate Gaussian distribution with a correlation of $\rho$ then it is straightforward to generalise     \eqref{eqn.PMPM.aa} and, for a given $\rho$, to optimise the efficiency with respect to $(\mu,\sigma_a^2,\sigma^2)$. Figure \ref{fig.dapspsrwmvsrho} plots the optimal $(\mu,\sigma_a^2,\sigma^2)$ and the efficiency relative to optimally tuned PsMRWM, all against the correlation, $\rho$, and under the assumption of a bivariate Gaussian density.

\begin{figure}[h]
\begin{center}
  \includegraphics[width = 0.8 \textwidth]{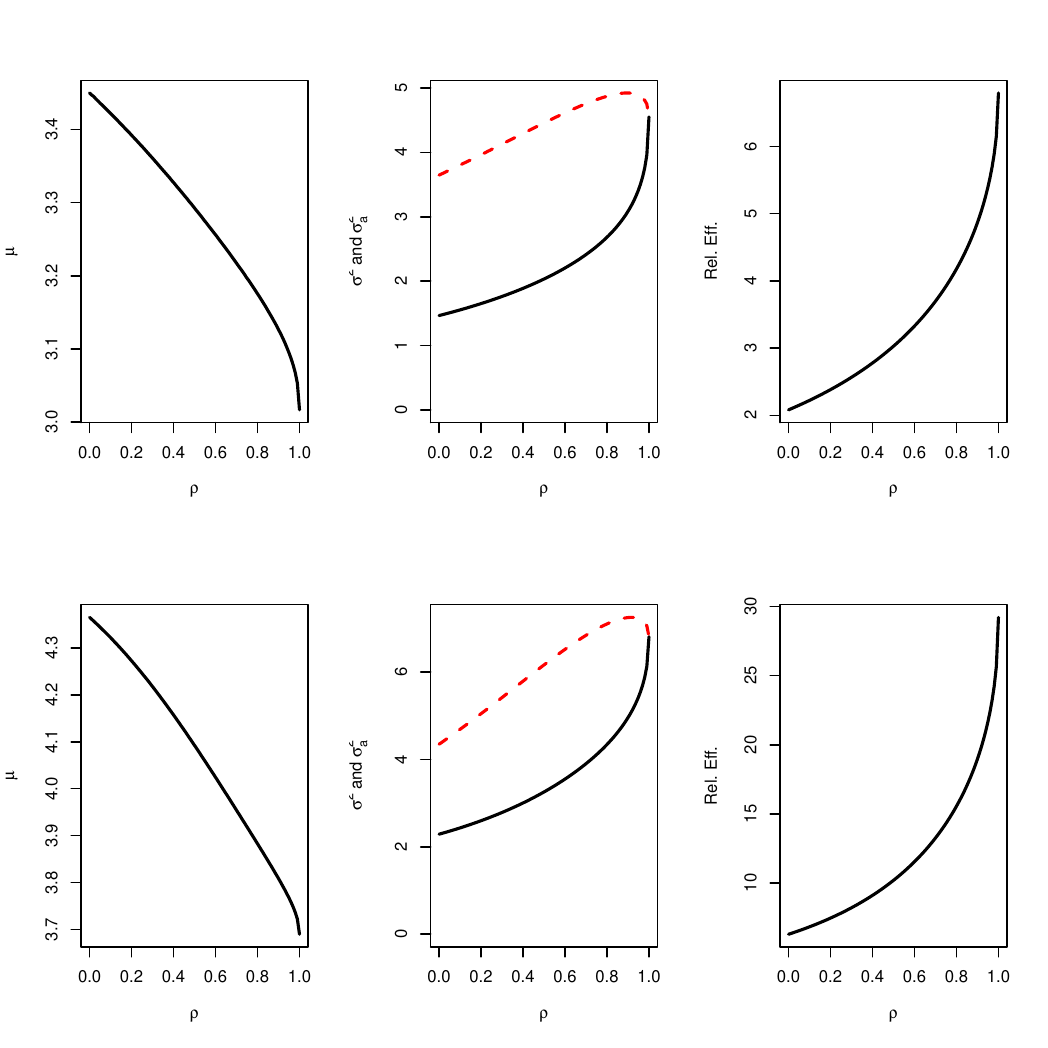}
\caption{
Optimal values for $\mu$ (left) and $\sigma_a^2$ (red dashed line) and $\sigma^2$ (solid black line) (centre), as well as the efficiency relative to optimised pseudo-marginal RWM (right), when $\eta=0.1$ (top) and $\eta=0.01$ (bottom). Plots created under the assumption of a bivariate Gaussian density for $(W_{a,\Delta}^\infty,W_\Delta^{\infty})$.
\label{fig.dapspsrwmvsrho}
}
\end{center}
\end{figure}

As anticipated, the optimal efficiency increases with $\rho$.
Also unsurprisingly, $\sigma_a^2\ge \sigma^2$; further, the optimal variance tends to increase with $\rho$ because a large positive correlation reduces the difference $W_\Delta^\infty -W_{a,\Delta}^\infty$. However, when $\rho=1$, $\widehat{\sigma}^2_a=\widehat{\sigma}^2$ precisely so that (in this limit where $\beta_1=\beta_2=0$) $W_\Delta^\infty -W_{a,\Delta}^\infty=0$ and acceptance is guaranteed at Stage Two. In this case the algorithm becomes PsRWM on the unbiased $\widehat{\pi}_a$, but the cost is $(\eta+\alpha_1)/\sigma^2$ instead of $1/\sigma^2$, potentially leading to a very large improvement in efficiency. In reality, when $\pihat_a$ is biased, but the bias is small, strong positive correlation in the noises of the two approximations should, therefore, lead to substantial improvements in efficiency.

\subsection{Steps in the derivation of \eqref{eqn.PMPM.a} and \eqref{eqn.PMPM.aa}}
\label{sec.PMPM.proof.sketch}
The derivation of \eqref{eqn.PMPM.a} and \eqref{eqn.PMPM.aa} follows analogous steps to those in Section 4 of the main text. We describe the intermediate results; the proofs are either special cases of or slight variations of the proofs of the equivalent steps for the results in the main text and are omitted.

Lemma 4.1 is still applicable, since $(q_{\Delta}^{(d)}(x^{(d)},X^{(d),*}),s_{\Delta}^{(d)}(x^{(d)},X^{(d),*}))$ and  $W_a^{(d),*}-W_a^{(d)}$ are independent; we include the joint limit with $W_a^{(d),*}-W_a^{(d)}$ and $W^{(d),*}-W^{(d)}$ to aid with the exposition on correlated noise.

\begin{prop} \label{prop.equiv.lem.pivotal}
Let Assumptions 1 and 2 hold for both the noise in the Stage One approximation and the noise in the Stage Two approximation, with variances of $\sigma_a^2$ and $\sigma^2$, respectively. Further, let Part 1 of Assumptions 4 hold.
Let $\{x_i\}_{i \geq 1}$ be the realisation of an i.i.d sequence marginally distributed as $\pi$. For $d \geq 1$, set $\bmx^{(d)} = (x_1, \ldots, x_d) \in \RR^d$ and let $\bmX^{(d),*}$ and $\bmZ^{(d)}$ be as defined in (4.1).
For almost all realisations $\{x_i\}_{i \geq 1}$, the following limit 
\begin{align}
\label{eq.limiting.Q.Wa.W}
\lim_{d \to \infty} \;
  \begin{bmatrix}
    q^{(d)}_{\Delta}(\bmx^{(d)}, \bmX^{(d),*})\\
    s^{(d)}_{\Delta}(\bmx^{(d)}, \bmX^{(d),*})\\
    W_a^{(d),*}-W_a^{(d)}\\
    W^{(d),*}-W^{(d)}
  \end{bmatrix}
\;=\;
  \begin{bmatrix}
    Q^{\infty}_{\Delta}\\
    S^{\infty}_{\Delta}\\
    W^{\infty}_{a,\Delta}\\
    W^{\infty}_{\Delta}\\
  \end{bmatrix}
\dist \Normal{
  \begin{bmatrix}
    -\frac{1}{2}\mu^2\\
    \frac{1}{2}\mu^2\beta_1\\
  -\sigma^2_a\\
  -\sigma^2
  \end{bmatrix}
  ,
  \begin{bmatrix}
  \mu^2&-\mu^2\beta_1&0&0\\
  -\mu^2\beta_1&\mu^2\beta_2^2&0&0\\
  0&0&2\sigma_a^2&0\\
  0&0&0&2\sigma^2
  \end{bmatrix}
}
\end{align}
holds in distribution.
\end{prop}

Analogously to Proposition 4.1 (but, for simplicity of presentation, taking Assumption 2 as well) we then obtain the limiting acceptance probabilities as follows.

\begin{prop} \label{prop.pmpm.limi.accept.proba}
Let Assumptions 1 and 2 hold both for the Stage One approximation and the Stage Two approximation, with variances of $\sigma^2_a$ and $\sigma^2$ respectively. Let Part 1 of Assumptions 4 hold.
Then
\begin{align*}
\lim_{d \to \infty} \; \accado{\bmX^{(d)},W_a^{(d)}} \; \stackrel{L^2}{=} \; \acc_1
\quad \textrm{and} \quad
\lim_{d \to \infty}  \acc_{12}^{(d)}(\bmX^{(d)}, W_a^{(d)},W^{(d)}) \; \stackrel{L^2}{=} \; \acc_{12}
\end{align*}
where the limiting acceptance rates are given by
\begin{align}
\label{eq.limiting.rate.supp}
\left\{
\begin{aligned}
\alpha_1 
&= \Expect{ F\left( Q^{\infty}_{\Delta}+S^{\infty}_{\Delta}  + W^{\infty}_{a,\Delta} \right) } \\
\alpha_{12} 
&= \Expect{  F\left( Q^{\infty}_{\Delta}+S^{\infty}_{\Delta} + W^{\infty}_{a,\Delta}\right) \times F\left( W^\infty_{\Delta}  - S^{\infty}_{\Delta} -W^{\infty}_{a,\Delta} \right) }
\end{aligned} \right.
\end{align}
for $(Q^{\infty}_{\Delta} ,S^\infty_\Delta,W_{a,\Delta}^{\infty},W^\infty_\Delta )$ as described in \eqref{eq.limiting.Q.Wa.W}.
The dependence of $\alpha_1$ and $\acc_{12}$ upon $(\mu,\sigma_a^2,\sigma^2,\beta_1,\beta_2)$ is implicit.
\end{prop}

\section{Envelope width}
\label{app.envelopeWidth}
Figures \ref{fig.da.optimal.lambda} (left) and \ref{fig.musigma2scatterpm.app} provide a look-up, given $\eta$ and $\alpha_{2|1}$ at the optimal parameter setting, of the ratio of the optimal tuning parameter (scaling for DARWM, scaling and variance for DAPMRWM) for the DA algorithm to the optimal value for the parent algorithm. However, in places the envelope of possible values is relatively wide. Here, we investigate this envelope.

Figure \ref{fig.envelopeplots} shows the plots in Figures \ref{fig.da.optimal.lambda} (left) and \ref{fig.musigma2scatterpm.app}, but specifically for $\eta=0.01$, and with points coloured according to $\beta_2$. The exception is the top-right plot, which repeats Figure \ref{fig.da.optimal.lambda} (left) but coloured according to theoretical efficiency.

\begin{figure}
\begin{center}
\subfigure{
  \includegraphics[width=0.4 \textwidth,angle=0]{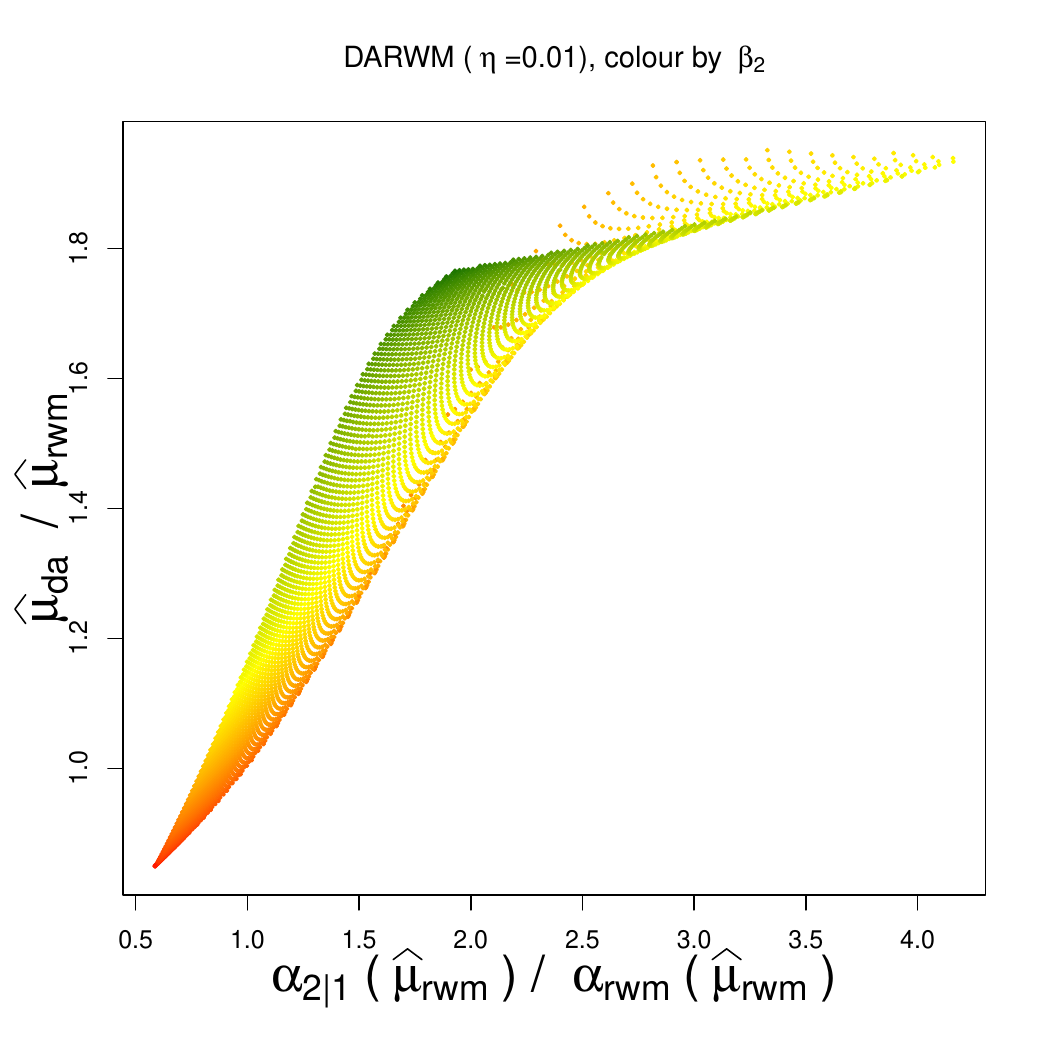}
}
\subfigure{
  \includegraphics[width=0.4 \textwidth,angle=0]{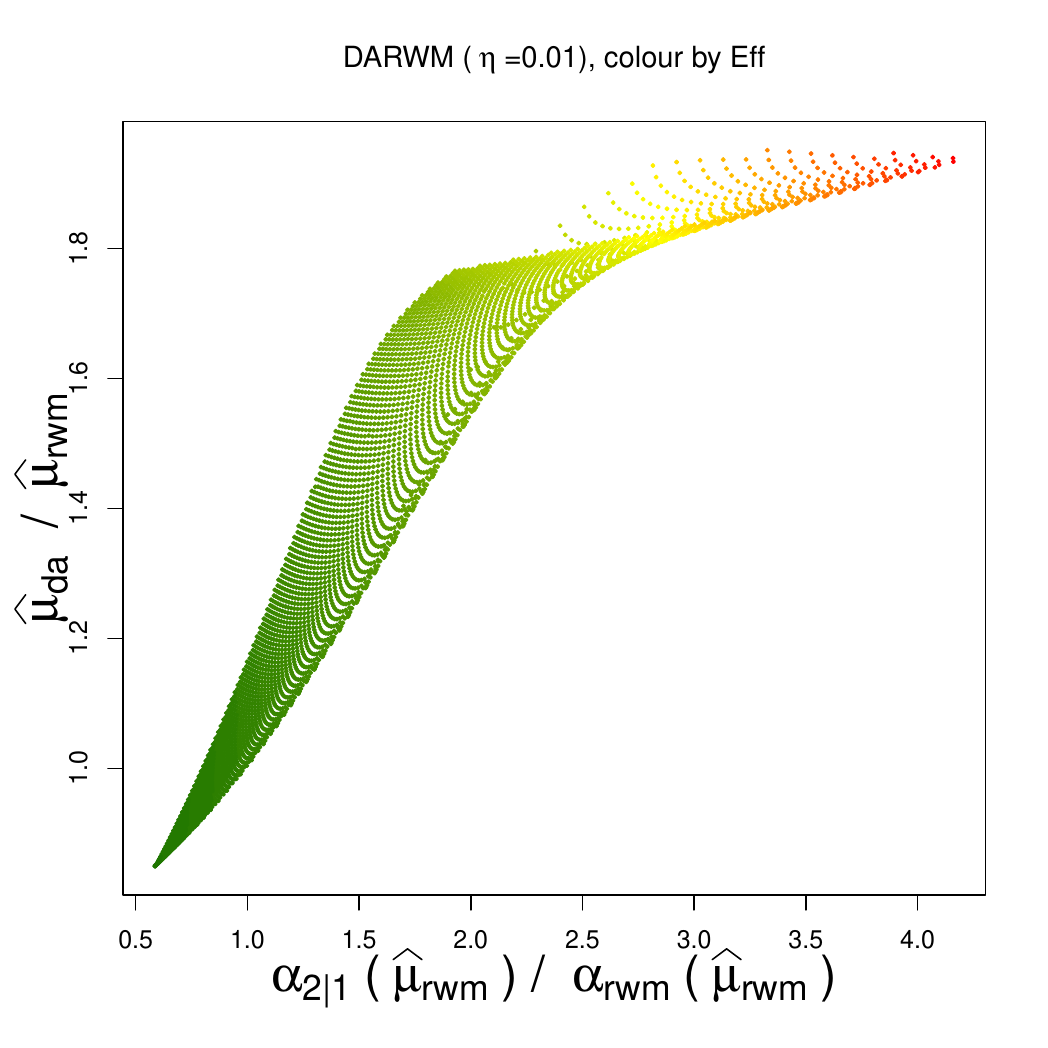}
}
\subfigure{
  \includegraphics[width=0.4 \textwidth,angle=0]{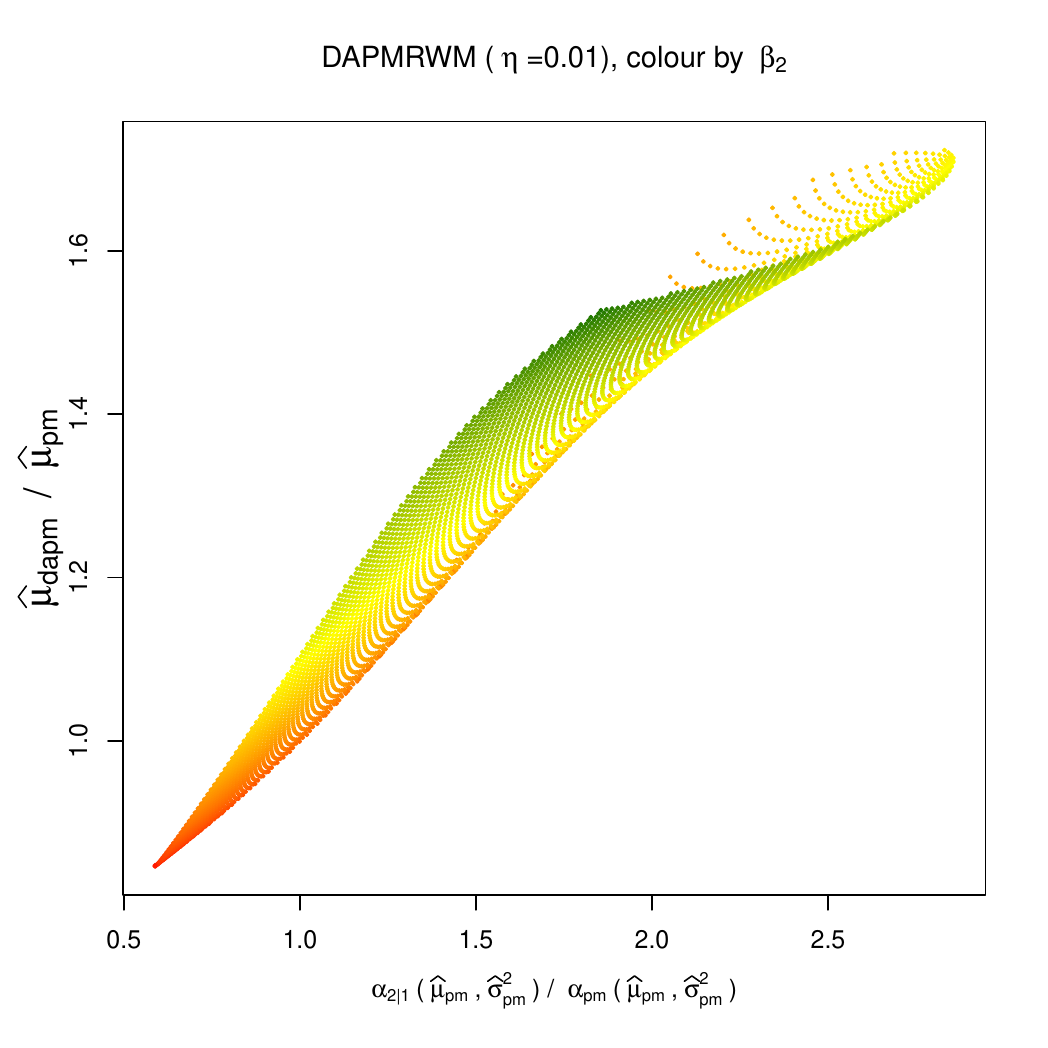}
}
\subfigure{
  \includegraphics[width=0.4 \textwidth,angle=0]{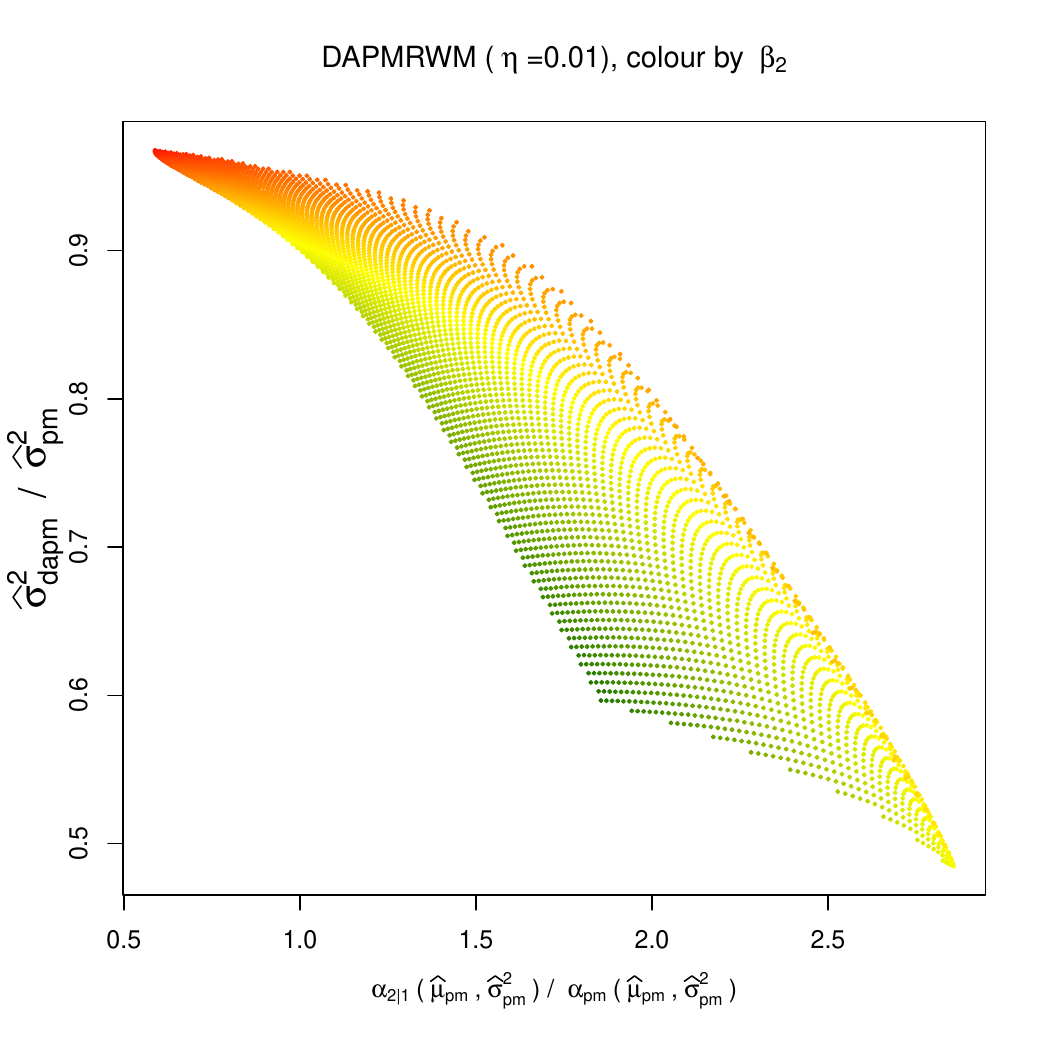}
}
\caption{
Scatter plots of (top) $\muhat_{\textrm{darwm}}/\muhat_{\textrm{rwm}}$, (bottom left) $\muhat_{\textrm{dapm}}/\muhat_{\textrm{pm}}$ and  (bottom right) $\sigmahat^2_{\textrm{dapm}}/\sigmahat^2_{\textrm{pm}}$ (right) all against $\alpha_{2|1}(\muhat_{\textrm{pm}},\sigmahat^2_{\textrm{pm}})$. Colour is according to $\beta_2$ except for the top right plot where colour is by theoretical efficiency; lowest values appear in dark green and highest values appear in red. All plots are for $\eta=0.01$.
\label{fig.envelopeplots}
}
\end{center}
\end{figure}

Firstly we examine the choice of scaling. The theoretical efficiency curve shows clearly that for the DARWM with $\eta=0.01$, moderate to large efficiencies are only obtained when $\alpha_{2|1}$ is large. This pattern is repeated, though not shown here, across different $\eta$ values and for PMRWM (where for $\eta=0.01$, the `yellow' region, corresponding to between 1/3 and 2/3 of the maximum value, is at an x axis value of $2.0$ rather than $3.0$). When $\alpha_{2|1}$ is large the envelope for $\widehat{\mu}$ is narrow. Thus for the more successful DA strategies there is little uncertainty in the recommended choice of scaling.

For large $\alpha_{2|1}$, the fine spray of points above the main line correspond to scenarios where $\beta_1\approx \beta_2$. Recall from Section \ref{sec.deter.approx} that $|\beta_1|\le \beta_2$, and notice that $\beta_1=\beta_2$ corresponds to the case where, in deriving the bound, the Cauchy-Schwarz inequality is exact and so the gradient of the random function, $S$ that is the error in the approximation, \emph{is proportional to} the gradient of $\log \pi$, with a positive coefficient of proportionality; the random function from which our approximation derives is, in fact, deterministic. This is counter to the generality of our set up of using a realisation from a \emph{random} function, and, moreover, corresponds to an approximation which matches the  mode(s) of the target and has steeper gradients everywhere else, which would not be sensible. We, therefore, recommend ignoring this fine spray of points (indeed we have already removed points with $|\beta_1|/\beta_2>0.9$) and using the more solid part of the envelope.

The top-left and bottom-left panels show that amongst the less efficienct DA strategies, the more efficient of these occur nearer the top of the envelope. With a chicken-and-egg assumption that one would not be using a DA approximation unless it was at least reasonably efficient, we would suggest choose a value towards the top of the envelope.

In contrast to the behaviour for the scaling parameter, for relatively large $\alpha_{2|1}$, the range of possible values for $\sigma^2$ is large, and all except the largest values correspond to the lower $\beta_2$ values. Interestingly, in the example in Section \ref{sect.sim} the curve for efficiency as a function of $\sigma^2$ has a  flat peak which covers roughly a doubling of the variance. Thus we conjecture that choosing any value within the envelope will lead to a close-to-optimal algorithm.

\section{DARWM: Gaussian target with logistic approximation}
\label{app.DARWMbetas}

We consider a scenario where the true target is a product of standard Gaussians
and the deterministic approximation is a product
of logistic densities with a mode at $\phi_1$ and inverse-scale parameter $\phi_2$,
\begin{align}
\pi(x) \propto \exp\curBK{-\frac{1}{2}\sum_{i=1}^dx_i^2}
\qquad \textrm{and} \qquad
\widehat{\pi}_a(x)\propto\prod_{i=1}^{d}\frac{e^{\phi_2 (x_i-\phi_1)}}{\left(1+e^{\phi_2 (x_i-\phi_1)}\right)^2}.
\label{eqn.simpsimstudytargets}
\end{align}
We consider fourteen scenarios: ten different combinations of values for $(\phi_1,\phi_2)$, three approximations where the values of $\phi_1$ or $\phi_2$ vary from component to component, and the `perfect approximation', $\pi_a=\pi$; see Table
\ref{table.simple.sim.stud} for further details.

\begin{table}
\begin{center}
\begin{tabular}{l|cccc|ll}
Algorithm&$\phi_1$&$\phi_2$&$\beta_1$&$\beta_2$&$\alpha_1$&$\alpha_{2|1}$\\
\hline
RWM&&&&&0.2616&\\
\hline
DA &0.0&0.6&0.834&0.834& 0.261&0.128 \\
DA &0.0&1.2&0.441&0.449&  0.069&0.533\\
DA &0.0&1.8&-0.042&0.262&0.041&0.738 \\
DA &0.0&2.3&-0.467&0.649& 0.034&0.595\\
DA &0.0&2.7&-0.810&1.025&0.032&0.492 \\
DA &0.5&1.2&0.466&0.552&0.370&0.547\\
DA &1.0&1.2&0.535&0.763&0.140&0.151  \\
DA &1.5&1.2&0.630&0.979&0.482&0.276\\
DA &0.6&1.8&0.056&0.681&0.0650&0.279 \\
DA &0.5&2.3&-0.351&0.941&0.049&0.289 \\
\hline
DA &0.0&1.5--2.0&&&0.248&0.772\\
DA &0.0&1.2--2.7&&&0.238&0.609\\
DA &0.0--1.0&1.2&&&0.377&0.517\\

\end{tabular}
\caption{Values of $\phi_1$ and $\phi_2$ used in \eqref{eqn.simpsimstudytargets}, and the corresponding values of $\beta_1$ and $\beta_2$ (where calculable), $\alpha_1(\lambdahat_{\textrm{rwm}})$ and $\alpha_{2|1}(\lambdahat_{\textrm{rwm}})$.
\label{table.simple.sim.stud}}
\end{center}
\end{table} 

Empirical effective sample sizes (ESSs) for each of the $d$ components are calculated using the \texttt{coda} package in \texttt{R} \cite{CODA}; the overall ESS is taken to be the average of the ESSs over the $d$ individual components. All algorithms were run for $10^6$ iterations.

We first obtained the optimal scaling, $\lambdahat_{\textrm{rwm}}$, for a RWM targeting $\pi$ by optimising the empirical ESS, and evaluated $\alpha_{\textrm{rwm}}(\lambdahat_{\textrm{rwm}})$ as well as the empirical ESS at this tuning. Then we ran the DA algorithm with this scaling to find $\alpha_{2|1}(\lambdahat_{\textrm{rwm}})$.
Next, we artificially induced three different values of
$\eta$: $0.1$, $0.01$, $0.001$ and evaluated the efficiency, (empirical ESS-100) /  CPU time) over a grid of possible scalings, $\lambda$, to find the optimal scaling. The regularisation penalty is needed because for very poorly mixing chains the empirical ESS tends to overestimate the true efficiency.

Figure
\ref{fig.mueffscatterwsim} reproduces Figure 1
for $\eta\in\{0.1,0.01,0.001\}$, but in three shades of grey, then plots $\lambdahat_{\textrm{da}}/\lambdahat_{\textrm{rwm}}$ ($d=10$ and $d=25$); a similar plot for the predicted and realised relative efficiency ($d=25$) against $\alpha_{2|1}(\lambdahat_{\textrm{da}})/\alpha_{\textrm{rwm}}(\lambdahat_{\textrm{rwm}})$ is also provided. At $d=10$ the theory sometimes slightly overestimates the increase in scaling that is required, although (not shown) the predicted range of gains in efficiency is accurate except when $\eta$ is small and $\alpha_{2|1}(\muhat_{\textrm{rwm}})/\alpha_{\textrm{rwm}}(\muhat_{\textrm{rwm}})$ is large, but by $d=25$ the theoretical prediction of the ratio is quite accurate, as is the predicted efficiency gain. Essentially, with a larger scaling and a smaller dimension the diffusion approximation is less accurate.

\begin{figure}
\begin{center}
\subfigure{
  \includegraphics[width = 0.3 \textwidth,angle=0]{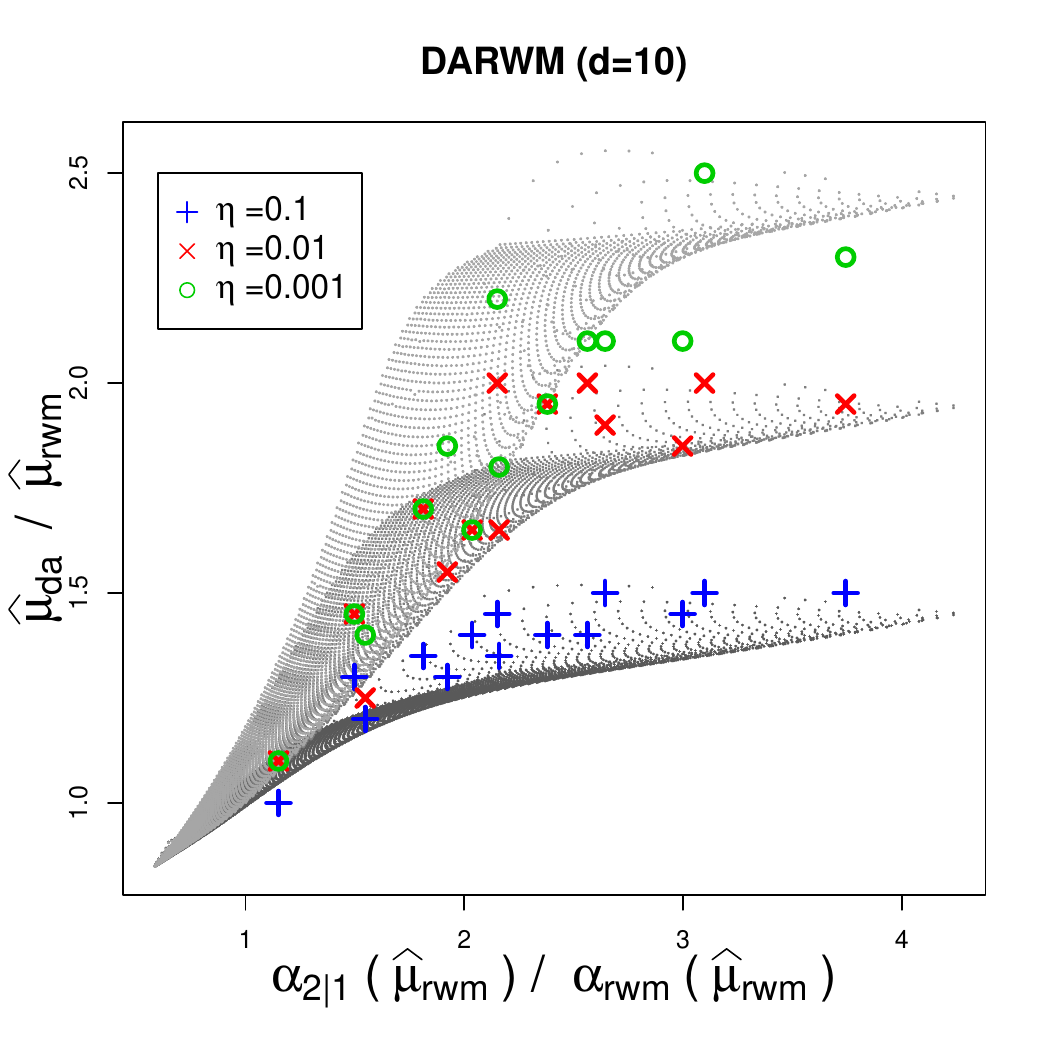}
}
\subfigure{
  \includegraphics[width = 0.3 \textwidth,angle=0]{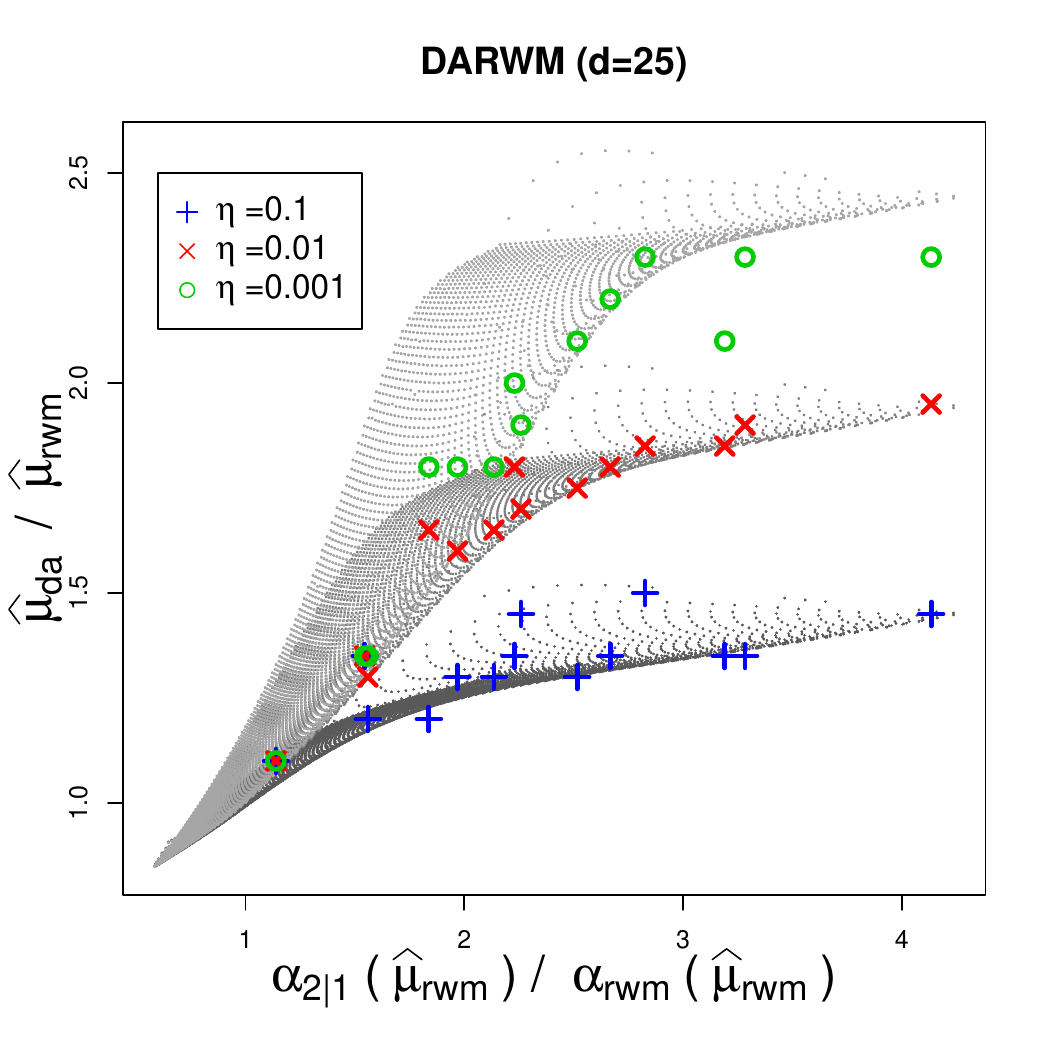}
}
\subfigure{
  \includegraphics[width = 0.3 \textwidth,angle=0]{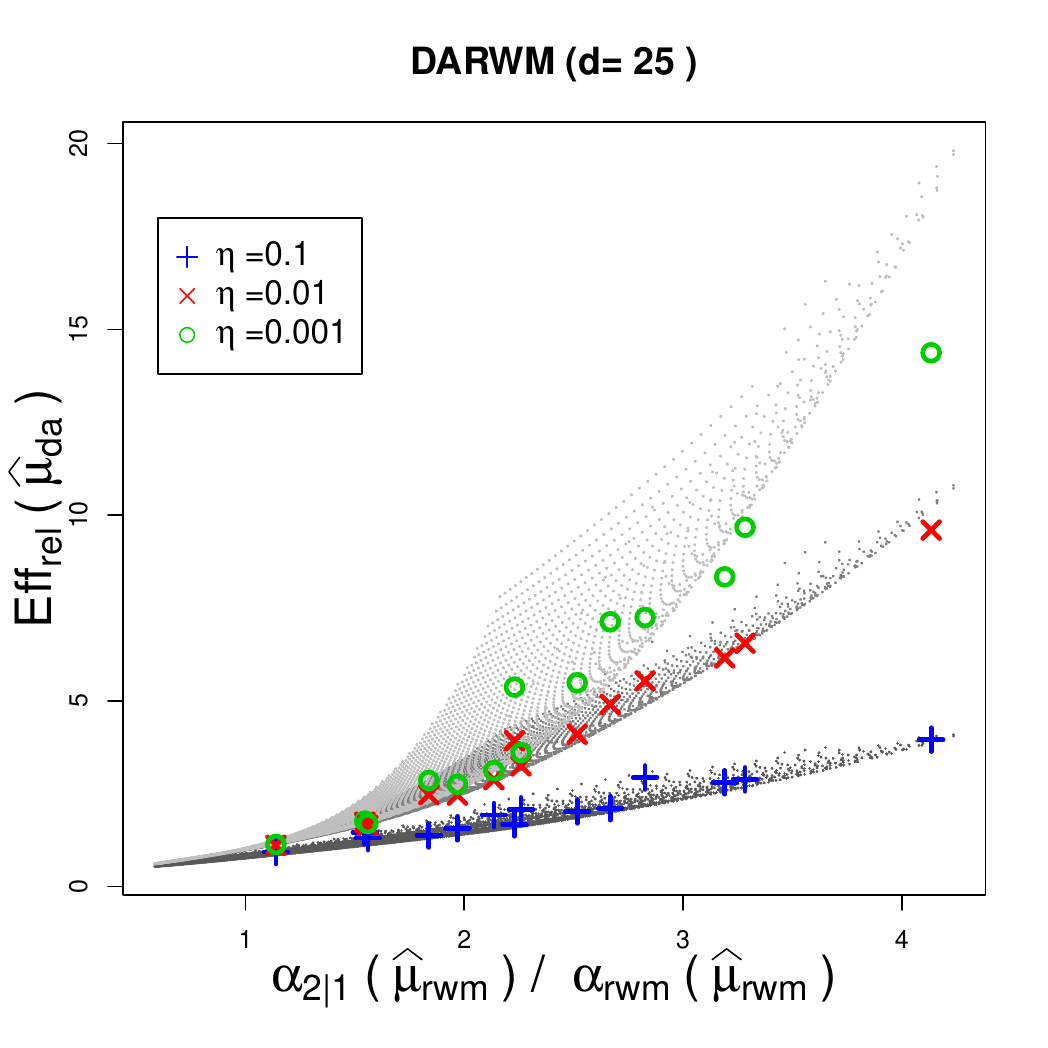}
}
\caption{
Scatter plots of $\muhat_{\textrm{da}}/\muhat_{\textrm{rwm}}$ ($d=10$, left and $d=25$, centre) and $\eff^{\textrm{rel}}_{\textrm{da}}(\muhat_{\textrm{da}})$ ($d=25$, right), vs $\alpha_{2|1}(\muhat_{\textrm{rwm}})$, partitioned by $\eta$.
\label{fig.mueffscatterwsim}
}
\end{center}
\end{figure}

\section{DAPsRWM: tuning advice and Lotka Volterra simulation study}
\label{sec.DAPsRWMtuneandsim}
\subsection{DAPsRWM tuning advice}

 Our theory predicts behaviour in terms of the theoretical scaling, $\mu$, and the variance $\sigma^2$, whereas the quantities the user wishes to tune are the actual scaling, $\lambda$, and the number of particles, $m$. Since $\lambda \propto \mu$ \eqref{eq.proposal.high.dim}, and $\sigma^2\propto 1/m$ (see the discussion following Assumptions \ref{ass.effort.vs.variance}), but with  unknown proportionality constants, we use our theory to predict $\muhat_{\textrm{da}}/\muhat_{\textrm{rwm}}=\lambdahat_{\textrm{da}}/\lambdahat_{\textrm{rwm}}$ and $\sigmahat^2_{\textrm{dapm}}/\sigmahat^2_{\textrm{pm}}=\widehat{m}_{\textrm{pm}}/\widehat{m}_{\textrm{dapm}}$.
Importantly, our tuning guidelines are based on quantities that can straightforwardly and robustly be estimated from a short MCMC trajectory.

As described in Section \ref{sec.dapmrwm} there are two possible tuning strategies for the DAPsMRWM. 
\begin{itemize}
\item Analogously to a strategy for the pseudo-marginal RWM (see
Section \ref{sect.lit.review}),  since the effect of altering the number of particles is approximately orthogonal to the effect of altering the scaling provided $\sigma^2\ge 1$, choose a number of particles that gives $\sigma^2 \ge 1$,  conditional on this tune the scaling to optimise efficiency and then, with this scaling, choose the number of particles to optimise efficiency.
\item Alternatively, given an approximately optimally tuned pseudo-marginal RWM algorithm, together with the parameters $\widehat{m}_{\textrm{pm}}$ and $\widehat{\lambda}_{\textrm{pm}}$, a single run of the DAPsMRWM with these parameters provides  $\alpha_{2|1}(\lambdahat_{\textrm{pm}},\sigmahat^2_{\textrm{pm}})$ and the value of $\eta > 0$. One may then obtain the ratios $\widehat{\lambda}_{\textrm{dapm}}/\widehat{\lambda}_{\textrm{pm}}$ and $\widehat{m}_{\textrm{pm}}/\widehat{m}_{\textrm{dapm}}$ from the tuning scatter plots in Figure \ref{fig.musigma2scatterpm.app}.
\end{itemize}

\subsection{Lotka-Volterra simulation study}
\label{sect.sim}
To illustrate the advice for the DAPsMRWM, and provide a check on its validity, we consider a Lotka-Volterra predator-prey model \cite{boys08}. The model describes the continuous time evolution of 
$\bmU_t=(U_{1,t},U_{2,t})$ where $U_{1,t}$ (prey) and $U_{2,t}$ (predator) are 
non-negative integer-values processes. Starting from an initial value, which 
is assumed known for simplicity, $\bmU_t$ evolves according to a Markov jump process 
(MJP) parameterised by rate constants $\bmc=(c_1,c_2,c_3)$; details of the state transitions are provided in Section \ref{sec.LNA.volterra}.
The process is easily simulated via the Gillespie algorithm \cite{Gillespie77} and the 
pseudo-marginal RWM scheme is straightforward to apply \cite{GolightlyWilkinson:2011}.
We assume that the MJP is observed with Gaussian error every time unit for $n$ time units, $t=1,\ldots ,n$: $Y_{1,t}\sim \Normal{u_{1,t},s_1^2}$ and 
$Y_{2,t}\sim \Normal{u_{2,t},s_2^2}$, independently.
As all of the parameters of interest must be strictly positive, we consider inference for 
\begin{equation*}
\bmx=\left(\log(c_{1}),\log(c_{2}),\log(c_{3}),\log(s_{1}),\log(s_{2})\right).
\end{equation*}
Parameter values and prior distributions are provided in Section \ref{sec.LNA.volterra}. The DAPsMRWM scheme requires a computationally cheap approximation of the MJP. We follow 
\cite{GolightlyHendersonSherlock:2013} by constructing a linear noise approximation  
(see e.g. \cite{kampen2001}),  detailed in Section \ref{sec.LNA.volterra}.

For a pseudo-marginal RWM scheme \cite{SherlockThieryRobertsRosenthal:2013} 
suggests that for a Gaussian target (where, for each principal
component, $I$ is known) proposals with a variance of $\bmV_{\textrm{Gauss}}=\left({2.56^2}/{d}\right)\times\textrm{Var}(\bmX)$ 
would be optimal. We propose Gaussian jumps with a variance
of
$\bmV_{\textrm{prop}}=\gamma^2\widehat{\bmV}_{\textrm{Gauss}}$, where 
$\textrm{Var}(\bmX)$, has been replaced with an approximation, $\widehat{\textrm{Var}}(\bmX)$, created from an initial run. If the target were in fact a high-dimensional Gaussian, and the variance approximation were exact, then $\gamma$ would correspond exactly to the theoretical scaling, $\mu/\muhat_{\textrm{pm}}$, and $\gammahat_{\textrm{pm}}$ would be $1.0$. 
We found that the pseudo-marginal RWM was optimised at $\gammahat_{\textrm{pm}}\approx 1.2$. \cite{SherlockThieryRobertsRosenthal:2013} suggests that the optimal number of particles should lead to a variance in $\log \pih$ of approximately $3.3$. We found an optimum 
of $m=180$, which occurred when the $\textrm{Var}[\log \pih(x_*)]$ 
(with $x_*$ an initial estimate of the componentwise posterior median) 
was approximately $2.9$. The mean acceptance probability at this 
optimal tuning was $\alpha_{\textrm{pm}}\approx 8.0\%$ 
and the empirical efficiency (minimum, over each parameter component, effective sample size per second) was $0.067$.

The DAPsMRWM with $\gamma=1.2$ and $m=180$ gave 
$\alpha_{2|1}\approx 20.7\%$, so that 
$\alpha_{2|1}/\alpha_{\textrm{pm}}\approx 2.6$; timing diagnostics gave 
$\eta=0.0014$. For this combination, the tuning scatter plots suggest 
increasing the scaling by a factor of around $2.0$, decreasing 
the variance by a factor of between $0.7$ and $0.8$, and that this 
should lead to an increase in efficiency of a factor of between $6$ 
and $7$. The tuning suggestions translate to 
$\gamma\approx 2.4$ and $m\approx 225-255$. 
Alternatively, Figure \ref{fig.eff.mu.sigma2} suggests 
that provided $\sigma^2>1$, $m$ and $\gamma$ may be tuned independently.

\begin{figure}
\begin{center}
\includegraphics[angle=270, width=1.0\textwidth]{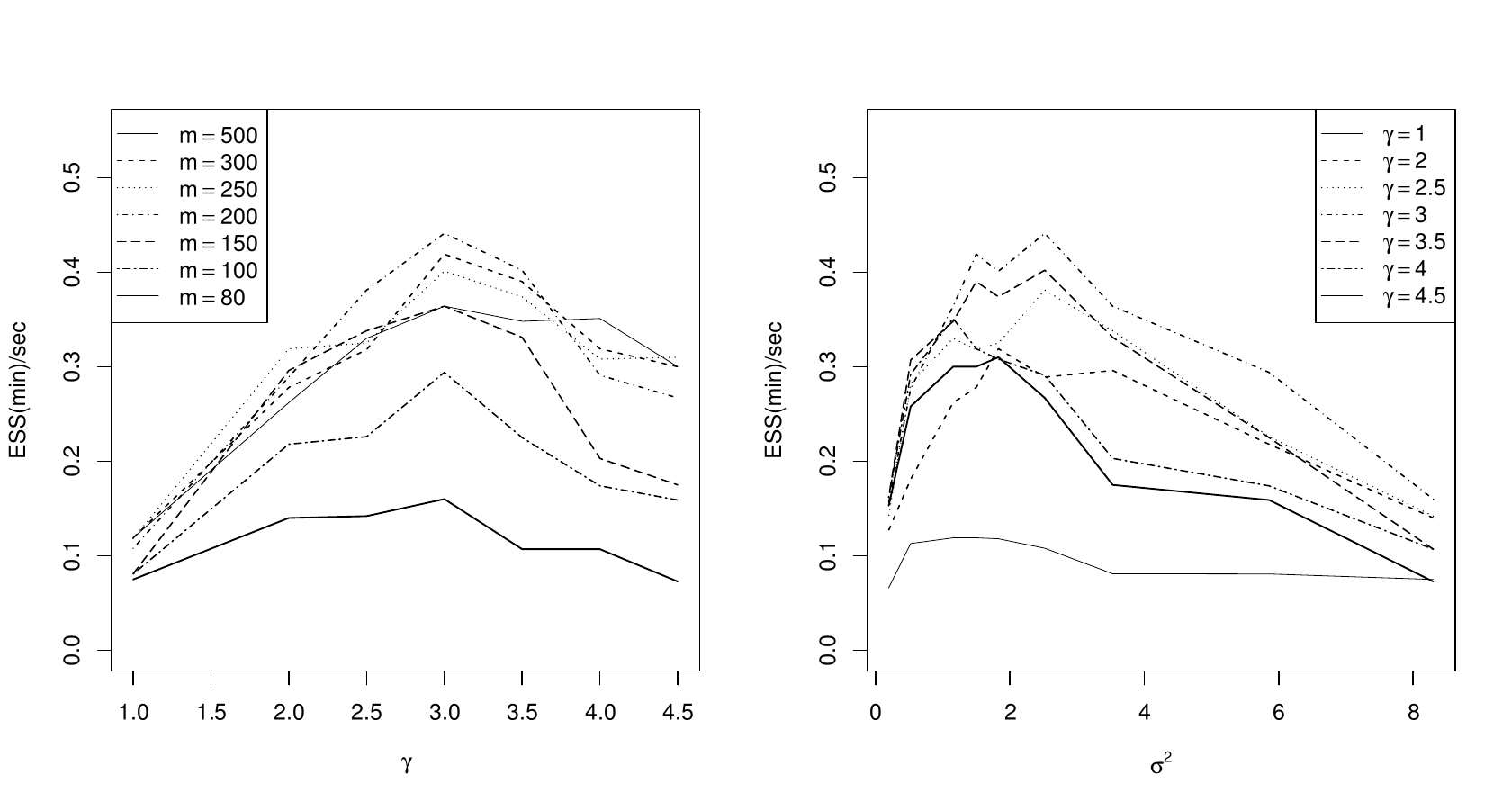}
\caption{Empirical efficiency measured as the effective sample size per CPU second. 
The left-hand panel gives the efficiency plotted against $\gamma$ for various numbers 
of particles. The right-hand panel gives the efficiency plotted against $\sigma$ (estimated 
at the posterior median), for various scalings.
\label{fig.ess}
}
\end{center}
\end{figure}

To confirm that the practical advice is reasonable and to test some of
the other predictions of our theory, the number of particles $m$ was varied between 
$80$ and $2000$ and, for each $m$, the scaling $\gamma$ was varied between $1$ and $4.5$. For each 
$(m, \gamma)$ pair, a long MCMC run (of at least $4\times 10^{5}$ iterations) was performed. Figure~\ref{fig.ess} 
shows empirical efficiency as a function of the scaling $\gamma$ (with a varying number of particles $m$) 
and as a function of $m$  (for various
$\gamma$)
and provides empirical evidence of 
the insensitivity of the optimal choice of scaling, $\gamma$, 
to the value of $\sigma^2$, for values of  $\sigma^2 \geq 0.89$;
furthermore, for
variances below $0.89$ the optimal scaling increases, as predicted by
our theory. Empirical efficiencies for each run, as well as empirical Stage-One and Conditional Stage-Two acceptance rates  are given in Table \ref{table.results}, and back up the heuristic from the figures, that 
$\gamma=2.5$ gives close to the optimal efficiency,
with $\gammahat\approx 3.1$, and $\widehat{m}\approx 220-250$
as predicted. The empirical efficiency gain from using the DAPsMRWM algorithm
compared to the pseudo-marginal RWM algorithm was $0.441/0.067\approx 6.6$, which is
in the centre of the range predicted by the theory.

\begin{table}
\begin{center}
\begin{tabular}{r|r|r|r|r|r|r|r|r|r|r}
& $m$    & $80$ &$100$ &$150$ &$200$ &$250$ &$300$ &$500$ &$800$&2000\\
$\gamma$& $\sigma^2$& $8.30$&$5.86$&$3.53$&$2.52$&$1.83$&$1.50$&$0.89$&$0.52$&$0.20$\\
\hline
\hline
1&mESS/s             &0.0750  &0.0808 &0.0810 &0.108  &0.118  &0.119  &0.119   &0.113 & 0.0661\\
&$\widehat{\alpha}_1$    &0.256   &0.255  &0.257  &0.255  &0.257  &0.254  &0.254   &0.255 & 0.258\\
&$\widehat{\alpha}_{2|1}$&0.0651  &0.0883 &0.170  &0.237  &0.289  &0.341  &0.447   &0.547 & 0.692\\
\hline
2&mESS/s             &0.140  &0.218  &0.296  &0.289  &0.319  &0.278  &0.262   &0.181 & 0.127\\
&$\widehat{\alpha}_1$    &0.0556 &0.0514 &0.0489 &0.0503 &0.0517 &0.0520 &0.0513  &0.0517& 0.0505\\
&$\widehat{\alpha}_{2|1}$&0.0619 &0.0895 &0.163  &0.213  &0.286  &0.313  &0.438   &0.522 & 0.674\\
\hline
2.5&mESS/s           &0.142  &0.226  &0.338  &0.381  &0.325  &0.318  &0.330   &0.282 & 0.142\\
&$\widehat{\alpha}_1$    &0.0244 &0.0237 &0.0234 &0.0259 &0.0264 &0.0234 &0.0241  &0.0230& 0.0250\\
&$\widehat{\alpha}_{2|1}$&0.0600 &0.0815 &0.159  &0.218  &0.252  &0.312  &0.434   &0.523 & 0.675 \\
\hline
3&mESS/s             &0.160  &0.294  &0.364  &0.441  &0.401  &0.419  &0.364  &0.277 &0.156 \\
&$\widehat{\alpha}_1$    &0.0143 &0.0123 &0.0119 &0.0114 &0.0131 &0.0120 &0.0114 &0.0124&0.0121 \\
&$\widehat{\alpha}_{2|1}$&0.0416 &0.101  &0.152  &0.233  &0.274  &0.320  &0.426  &0.516 &0.673 \\
\hline
3.5&mESS/s           &0.107  &0.225  &0.331   &0.402   &0.374   &0.390   &0.348   &0.307  &0.162  \\
&$\widehat{\alpha}_1$    &0.00629&0.00789&0.00763 &0.00684 &0.00669 &0.00663 &0.00725 &0.00634&0.00694  \\
&$\widehat{\alpha}_{2|1}$&0.0550 &0.0869 &0.170   &0.237   &0.273   &0.312   &0.424   &0.534  &0.673 \\
\hline
4&mESS/s             &0.107  &0.174  &0.176   &0.291   &0.308   &0.319   &0.351   &0.292  &0.162 \\
&$\widehat{\alpha}_1$    &0.00343&0.00318&0.00401 &0.00388 &0.00372 &0.00357 &0.00377 &0.00402&0.00418 \\
&$\widehat{\alpha}_{2|1}$&0.0680 &0.105  &0.151   &0.215   &0.287   &0.310   &0.407   &0.500  &0.681 \\
\hline
4.5&mESS/s           &0.0728 &0.159  &0.150   &0.267   &0.310   &0.300   &0.300   &0.258  &0.153  \\
&$\widehat{\alpha}_1$    &0.00220&0.00183&0.00207 &0.00247 &0.00230 &0.00256 &0.00224 &0.00249&0.00226 \\
&$\widehat{\alpha}_{2|1}$&0.0527 &0.111  &0.143   &0.213   &0.265   &0.280   &0.424   &0.491  &0.658 
\end{tabular}
\caption{Minimum effective sample size (mESS) per second, stage 1 acceptance probability $\widehat{\alpha}_{1}$ 
and stage 2 acceptance probability $\widehat{\alpha}_{2|1}$ as functions of the number of particles $m$ and scaling $\gamma$. The variance 
($\sigma^{2}$) of the estimated log-posterior at the median is also shown for each choice of $m$.
\label{table.results}}
\end{center}
\end{table}

Proposition
\ref{prop.acc.rates.dec} proves that, subject to assumptions, the Stage 
2 acceptance probability decreases as the variance in the
log-posterior ($\sigma^2$) increases and the Stage 1 acceptance
probability decreases as the scaling increases; Table \ref{table.results} shows that these
patterns are observed in our experiments. 

As with standard MCMC, the samples from tuning runs could be combined with the sample from the run at the optimal parameter values to decrease the variance of any estimator still further.

\subsection{Lotka-Volterra details}
\label{sec.LNA.volterra}
The Lotka-Volterra MJP is characterised by transitions over $(t,t+dt]$ of the form
\begin{eqnarray*}
\Prob{U_{1,t+dt}=u_{1,t}+1,U_{2,t+dt}=u_{2,t}|u_{1,t},u_{2,t}} &=& c_{1}u_{1,t}dt+o(dt),\\
\Prob{U_{1,t+dt}=u_{1,t}-1,U_{2,t+dt}=u_{2,t}+1|u_{1,t},u_{2,t}} &=& c_{2}u_{1,t}u_{2,t}dt+o(dt),\\
\Prob{U_{1,t+dt}=u_{1,t},U_{2,t+dt}=u_{2,t}-1|u_{1,t},u_{2,t}} &=& c_{3}u_{2,t}dt+o(dt).
\end{eqnarray*}
Data were simulated using an initial value $\bmu_0=(71,79)$ for $n=50$ time units 
with $\bmc=(1.0,0.005,0.6)$ and $s_1=s_2=8$. The parameters $\bmx=(\log c_1, \log c_2, \log c_3, \log s_1, \log s_2)$ were assumed 
to be independent \emph{a priori} with proper Uniform densities on the interval 
$[-8,8]$ ascribed to $X_{i}$, ($i=1,\ldots,5$).

Under the linear noise approximation (LNA) we have that $\bmU_{t} \dist \Normal{\bmz_t+\bmm_t\,,\,\bmV_t}$ where $\bmz_t$, $\bmm_t$ and $\bmV_t$ satisfy a coupled ODE system
\begin{equation}\label{LNAode}
\left\{\begin{array}{lll}
\dot{\bmz}_t&=&\bmS\,\bmh(\bmz_t,\bmc) \\
\dot{\bmm}_t&=&\bmF_t\bmm_t\\
\dot{\bmV}_t&=&\bmV_t\bmF_{t}^{T}+\bmS\textrm{diag}\left\{\bmh(\bmz_t,\bmc)\right\}\bmS^T+\bmF_t\bmV_t
\end{array}\right.
\end{equation}
For the Lotka-Volterra model, the rate vector $\bmh(\bmz_t,\bmc)$, stoichiometry matrix $\bmS$ and Jacobian matrix $\bmF_t$ 
are given by $\bmh(\bmz_t,\bmc)=(c_{1}z_{1,t},c_{2}z_{1,t}z_{2,t},c_{3}z_{2,t})$ as well as
\begin{equation*}
\bmS=\left(\begin{array}{ccc}1&-1&0\\0&1&-1\end{array}\right),\quad \textrm{and} \quad
\bmF_t=\left(\begin{array}{cc}c_{1}-c_{2}z_{2,t}&-c_{2}z_{1,t}\\c_{2}z_{2,t}&c_{2}z_{1,t}-c_{3}\end{array}\right).
\end{equation*}
We now describe an algorithm for evaluating the posterior (up to proportionality) 
under the LNA. For further details regarding the LNA and its use as an approximation to a MJP, we 
refer the reader to \cite{fearnhead12} and \cite{GolightlyHendersonSherlock:2013}. 
For simplicity of exposition we assume an observation regime of the form
$\bmY_{t}=\bmU_{t}+\bmepsilon_{t}$ with $\bmepsilon_{t}\sim \Normal{0,\bmSigma}$ where $\bmepsilon_{t}$ is a length-$d_x$ Gaussian random vector. Suppose that 
$\bmU_{1}$ is fixed at some value $\bmu_{1}$. The marginal likelihood (and hence the 
posterior up to proportionality) under the LNA, $\pi_a(\bmy_{1:n}|\bmx)$ can be obtained as follows.
\begin{enumerate}
\item Initialisation. Compute 
$
\pi_a(\bmy_{1}|\bmx)=\phi\left(\bmy_{1}\,;\, \bmu_{1}\,,\,\bmSigma\right)
$
where $\phi\left(\bmy_{1}\,;\, \bmu_{1}\,,\,\bmSigma\right)$ denotes the Gaussian density 
with mean vector $\bmu_{1}$ and variance matrix $\bmSigma$. 
Set $\bma_{1}=\bmu_{1}$ and $\bmC$ to be the $d_{x}\times d_{x}$ matrix of zeros.
 
\item For times $t=1,2,\ldots ,n-1$,
\begin{itemize}
\item[(a)] Prior at $t+1$. Initialise the LNA with $\bmz_{t}=\bma_{t}$, $\bmm_{t}=0$ and $\bmV_{t}=C_{t}$. 
Note that $\bmm_{s}=\bmzero$ for all $s>t$. Integrate the ODE system (\ref{LNAode}) 
forward to $t+1$ to obtain $\bmz_{t+1}$ and $\bmV_{t+1}$. Hence
$\bmX_{t+1}|\bmy_{1:t}\sim \Normal{\bmz_{t+1},\bmV_{t+1}}\,.$
\item[(b)] One-step forecast. Using the observation equation, we have that 
$
\bmY_{t+1}|\bmy_{1:t}\sim \Normal{\bmz_{t+1},\bmV_{t+1}+\bmSigma}
$. Compute
$
\pi_{a}(\bmy_{1:t+1}|\bmx)=\pi_{a}(\bmy_{1:t}|\bmx)\,\phi\left(\bmy_{t+1}\,;\, \bmz_{t+1}\,,\,\bmV_{t+1}+\bmSigma\right)
$.
\item[(c)] Posterior at $t+1$. Combining the distributions in (a) and (b) gives 
$\bmU_{t+1}|\bmy_{1:t+1}\sim \Normal{\bma_{t+1},\bmC_{t+1}}$ where
$\bma_{t+1} = \bmz_{t+1}+\bmV_{t+1} \left( \bmV_{t+1}+\bmSigma\right)^{-1}\left(\bmy_{t+1}-\bmz_{t+1}\right) $
and
$
\bmC_{t+1} = \bmV_{t+1}-\bmV_{t+1}  (\bmV_{t+1}+\bmSigma)^{-1}\bmV_{t+1}
$.
\end{itemize}
\end{enumerate}



\end{document}